\documentclass[11pt]{article}
\usepackage{amsfonts}
\usepackage[margin=1.1in]{geometry}
\usepackage{amsmath,amssymb,amsthm,mathtools}
\usepackage{mathrsfs}
\usepackage{enumitem}
\usepackage{tocloft}
\usepackage{url}
\usepackage{microtype}
\usepackage[nottoc]{tocbibind}
\usepackage{hyperref}
\usepackage{cleveref}
\hypersetup{
    colorlinks=true,
    linkcolor=blue,
    citecolor=blue,
    urlcolor=blue}
\numberwithin{equation}{section}

\newtheorem{definition}{Definition}[section]
\newtheorem{theorem}[definition]{Theorem}
\newtheorem{proposition}[definition]{Proposition}
\newtheorem{lemma}[definition]{Lemma}
\newtheorem{corollary}[definition]{Corollary}
\newtheorem{maintheorem}[definition]{Main Theorem}
\newtheorem{question}[definition]{Question}
\newtheorem{remark}[definition]{Remark}

\newcommand{\C}{\mathbb{C}}
\newcommand{\Z}{\mathbb{Z}}
\newcommand{\Zge}{\mathbb{Z}_{\geq 0}}
\newcommand{\Hh}{\mathbb{H}}
\newcommand{\Hom}{\operatorname{Hom}}
\newcommand{\End}{\operatorname{End}}

\newcommand{\wt}{\operatorname{wt}}
\newcommand{\Span}{\operatorname{span}}
\newcommand{\Res}{\operatorname{Res}}

\newcommand{\Ext}{\operatorname{Ext}}
\newcommand{\im}{\operatorname{im}}
\def\Id{{\rm Id}}
\def\Rad{{\rm Rad}}
\def\span{{\rm Span}}
\def\tr{{\rm tr}}
\def\Log{{\rm Log}}
\def\Ann{{\rm Ann}}
\def\im{{\rm Im}}
\newcommand{\de}{\operatorname{d}}

\begin{document}

\begin{center}
{\Large {\bf Equivalence between rationality and modular invariance under $C_2$-cofiniteness\footnote{2010 Mathematics Subject Classification 17B69}}} \\

\vspace{0.5cm}
Hao Wang\footnote{Email address: whaomath@nwu.edu.cn. Supported by an NSFC grant 12001426}\\
 School of Mathematics, Northwest University, Xi'an 710127, Shaanxi, China\\
 Jiang Yang\footnote{Email address: yangjiangdy@126.com}\\
 School of Mathematical Sciences, Guangxi Minzu University, Nanning, 530006, Guangxi, China
\end{center}
\hspace{1.5 cm}

\begin{abstract}
  We prove that for a simple and $C_2$-cofinite vertex operator algebra, the rationality is equivalent to the modular invariance of the space of complete ordinary one-point trace functions of its irreducible modules. This proves a conjecture given by Dong-Ren in \cite{DongRen}, which is also the inverse theorem of Zhu's modular invariance theorem \cite{Zhu}.
\end{abstract}

\tableofcontents

\section{Introduction}
Modular invariance is one of the fundamental links between the representation theory of vertex operator algebras and two-dimensional conformal field theory.  In rational conformal field theory, modular transformations are closely related to fusion rules, braiding, and the consistency relations of Moore and Seiberg, for example, see Verlinde~\cite{V} and Moore-Seiberg~\cite{MS}.  In the framework of vertex operator algebras, this connection is established mathematically by Zhu~\cite{Zhu}.  Roughly speaking, Zhu's modular invariance theorem asserts that, under rationality and $C_2$-cofiniteness, the genus-one trace functions of irreducible modules span a finite-dimensional representation of the modular group.  The theory is subsequently developed in several directions, including orbifold theory and generalized Moonshine by Dong-Li-Mason~\cite{DLM-Modular}, the Verlinde conjecture and modular tensor category structure by Huang \cite{Huang-Modular,Hu1}, and the congruence properties of the resulting modular representations by Dong-Lin-Ng \cite{DLN}, and super vertex operator algebras \cite{DongZhao}, etc.

The interplay between representation-theoretic finiteness and modularity is already visible in the basic examples of vertex operator algebras.  Lattice vertex operator algebras and their twisted modules are studied systematically by Dong~\cite{D,D1}, while the rationality of the orbifold algebras \(V_L^+\) is developed through a series of works culminating, in higher rank, in Dong-Jiang-Lin~\cite{DJL}.  Genus-one trace functions for lattice and free-boson theories are studied by Dong-Mason-Nagatomo~\cite{DMN}, and higher-weight one-point functions arising from Moonshine are investigated by Dong-Mason~\cite{DM2}.  Trace functions for parafermion vertex operator algebras provide another important family of examples~\cite{DKR}.  More recently, the relation between modular data and fusion has continued to play a central role in orbifold theory: Dong and Lin establish a twisted Verlinde formula relating fusion rules of twisted modules to the orbifold \(S\)-matrix~\cite{DL}.  These developments illustrate that modular transformation properties are not merely analytic consequences of rationality, but encode substantial information about the representation category.

The finiteness conditions of rationality, regularity and \(C_2\)-cofiniteness have therefore played a central role in the structure theory of vertex operator algebras.  Their relations are studied by Dong-Li-Mason~\cite{DLM-Assoc,DLM}, Abe-Buhl-Dong~\cite{ABD}, Dong-Jiang~\cite{DJ,DJ2}, and the others \cite{DM,DZ}.  In particular, in the standard positive-energy setting, rationality together with \(C_2\)-cofiniteness leads to strong finiteness and semisimplicity properties. However, the \(C_2\)-cofinite condition by itself does not force rationality.  The triplet vertex operator algebras \(\mathcal W(p)\) give fundamental \(C_2\)-cofinite but nonsemisimple examples~\cite{AM,TW}. Even under further assumptions such as self-contragredient, CFT-type and modular invariance of irreducible characters, $C_2$-cofiniteness still does not imply rationality, see \cite{WY}. Thus, within the class of \(C_2\)-cofinite vertex operator algebras, one must distinguish genuinely semisimple theories from logarithmic ones.

This distinction is reflected precisely in genus-one modularity. Although modular invariance of irreducible characters does not imply rationality, we can assume the modular invariance of the complete one point trace functions of irreducible modules. Since characters are the trace functions evaluated at the vacuum vector, this procedure loses a lot of information. For a nonsemisimple \(C_2\)-cofinite vertex operator algebra, the ordinary traces of irreducible modules need not form a modularly closed space.  Miyamoto's modular invariance theorem~\cite{Miyamoto} shows that one must in general enlarge the ordinary trace space by pseudo-trace functions.  Pseudo-traces and their role in logarithmic conformal field theory have been studied further, for example, see Arakawa-Kawasetsu~\cite{AK}, Arike-Nagatomo~\cite{AN} and Gainutdinov-Runkel~\cite{GR}, while the analytic theory of logarithmic genus-one correlation functions is developed by Fiordalisi~\cite{Fiordalisi}.  From the categorical point of view, nonsemisimple modular phenomena are naturally related to braided finite tensor categories and their generalizations~\cite{Ly,S}.  Thus the following question is natural:
\begin{question}
If a simple \(C_2\)-cofinite vertex operator algebra has the property that its ordinary irreducible one-point trace functions are already closed under modular transformations, without the appearance of pseudo-traces, must the vertex operator algebra be rational?
\end{question}
Dong-Ren formulate this question explicitly in~\cite{DongRen}.  The purpose of the present paper is to answer it affirmatively.

Suppose $V$ is a simple and \(C_2\)-cofinite lower truncated vertex operator algebra, and $W^0=V,W^1,\ldots,W^r$ are representatives of the inequivalent irreducible ordinary \(V\)-modules.  For an ordinary module \(W\), we use the geometrically modified complete one-point function
\[
F_W(v,\tau)
=
\tr_W
o_W(\mathcal U(1)v)
q^{L_W(0)-c/24},
\qquad
q=e^{2\pi i\tau},
\]
where \(\mathcal U(1)\) is the geometric coordinate-change operator.  The modular generator \(S\) acts by
\[
(SF_W)(v,\tau)
=
F_W\left(
(-\tau^{-1})^{L(0)}v,-\tau^{-1}
\right).
\]
It is important that \(v\) is allowed to vary in all of \(V\).  Thus the hypothesis below concerns the complete one-point functions, and is substantially stronger than an identity involving only the characters
\[
F_W(\mathbf1,\tau)
=
\tr_Wq^{L_W(0)-c/24}.
\]

Our main result may be stated in the following form.

\begin{maintheorem}\label{thm1.2}
Let \(V\) be a simple \(C_2\)-cofinite lower-truncated vertex operator algebra.  Suppose there exist constants \(s_j\in\mathbb C\) which are independent on $v\in V,\tau\in \mathbb H$, such that
\[
SF_V
=
\sum_{j=0}^{r}s_jF_{W^j}
\]
as an identity of complete one-point functions, i.e., as an identity depending simultaneously on \(v\in V\) and \(\tau\in\mathbb H\).  Then \(V\) is rational.
\end{maintheorem}
In fact, the proof passes through the stronger conclusion that the relevant category of grading-restricted generalized modules is semisimple and that \(V\) is regular.  Combining the theorem with Zhu's modular invariance theorem and its trace-function formulation developed by Dong-Li-Mason~\cite{DLM-Modular,Zhu}, we obtain the following equivalence.
\begin{corollary}\label{co1.3}
Let \(V\) be a simple \(C_2\)-cofinite vertex operator algebra.  Then \(V\) is rational if and only if the space
\[
\mathcal H
=
\operatorname{span}_{\mathbb C}
\{F_{W^0},F_{W^1},\ldots,F_{W^r}\}
\]
of complete ordinary one-point functions of its irreducible modules is invariant under the modular group.
\end{corollary}

This proves a conjecture of Dong-Ren~\cite{DongRen} and gives a converse theorem to Zhu's modular invariance theorem within the \(C_2\)-cofinite setting.  Conceptually, the result says that the genuinely nonsemisimple contributions to genus-one modularity cannot disappear accidentally, if the ordinary irreducible trace functions already form the required modularly invariant space, then the representation theory is forced to be semisimple.

A closely related recent result is due to McRae~\cite{McRae}.  For a simple, self-contragredient, nonnegatively graded \(C_2\)-cofinite vertex operator algebra, McRae proves that if the \(S\)-transformation of the vacuum character is a linear combination of ordinary module characters, then the category of grading-restricted generalized modules is rigid. He further shows that a rigid category in this setting is a factorizable finite ribbon category, and obtains rationality under the additional hypothesis that the Zhu algebra is semisimple. The distinction between rigidity and semisimplicity is essential. A finite braided tensor category may be rigid without being semisimple. The present work addresses precisely this gap for the complete one-point-function condition, and at the same time removes the assumptions that the original vertex operator algebra be self-contragredient and nonnegatively graded.

There are two structural difficulties behind these extensions.  First, for the original conformal structure the contragredient module $K=DV=V'$
need not be isomorphic to the tensor unit \(V\).  Thus the contragredient duality initially gives a Grothendieck--Verdier duality rather than rigid duality.  Second, the original conformal grading may have a finite negative part.  This issue is not artificial. Shifted vertex operator algebras already provide natural examples in which the underlying vertex algebra is rational while the conformal grading has been shifted~\cite{DongMason}.  Consequently, one cannot simply replace the original conformal structure by a positive-energy one at the outset.  Our argument keeps the original module category and the original genus-one trace functions throughout.

We now describe the main ingredients of the proof.

The first step is categorical.  Let \(\mathcal C\) denote the category of grading-restricted generalized \(V\)-modules.  Using \(C_2\)-cofiniteness, we prove that \(\mathcal C\) is a finite abelian braided category whose objects are precisely the finitely generated weak modules.  The contragredient functor $D$ is an exact Grothendieck--Verdier duality with dualizing object $K=DV,$
and the conformal twist is \(e^{2\pi iL(0)}\).  At this stage neither rigidity nor semisimplicity is assumed.  We establish an abstract internal-Hom criterion for dualizability: for an object \(W\), if the internal unit of \([W,W]\) lies in the image of the canonical rank-one map
\[
u:\,W\boxtimes[W,V]\longrightarrow[W,W],
\]
then \(W\) is rigid.  We also prove an invertibility criterion for the dualizing object \(K\).  In particular, it is enough to construct an object \(B\in\mathcal C\) together with a nonzero morphism
\[
K\boxtimes B\longrightarrow V.
\]

The second step is analytic.  We develop the genus-one trace identities needed in the lower-truncated setting, using geometric coordinates throughout.  Products of intertwining operators give multivalued genus-one correlation functions with the required convergence, analytic continuation and collision expansions.  Moving an insertion around the cycles of the torus is expressed categorically through braiding and monodromy.  Combining collision, torus translation and modular transformation produces comparison identities which relate the assumed expansion of \(SF_V\) to trace functions attached to tensor-product morphisms.  These identities form the analytic bridge between the modular hypothesis and categorical dualizability.

Assume for the moment that the dualizing object \(K\) is invertible.  We then normalize the Grothendieck--Verdier duality by $D_0X=K^{-1}\boxtimes DX,$
so that its dualizing object is the tensor unit \(V\).  For a simple object \(W\), set $R=D_0W, Q=W\boxtimes R.$
The internal-Hom rank-one morphism becomes $\Phi:Q\longrightarrow D_0Q.$
The genus-one comparison formula shows that $\ker\Phi\subseteq\ker t,$
where \(t:Q\to V\) is the twist-modified evaluation morphism.  Exact contravariant duality reverses the corresponding image--coimage factorization, while a self-transpose identity $\Phi=D_0(\Phi)j_Q$
identifies \(\operatorname{Im}D_0(\Phi)\) with \(\operatorname{Im}\Phi\).  Since the internal unit is \(D_0(t)u_0\), it follows that the internal unit lies in \(\operatorname{Im}\Phi\), see sections \ref{sec3.1} and \ref{sec5.1} for details.  The internal-Hom criterion therefore gives rigidity of every simple object.  Since rigid objects are closed under extensions in the present finite-length braided category~\cite{EtingofPenneys}, the entire category \(\mathcal C\) is rigid.

Rigidity alone is not sufficient, so the modular hypothesis is used a second time.  For a rigid object \(W\), its evaluation and coevaluation maps, together with the genus-one comparison morphisms, define a natural endomorphism
$
h_{W,-}:\operatorname{Id}_{\mathcal C}
\longrightarrow
\operatorname{Id}_{\mathcal C}.
$
If \(W=W^l\) and \(W^j\) are simple, Schur's lemma gives
$
h_{W^l,W^j}
=
\lambda_{lj}\operatorname{Id}_{W^j}.
$
The comparison identity yields
\[
SF_{W^l}
=
\sum_j s_j\lambda_{lj}F_{W^j}.
\]
We prove that the complete one-point functions of pairwise inequivalent irreducible ordinary modules are linearly independent.  Hence the matrix \((\lambda_{lj}s_j)\) is invertible.  In particular, the natural transformations \(h_{W^l,-}\) separate distinct simple objects.  A simple categorical argument then shows that
\[
\operatorname{Ext}_{\mathcal C}^1(W^j,W^k)=0,\qquad j\neq k.
\]
Together with the vanishing of the self-extension of the tensor unit in a finite tensor category over \(\mathbb C\)~\cite{EGNO}, this implies that \(V\) is projective.  Rigidity then implies that every object is projective, and consequently \(\mathcal C\) is semisimple.  Since every cyclic weak module belongs to \(\mathcal C\), it follows that every weak \(V\)-module is a direct sum of irreducible ordinary modules.  Thus \(V\) is regular.

It remains to show $K$ is invertible, this is equivalent to constructing a nonzero morphism $K\boxtimes B\longrightarrow V$
without assuming \(K\simeq V\) or nonnegative conformal grading.  This is the third and most VOA-specific part of the proof.  Let \(V_{-s}\) be the lowest nonzero homogeneous subspace and define
\[
f(v)=\tr_{V_{-s}}v_{-1}
\]
on \(V_0\), extending \(f\) by zero to the other homogeneous spaces.  The functional \(f\) annihilates \(C_2(V)\), and the pairing $B_f(a,b)=f(a_1b)$
on \(V_1\) measures the obstruction to correcting the M\"{o}bius operator \(L(1)\).  The complete modular identity forces this obstruction to vanish.  We thereby obtain a correction vector \(h\in V_1\) satisfying the required trace identities.

The zero mode \(h_0\) determines a corrected grading operator.  A lowest-weight argument shows that its generalized zero-eigenspace has the positivity properties needed to construct a corrected M\"{o}bius structure.  After a further harmless correction by a second derivative, we obtain a Virasoro vector with the same translation operator as the original conformal structure and with M\"{o}bius operators acting on all original weak modules.

We then compare the inversion operator associated with the original conformal structure with the inversion operator associated with the corrected M\"{o}bius structure.  Their quotient defines a Laurent-valued pseudoautomorphism \(E(z)\), for definition of pseudoautomorphism see \cite{Li-Pseudo}.  Initially the construction may involve complex powers and logarithms.  A creation argument on the restricted contragredient module removes the monodromy and shows that the relevant zero mode is semisimple with half-integral spectrum.  Consequently \(E(z)\) becomes Laurent valued on every state.  Modification by \(E(z)\) produces another simple ordinary module structure on $V$, denoted by $B=V^E$,
and a nonzero intertwining operator $I^E$ of type $\binom{V}{K\;B},$ which is equivalent to saying
\[
\operatorname{Hom}_{\mathcal C}(K\boxtimes B,V)\neq0.
\]
The categorical invertibility criterion now shows that \(K\) is invertible.  Using the rigidity and semisimplicity arguments described above, we eventually arrive at the regularity and rationality of $V$.

The proof therefore separates naturally into the following chain:
\[
\begin{array}{c}
\text{complete modular identity}
\\[2mm]
\Downarrow
\\[2mm]
\text{conformal correction and relative inversion}
\\[2mm]
\Downarrow
\\[2mm]
\text{a nonzero morphism}\quad K\boxtimes B\longrightarrow V
\\[2mm]
\Downarrow
\\[2mm]
K\text{ invertible}
\\[2mm]
\Downarrow\rlap{\;\text{with complete modular identity}}
\\[2mm]
\mathcal C\text{ rigid}
\\[2mm]
\Downarrow\rlap{\;\text{with complete modular identity}}
\\[2mm]
\mathcal C\text{ semisimple}
\\[2mm]
\Downarrow
\\[2mm]
V\text{ regular}
\\[2mm]
\Downarrow
\\[2mm]
V\text{ rational}.
\end{array}
\]

We emphasize two features of this argument.  First, the hypothesis is genuinely a complete one-point-function hypothesis: the coefficients in
\[
SF_V=\sum_j s_jF_{W^j}
\]
are independent of the state \(v\) and the coordinate $\tau$, and the proof repeatedly uses identities with nontrivial insertions.  An assumption concerning only the vacuum character would not provide the same trace identities used to control the internal-Hom morphisms and extensions.  Second, the proof takes place in the original module category.  The corrected conformal and M\"{o}bius structures are auxiliary devices used to construct the required modified module $B$, the final rigidity, semisimplicity and regularity statements concern the original vertex operator algebra and its original conformal grading.

The paper is organized as follows.  Section~\ref{sec2} fixes the conventions for vertex operator algebras, generalized modules, tensor categories, geometric coordinates and modular transformations.  Section~\ref{sec3} develops the categorical foundations, including the dualizability and invertibility criteria, and proves that the category of grading-restricted generalized modules is a finite abelian braided Grothendieck--Verdier category with dualizing object \(K=DV\).  Section~\ref{sec4} establishes the genus-one analytic properties and the comparison identities which connect the modular \(S\)-transformation with tensor-product morphisms.  Section~\ref{sec5} uses these identities to prove rigidity and semisimplicity under the assumption that the dualizing object is invertible, and deduces regularity.  Section~\ref{sec6} constructs the conformal correction and relative pseudoautomorphism, produces the required nonzero morphism $K\boxtimes B\longrightarrow V,$ this overcomes the assumption used in section \ref{sec5}.
 Thus using the results of section \ref{sec5}, we complete the proof of the main theorem.  The analytic details underlying the genus-one convergence and continuation theorem are collected in the Appendix \ref{seca}.

\section{Preliminaries}\label{sec2}

Throughout the paper all vector spaces and categories are over
\(\C\).  We first recall the basic definitions and conventions concerning
vertex operator algebras and their modules.  We then review the categorical
terminology needed in the sequel, and then recall the geometric
coordinate convention and the modular and elliptic functions appearing
in genus-one trace identities, we finally recall some conventions and facts about complex analysis that will use later.

\subsection{Vertex operator algebras and modules}\label{sec2.1}

We begin with the algebraic notions that will be used throughout the paper.
\begin{definition}\cite{LL,FLM}
A \emph{vertex operator algebra} is a quadruple $(V,Y,\mathbf 1,\omega),$ where $V=\bigoplus_{n\in\Z}V_n$
is a \(\Z\)-graded complex vector space satisfying
\[
    \dim V_n<\infty,
    \qquad
    V_n=0
    \quad\text{for }n\ll0,
\]
\(\mathbf 1\in V_0\) is the vacuum vector, \(\omega\in V_2\) is the conformal vector, and
\[
    Y(\cdot,z):
    V\longrightarrow (\End V)[[z,z^{-1}]]
\]
is the state-field correspondence.

$Y(\textbf{1},z)=\Id$ and $\lim_{z\rightarrow0}Y(u,z)\textbf{1}=u$ for all $u\in V$.

 Let \(u\in V\), we write $Y(u,z)=\sum_{n\in\Z}u_nz^{-n-1}.$ Then for any $v\in V$, we have $u_nv=0$ if $n>>0.$

Write $Y(\omega,z)=\sum_{n\in\Z}L(n)z^{-n-2}.$
The operators \(L(n)\) satisfy the Virasoro relations
\[
    [L(m),L(n)]
    =
    (m-n)L(m+n)
    +
    \frac{m^3-m}{12}
    \delta_{m+n,0}\,c\,\Id_V,
\]
where \(c\in\C\) is the central charge. Moreover, $L(0)v=n v,~v\in V_n.$ In this case, we say \(v\) is homogeneous, its weight is denoted by $\wt v.$ Thus the mode \(v_n\) is homogeneous of degree $\wt v-n-1.$

For \(u,v\in V\), we have the $L(-1)$-bracket property $$[L(-1),Y(u,z)]=Y(L(-1)u,z)=\frac{d}{dz}Y(u,z).$$Taking residue, we have $(L(-1)u)_n=-n\,u_{n-1}.$ Moreover,  the skew-symmetry identity
\begin{eqnarray}\label{id2.1}
    Y(u,z)v=e^{zL(-1)}Y(v,-z)u.
\end{eqnarray}

The Jacobi identity holds for all $u,v\in V$, i.e.
\begin{eqnarray*}
  & &z_0^{-1}\delta(\frac{z_1-z_2}{z_0})Y(u,z_1)Y(v,z_2)-z_0^{-1}\delta(\frac{-z_2+z_1}{z_0})Y(v,z_2)Y(u,z_1)\\
  & &\ \ \ \ \ =z_1^{-1}\delta(\frac{z_2+z_0}{z_1})Y(Y(u,z_0)v,z_2).
\end{eqnarray*}
\end{definition}

Throughout the paper the conformal grading is allowed to have a finite negative part.  Thus, if \(V\neq0\), we may write
\[
    V=\bigoplus_{n\geq -s}V_n
\]
for some \(s\in\mathbb{N}\).  Since the vacuum has weight zero, necessarily $s\geq0.$

\begin{proposition}\cite{LL}
Jacobi identity is equivalent to following two properties:

(i) Weak associativity: for any $u,v,w\in V$, there exists $N\in\mathbb{N}$ such that
\begin{eqnarray}\label{WA}
(z_0+z_2)^lY(u,z_0+z_2)Y(v,z_2)w=(z_0+z_2)^lY(Y(u,z_0)v,z_2)w
\end{eqnarray}
  holds;

(ii)  Weak commutativity: for any $u,v\in V$, there exists $N\in\mathbb{N}$ such that
\begin{eqnarray}\label{WA}
(z_1-z_2)^lY(u,z_1)Y(v,z_2)=(z_1-z_2)^lY(v,z_2)Y(u,z_1)
\end{eqnarray}
  holds.
\end{proposition}

\begin{lemma}\cite{FLM,LL}
Let \(a\in V_1\).  Then \(a_0\) is a degree-zero derivation of \(V\).
In particular,
\[
    a_0(u_nv)
    =
    (a_0u)_nv
    +
    u_n(a_0v).
\]
Since each original homogeneous space \(V_n\) is finite-dimensional,
\(a_0\) is locally finite.
\end{lemma}

\begin{definition}
For any locally finite degree-zero derivation \(\delta\), define
\[
    V^{\mathrm{gen}(\delta=\lambda)}
    =
    \left\{
        u\in V\mid
        (\delta-\lambda)^N u=0
        \text{ for some }N\geq1
    \right\},
\]it is called the generalized eigenspace of $\delta$ with generalized $\delta$-eigenvalue $\lambda.$
\end{definition}

\begin{lemma}
Generalized \(\delta\)-eigenvalues add under all vertex products.  More precisely, for any $u\in V^{\mathrm{gen}(\delta=\lambda)},v\in V^{\mathrm{gen}(\delta=\mu)},$ we have $u_nv\in V^{\mathrm{gen}(\delta=\lambda+\mu)}$.

Consequently, $V^{\mathrm{gen}(\delta=0)}$ is a vertex subalgebra containing \(\mathbf1\).  If \(\delta=a_0\) for \(a\in V_1\), this subalgebra is also stable under \(L(-1)\).
\end{lemma}

\begin{definition}\cite{LL,ABD}
(i) A \emph{weak \(V\)-module} is a vector space \(W\) equipped with a map $Y_W(\cdot,z): V\longrightarrow(\End W)[[z,z^{-1}]],$ written $Y_W(v,z)=\sum_{n\in\Z}v_nz^{-n-1},$ satisfying:

(a) $u_nw=0$ for $n>>0.$

(b) $Y_W(\textbf{1}, x)=\Id_W$.

(c) The Jacobi identity holds, i.e.,
\begin{eqnarray*}
&z_{0}^{-1}\delta\left(\frac{z_{1}-z_{2}}{z_{0}}\right)Y_W(u,z_{1})Y_W(v,z_{2})-z_{0}^{-1}\delta\left(
\frac{z_{2}-z_{1}}{-z_{0}}\right)Y_W(v,z_{2})Y_W(u,z_{1})\\
&\quad=z_{2}^{-1}\delta\left(\frac{z_{1}-z_{0}}{z_{2}}\right)Y_W(Y(u,z_{0})v,z_{2}).
\end{eqnarray*}

(ii) A weak \(V\)-module \(W\) is \emph{admissible} if $W=\bigoplus_{n\in\Zge}W(n)$ and, for every homogeneous \(v\in V\),
\[
    v_m W(n)
    \subset
    W(n+\wt v-m-1).
\]

(iii) A weak \(V\)-module \(W\) is \emph{ordinary} if it is a direct sum of
finite-dimensional \(L(0)\)-eigenspaces and its weights lie in a finite
union of lower-bounded integer translates.  Thus
\[
    W
    =
    \bigoplus_{\lambda\in\C}W_\lambda,
    \qquad
    W_\lambda
    =
    \{w\in W\mid L(0)w=\lambda w\}, \qquad \dim W_\lambda<\infty,
\]
and the set of \(\lambda\) for which \(W_\lambda\neq0\) is contained in a finite union of sets $h+\Zge.$

If \(W\) is irreducible, its lowest conformal weight will be denoted by \(h_W\), called the conformal weight of $W$.

(iv) A weak \(V\)-module \(W\) is a
\emph{grading-restricted generalized \(V\)-module} if $W=\bigoplus_{\lambda\in\C}W_{[\lambda]},$ where
\[
    W_{[\lambda]}
    =
    \left\{
        w\in W\mid
        (L(0)-\lambda)^Nw=0
        \text{ for some }N
    \right\},
\]
each \(W_{[\lambda]}\) is finite-dimensional, and the generalized weights
lie in a finite union of lower-bounded integer translates.

We use \(W_{[\lambda]}\) throughout for the generalized eigenspace of
generalized weight \(\lambda\).
\end{definition}

From definitions, it is obvious that an ordinary module is a grading-restricted generalized module on
which \(L(0)\) acts semisimply.

\begin{proposition}\label{Td}\cite{LL}
Jacobi identity is equivalent to following two properties:

(i) Weak associativity: for any $u,v\in V,~w\in W$, there exists $N\in\mathbb{N}$ such that
\begin{eqnarray}\label{WAM}
(z_0+z_2)^NY_W(u,z_0+z_2)Y_W(v,z_2)w=(z_0+z_2)^NY_W(Y(u,z_0)v,z_2)w
\end{eqnarray}
  holds;

(ii)  Weak commutativity: for any $u,v\in V$, there exists $N\in\mathbb{N}$ such that
\begin{eqnarray}\label{WAM}
(z_1-z_2)^NY_W(u,z_1)Y_W(v,z_2)=(z_1-z_2)^NY_W(v,z_2)Y_W(u,z_1)
\end{eqnarray}
  holds.
\end{proposition}

\begin{definition}\cite{ABD,LL,Zhu}
(i) $V$ is called \emph{simple} if it is a simple $V$-module.

(ii)  \(V\) is called \emph{rational} if every admissible \(V\)-module is completely reducible.

(iii) $V$ is called \emph{regular} if every weak \(V\)-module is a direct sum of
irreducible ordinary \(V\)-modules.
\end{definition}

\begin{definition}\cite{Li-Abelian,Zhu}
(i) For \(m\geq2\), define $C_m(V)=\Span_{\C}\{u_{-m}v\mid u,v\in V\}.$ The vertex operator algebra \(V\) is called \emph{\(C_m\)-cofinite} if $\dim V/C_m(V)<\infty.$

(ii) For a \(V\)-module \(W\), we similarly write $C_m(W)=\Span_{\C}\{u_{-m}w\mid u\in V,\ w\in W\},$ $W$ is called \emph{$C_m$-cofinite} if $\dim W/C_m(W)<\infty.$

(iii) If $m=1$, define $$C_1^+(W)=\{u_{-1}w\mid w\in W,u\in\oplus_{n>0}V_n\},$$$$C_1^a(V)=\{u_{-1}v\mid u,v\in \oplus_{n>0}V_n\}+L(-1)V.$$$W,V$ are called \emph{$C_1$-cofinite} if $\dim W/C_1^+(W)<\infty,~\dim V/C_1^a(V)<\infty$, respectively.
\end{definition}

\begin{remark}\cite{ABD,Li-Abelian}
(i) Regularity implies rationality.  Under suitable cofiniteness hypotheses, the converse is also true; in the present paper regularity will be obtained directly from the complete genus-one trace identity.

(ii) $C_2$-cofiniteness of $V$ implies $C_m$-cofiniteness of $V$ for $m>1.$
\end{remark}

The unrestricted \(C_m\)-convention above is the convention used in the
later cofiniteness arguments: the state \(u\) is not required to have
positive conformal weight.

\begin{definition}\cite{FHL}\label{def2.9}
(i) Let \(W\) be a grading-restricted generalized \(V\)-module. Define $W'=\bigoplus_{\lambda\in\C}W_{[\lambda]}^*,$ and
\[
\left\langle
    Y_{W'}(u,z)f,w
\right\rangle
=
\left\langle
    f,
    Y_W\!\left(
       e^{zL(1)}
       (-z^{-2})^{L(0)}u,
       z^{-1}
    \right)w
\right\rangle.
\]
Then $(W',Y_{W'})$ is a grading-restricted generalized \(V\)-module, it is called the \emph{contragredient module} of $W$.

(ii) A vertex operator algebra \(V\) is called \emph{self-contragredient} if $V'\cong V$ as \(V\)-modules.
\end{definition}

\begin{definition}\cite{HLZ,Fiordalisi}
  (i) Let \(W^1,W^2,W^3\) be grading-restricted generalized \(V\)-modules. A \emph{logarithmic intertwining operator} of type $\binom{W^3}{W^1\;W^2}$ is a linear map
  \begin{eqnarray*}
  \mathcal {Y}(\cdot,z):~W^1&\rightarrow&~\Hom(W^2,W^3)\{z\}[\log z]],\\
  w^1&\mapsto& \mathcal {Y}(w^1,z)=\sum_{t\in\mathbb{C}}\sum_{k\in\mathbb{N}}w^1_{t,k}z^{-t-1}(\log z)^{k},
  \end{eqnarray*}
satisfying:

  (a) For any $w^1\in W^1,w^2\in W^2$, $w^1_{t,k}w^2=0$ for $\Re t\in\mathbb{Z}$ sufficiently large, where $\Re t$ is the real part of $t$.

  (b) For any $w^1\in W^1$, $\frac{d}{dz}\mathcal {Y}(w^1,z)=\mathcal {Y}(L(-1)w^1,z)$.

  (c) For any $v\in V, w^1\in W^1$, the following Jacobi identity holds:
  \begin{eqnarray*}
  z_0^{-1}\delta(\frac{z_1-z_2}{z_0})Y_{W^3}(v,z_1)\mathcal {Y}(w^1,z_2)-z_0^{-1}\delta(\frac{-z_2+z_1}{z_0})\mathcal {Y}(w^1,z_2)Y_{W^2}(v,z_1)\\
  =z_2^{-1}\delta(\frac{z_1-z_0}{z_2})\mathcal {Y}(Y_{W^1}(v,z_0)w^1,z_2).
  \end{eqnarray*}

  (ii) The set of all intertwining operators of type $\tbinom{W^3}{W^1~~W^2}$ forms a vector space, denoted by $I_V\tbinom{W^3}{W^1~~W^2}$. Its dimension is called the \emph{fusion rule}, denoted by $N_V\tbinom{W^3}{W^1~~W^2}$,
\end{definition}

\begin{lemma}
For a logarithmic intertwining operator $\mathcal Y(\cdot,z)$ and $w^1\in W^1$, we have $$[L(0),\mathcal Y(w^1,z)]=z\mathcal Y(L(-1)w^1,z)+\mathcal Y(L(0)w^1,z).$$
\end{lemma}

\begin{remark}
If \(W^1,W^2,W^3\) are ordinary modules, the \(L(0)\)-commutator identity forces all logarithmic coefficients of a logarithmic intertwining
operator of type $\binom{W^3}{W^1\;W^2}$ to vanish. Thus logarithmic intertwining operators between ordinary modules reduce
to ordinary intertwining operators.
\end{remark}

\begin{definition}\cite{HL1,HL2,HL3,HL4,HLZ}
Let \(\mathcal C\) be a category of grading-restricted generalized \(V\)-modules.  A \emph{tensor product} of \(W^1,W^2\in\mathcal C\), when it
exists, is an object $W^1\boxtimes W^2$ equipped with a universal tensor-product intertwining operator $\mathcal Y^{W^1,W^2}$ such that composition gives natural isomorphisms
\[
    \Hom_{\mathcal C}(W^1\boxtimes W^2,W^3)
    \cong
    \left\{
       \text{intertwining operators of type }
       \binom{W^3}{W^1\;W^2}
    \right\}
\]
for \(W^3\in\mathcal C\).
\end{definition}

In logarithmic tensor product theory, this tensor product arises from the
\(P(z)\)-tensor products and their parallel transport.  We use the symbol
\(\boxtimes\) for the tensor product throughout the rest of the paper.

\subsection{Tensor categories and vertex tensor categories}\label{subsec2.2}

We next recall the categorical terminology needed later. Basic categories and functors may be found in \cite{Ma}. General background on finite tensor categories may be found in \cite{BK,EGNO}.  The Grothendieck--Verdier terminology used here follows \cite{ALSW}.

\begin{definition}\cite{ALSW,EGNO}\label{def2.13}
(i) A \emph{monoidal category} consists of a category \(\mathcal C\), a bifunctor $\boxtimes:\mathcal C\times\mathcal C\longrightarrow\mathcal C,$
a unit object \(V\), natural associativity isomorphisms $\alpha_{X,W,Z}:(X\boxtimes W)\boxtimes Z\longrightarrow X\boxtimes(W\boxtimes Z),$ and left and right unit isomorphisms $l_X:V\boxtimes X\longrightarrow X,r_X:X\boxtimes V\longrightarrow X,$ satisfying the pentagon and triangle axioms.

(ii) A monoidal category \(\mathcal C\) is called \emph{a braided monoidal category}, if it is equipped with natural isomorphisms $c_{X,W}:X\boxtimes W\longrightarrow W\boxtimes X$ satisfying the two hexagon identities.

(iii) A \emph{twist} on a braided monoidal category is a natural isomorphism $\theta_X:X\longrightarrow X$ such that $\theta_V=\Id_V$
and $\theta_{X\boxtimes W}=c_{W,X}c_{X,W}(\theta_X\boxtimes\theta_W).$ A braided monoidal category equipped with such a twist is called
\emph{balanced}.

(iv) A \(\C\)-linear abelian category \(\mathcal C\) is \emph{finite abelian} if its morphism spaces are finite-dimensional, every object has finite length, there are finitely many simple objects up to isomorphism, and all simple objects admit projective covers.

(v) A \emph{Grothendieck--Verdier duality} on a monoidal category \(\mathcal C\) consists of an object \(K\) and a contravariant
equivalence $D:\mathcal C^{\mathrm{op}}\longrightarrow\mathcal C$ together with natural isomorphisms
\[
    \Hom_{\mathcal C}(X\boxtimes W,K)
    \cong
    \Hom_{\mathcal C}(X,DW).
\]
The object \(K\) is called the \emph{dualizing object}.

(vi) Let \(W\) be an object in a monoidal category.  A \emph{left rigid dual} of \(W\) consists of an object \(W^*\) and morphisms $e:W^*\boxtimes W\longrightarrow V,i:V\longrightarrow W\boxtimes W^*$ satisfying the triangle identities
\[
    (\Id_W\boxtimes e)
    (i\boxtimes\Id_W)
    =
    \Id_W,
\qquad
 (e\boxtimes\Id_{W^*})
    (\Id_{W^*}\boxtimes i)
    =
    \Id_{W^*}.
\]
Canonical unit and associativity morphisms are suppressed in these formulas.

(vii) An object is \emph{rigid} if it admits the required duals, and a monoidal category is rigid if all its objects are rigid. In a braided category a left rigid dual also determines a right rigid dual.

(viii) An object \(K\) of a monoidal category is \emph{invertible} if there is an object \(B\) such that
\[
    K\boxtimes B\cong V
    \cong
    B\boxtimes K.
\]

(ix) A \emph{finite tensor category} is a rigid finite abelian monoidal category whose unit object is simple. A \emph{fusion category} over \(\C\) is a semisimple finite tensor category.

(x) $\mathcal {C}$ is called \emph{ribbon Grothendieck Verdier} if it is a braided Grothendieck-Verdier such that $$D\theta_X=\theta_{DX}$$ for all $X\in\mathcal {C}.$

(xi) A \emph{ribbon category} is a rigid braided monoidal category equipped with a twist compatible with duality, i.e., if \(X^*\) is a
chosen rigid dual of \(X\), then the twist satisfies the usual duality compatibility $\theta_{X^*}=(\theta_X)^*.$

(xii) Let $\mathcal {C}$ be a rigid ribbon category, for $W\in\mathcal {C}$ and $f\in\End(W)$, the \emph{trace} of $f$ is defined by $$\tr f=e_W\circ c_{W,W^*}\circ (f\boxtimes \Id_{W^*})\circ(i_W)\in\End(V).$$If $V$ is simple, then $\tr f$ is a scalar.

(xiii) A \emph{modular tensor category} is a nondegenerate ribbon fusion category.  Equivalently, in the semisimple setting over \(\C\), it is a ribbon fusion category whose categorical \(S\)-matrix is nondegenerate, where $S=(S_{ij})_{i,j}$ and $S_{ij}=\tr c_{X_j,X_i}\circ c_{X_i,X_j},$ and $X_i$ are the simple objects of $\mathcal {C}$.

(xiv) Let \(\mathcal C\) be a monoidal category.  For objects \(W,Z\), an \emph{internal Hom} \([W,Z]\), when it exists, is an object representing the functor $X\longmapsto\Hom_{\mathcal C}(X\boxtimes W,Z).$
Thus there are natural isomorphisms
\[
    \Hom_{\mathcal C}(X,[W,Z])
    \cong
    \Hom_{\mathcal C}(X\boxtimes W,Z).
\]
The associated evaluation morphism has source $[W,Z]\boxtimes W$
and target \(Z\).
\end{definition}

\begin{remark}
(i) The notation for the unit object has deliberately been chosen to be \(V\), since the unit object of a vertex tensor category of \(V\)-modules is the adjoint module \(V\).

(ii) In a vertex tensor category, the braiding is obtained by analytic
interchange of insertion points.

(iii) When \(K\cong V\), such a category is also called an
\(r\)-category.  A Grothendieck--Verdier duality does not, by itself,
imply rigidity.

(iv) The notion of a \emph{ribbon Grothendieck--Verdier category} used later is more general, i.e., it carries a Grothendieck--Verdier duality, braiding, and compatible twist, but rigidity is not assumed a priori \cite{ALSW}.
\end{remark}

We use the notion of a vertex tensor category in the sense of
Huang--Lepowsky--Zhang \cite{HLZ}, and recall only the structures
and notations needed below. Let \(\mathcal C\) be a category of
grading-restricted generalized \(V\)-modules. A vertex tensor category
structure on \(\mathcal C\) includes, in particular, \(P(z)\)-tensor
product bifunctors
\[
    \boxtimes_{P(z)}:
    \mathcal C\times\mathcal C\longrightarrow\mathcal C,
    \qquad z\in\mathbb C^\times,
\]
representing \(P(z)\)-intertwining operators, together with
parallel-transport, associativity, braiding, and unit isomorphisms
satisfying the corresponding compatibility and coherence conditions.

Fixing \(z=1\), together with the standard choices of paths for
parallel transport, yields the braided monoidal structure that will
be used throughout this paper. We write
\[
    W^1\boxtimes W^2
    =
    W^1\boxtimes_{P(1)}W^2,
\]
and denote the associativity and braiding isomorphisms by
\[
    \alpha_{X,W,Z}:
    (X\boxtimes W)\boxtimes Z
    \xrightarrow{\sim}
    X\boxtimes(W\boxtimes Z),
    \qquad
    c_{X,W}:
    X\boxtimes W\xrightarrow{\sim}W\boxtimes X.
\]
The tensor unit is \(V\), with left and right unit isomorphisms
\[
    l_X:V\boxtimes X\xrightarrow{\sim}X,
    \qquad
    r_X:X\boxtimes V\xrightarrow{\sim}X.
\]

We will use the following existence theorem for vertex tensor
category structures.

\begin{proposition}\cite{Huang-C1}\label{prop2.15}
Let \(V\) be a vertex operator algebra, and let
\(\mathcal C_1(V)\) be the category of \(C_1\)-cofinite
grading-restricted generalized \(V\)-modules. Then
\(\mathcal C_1(V)\) carries a natural vertex tensor category
structure. In particular, after fixing \(z=1\) and the standard
parallel-transport paths, \(\mathcal C_1(V)\) is a braided monoidal
category with twist
\[
    \theta_W=e^{2\pi iL(0)}.
\]
\end{proposition}

The tensor-product intertwining operator will be denoted by $\mathcal Y^{A,B}$ for objects \(A,B\in\mathcal C\), and the associativity isomorphism by $\alpha_{A,B,C}:(A\boxtimes B)\boxtimes C\longrightarrow A\boxtimes(B\boxtimes C).$
The unit constraints will be denoted by \(l\) and \(r\). We also use $\mathsf M^{A,B}$ to denote the double braiding $c_{B,A}\circ c_{A,B}$ on $A\boxtimes B.$

\begin{remark}
Under the \(C_2\)-cofiniteness and lower-truncated integer-grading
hypotheses considered in this paper, Section~\ref{sec3} proves that the category
\(\mathcal C\) of grading-restricted generalized \(V\)-modules is finite
abelian and has the required vertex tensor category structure. The Grothendieck-Verdier duality is its actual contragredient duality, and denoted by \(D\), its dualizing object is $K=DV=V',$ and
\[
    \Hom_{\mathcal C}(X\boxtimes W,K)
    \cong
    \Hom_{\mathcal C}(X,DW).
\]
The conformal twist is
\[
    \theta_X=e^{2\pi iL(0)}.
\]
At this stage rigidity is \emph{not} assumed and proving rigidity is one of the main steps of the argument.
\end{remark}

For comparison, in the classical semisimple situation one has the
following theorem.

\begin{theorem}\cite{Hu1}
Let \(V\) be a simple vertex operator algebra satisfying the standard
positive-energy assumptions, \(C_2\)-cofiniteness,
self-contragredience, and rationality.  Then the braided tensor category
of ordinary \(V\)-modules is rigid and nondegenerate and hence is a
modular tensor category.
\end{theorem}

\begin{remark}
The preceding modular-tensor-category theorem is included only for
context.  The main argument of the present paper does not assume from
the outset that the module category is rigid or semisimple.
\end{remark}

\subsection{Modular forms, geometric coordinates, and modular transformations}\label{sec2.3}

We finally recall the analytic and geometric notation used for genus-one
trace functions.  Our normalization follows
\cite{Zhu,DLM-Modular,Fiordalisi,La1,McRae}.

\begin{definition}\cite{La}
(i) Let $\Hh=\{\tau\in\C:\operatorname{Im}\tau>0\}.$
The modular group $SL_2(\Z)$ acts on \(\Hh\) by M\"{o}bius transformation. It is generated by
\[
    S=
    \begin{pmatrix}
       0&-1\\
       1&0
    \end{pmatrix},
\qquad
    T=
    \begin{pmatrix}
       1&1\\
       0&1
    \end{pmatrix}.
\]

(ii) Let \(k\in\Z\).  A holomorphic function \(f:\Hh\to\C\) is a \emph{modular form} of weight \(k\) for \(SL_2(\Z)\) if it is holomorphic at the cusp such that for every $
\begin{psmallmatrix}
a&b\\c&d
\end{psmallmatrix}
\in SL_2(\Z)$,
\[
    f\!\left(\frac{a\tau+b}{c\tau+d}\right)
    =
    (c\tau+d)^k f(\tau)
\].

(iii) Modular forms of weight $0$ are called \emph{modular functions}.
\end{definition}

\begin{remark}For any \(x\in\C\), we use the notation $q_x=e^{2\pi ix}.$ In particular, $q=q_\tau=e^{2\pi i\tau}.$
\end{remark}

\begin{definition}\cite{Hu2,Hu1997}
(i) Let \(A_j\in\C\), \(j\geq1\) be defined by the formal identity
\[
    \frac{e^{2\pi ix}-1}{2\pi i}
    =
    \exp\left(
       -\sum_{j\geq1}
       A_jx^{j+1}\frac{d}{dx}
    \right)x,
    \qquad
    A_1=-\pi i.
\]

(ii) The \emph{geometric coordinate-change operator} is
\[
    \mathcal {U}(1)
    =
    (2\pi i)^{L(0)}
    \exp\left(
       -\sum_{j\geq1}A_jL(j)
    \right),
    \qquad
    \mathcal {U}(q_z)
    =
    q_z^{L(0)}\mathcal {U}(1),
    \qquad
    q_z=e^{2\pi iz}.
\]
\end{definition}

Since the positive Virasoro modes lower the original conformal weight, the exponential $\exp\left(-\sum_{j\geq1}A_jL(j)\right)$ has finite state support on every fixed state.

On generalized module weight spaces we use the branch convention $\log(2\pi i)=\log(2\pi)+\frac{\pi i}{2}.$ These conventions agree with the geometric genus-one formalism of \cite{Fiordalisi,McRae}.

\begin{lemma}\cite{Hu2}\label{conjugate}
Let $\mathcal Y(\cdot,z)$ be an intertwining operator of type $\binom{W^3}{W^1~~W^2}$. for $w^1\in W^1$, we have the following conjugation identities:
\begin{eqnarray}
&&q_z^{L(0)}\mathcal Y(w^1,x)q_z^{-L(0)}=\mathcal Y(q_z^{L(0)}w^1,q_zx),\label{con-q}\\
&&\mathcal {U}(1)\mathcal Y(w^1,x)\mathcal {U}(1)^{-1}=\mathcal Y(e^{2\pi ixL(0)}\mathcal {U}(1)w^1,e^{2\pi ix}-1).\label{con-U}
\end{eqnarray}
\end{lemma}

\begin{lemma}\cite{Hu2}\label{conjugate-L(-1)}
 We have the following commutate relation $$\mathcal U(1)L(-1)=2\pi i(L(-1)+L(0))\mathcal U(1).$$
\end{lemma}

\begin{definition}\cite{Hu2}
(i) For a homogeneous \(v\in V\), define $o_W(v)=v_{\wt v-1},$ and extend \(o_W\) linearly to all of \(V\). For an ordinary \(V\)-module \(W\), define
\[
    F_W(v,\tau)
    =
    \tr_W
    o_W(\mathcal {U}(1)v)
    q^{L(0)-c/24},
    \qquad
    \tau\in\Hh,
\] called the \emph{complete ordinary one-point function} of $W$.

(ii) When $v=\textbf{1}$, define $\chi_W=F_W(\mathbf1,\tau)=\tr_W q^{L(0)-c/24}$, called the \emph{character} of $W$.
\end{definition}

Since \(\mathcal {U}(1)\) is invertible, the complete ordinary one-point function \(F_W(v,\tau)\) contains the same information as the usual Zhu's trace functions $\mathcal F_W(v,\tau)$. However, for a fixed insertion the geometrically modified trace and the unmodified Zhu's trace need not coincide.

\begin{definition}\label{ST}\cite{Hu1}
In geometric coordinates, the two standard generators act by
\[
    (SF)(v,\tau)
    =
    F\!\left(
       (-\tau^{-1})^{L(0)}v,
       -\tau^{-1}
    \right),
    \qquad
    (TF)(v,\tau)
    =
    F(v,\tau+1).
\]
\end{definition}

This is the geometric-coordinate version of Zhu's modular action using
square-bracket conformal weights.  With this convention, \(S\)
corresponds to the inverse of the usual \(S\)-matrix convention.

\begin{proposition}\cite{DLM-Modular}
Let \(W\) be an irreducible ordinary \(V\)-module with lowest conformal weight \(h_W\). Then $TF_W=e^{2\pi i(h_W-c/24)}F_W.$ No rationality assumption on \(h_W\) is required for this identity.
\end{proposition}

Let $V=W^0,W^1,\ldots,W^r$ be a complete set of irreducible ordinary \(V\)-modules, we write
\[
    \mathcal H
    =
    \Span_{\C}
    \{F_{W^0},\ldots,F_{W^r}\}.
\]
Thus modular invariance of the complete ordinary one-point functions means invariance of \(\mathcal H\) under the modular action above.

We next recall the elliptic coefficient functions appearing in the genus-one differential equations.

\begin{definition}\cite{Fiordalisi}
For \(k\geq1\), put
\[
    \widetilde G_{2k}(q)
    =
    2\zeta(2k)
    +
    \frac{2(2\pi i)^{2k}}{(2k-1)!}
    \sum_{n\geq1}
    \frac{n^{2k-1}q^n}{1-q^n},
\qquad
    G_{2k}(\tau)
    =
    \widetilde G_{2k}(q_\tau),
\]where $\zeta$ is the Riemann zeta function.
\end{definition}

For \(k\geq2\), \(G_{2k}\) is the lattice Eisenstein series for $\Z+\Z\tau$ and is a modular form of weight \(2k\).  The function \(G_2\) is the
usual quasimodular Eisenstein series.  In the present normalization,
\[
    G_2\!\left(
       \frac{a\tau+b}{c\tau+d}
    \right)
    =
    (c\tau+d)^2G_2(\tau)
    -
    2\pi i\,c(c\tau+d).
\]

For \(k\geq2\), the Eisenstein series \(G_{2k}\) is a polynomial in
\(G_4\) and \(G_6\).

\begin{definition}\cite{Fiordalisi}
Let $\rho_\tau=\min\left\{|a+b\tau|\mid (a,b)\in\Z^2\setminus\{(0,0)\}\right\}.$ For $0<|z|<\rho_\tau,$ define
\[
    \wp_1(z;\tau)
    =
    \frac1z
    -
    \sum_{k\geq1}
    G_{2k+2}(\tau)z^{2k+1},
\qquad
    \wp_m(z;\tau)
    =
    \frac{(-1)^{m-1}}{(m-1)!}
    \partial_z^{\,m-1}\wp_1(z;\tau), ~m>1.
\]
We write $\widetilde{\wp}_m(z;q)$ for the corresponding expansion in the variable \(q\).
\end{definition}

\begin{proposition}\cite{Fiordalisi}
The functions \(\wp_m(z;\tau)\) have the following properties:

(i) The local series above converge for $0<|z|<\rho_\tau$, and admit meromorphic continuation.

(ii) The function \(\wp_2(z;\tau)\) is the ordinary Weierstrass \(\wp\)-function. For \(m\geq2\), \(\wp_m\) is elliptic

(iii) The only poles in the insertion variable occur at the lattice points $z\in\Z+\Z\tau$.

(iv) All \(G_{2k}\), \(k\geq2\) are polynomials in \(G_4,G_6\), and all \(\wp_m\), \(m\geq2\) are polynomials in $\wp_2,\ \wp_3,\ G_4,\ G_6.$

(v) The corresponding \(q\)-expansions are holomorphic near \(q=0\) when the insertion variable remains in a sufficiently small neighborhood avoiding the limiting collision set \(\Z\).
\end{proposition}

\begin{remark}
The three structures reviewed above are related as follows. Intertwining operators give rise, under the appropriate convergence and cofiniteness
conditions, to the tensor product and the braided vertex tensor category of \(V\)-modules.  Geometrically modified intertwining operators and
their genus-one traces then connect this tensor category with modular and elliptic functions.  In a nonsemisimple \(C_2\)-cofinite theory, modular closure may in general involve pseudo-traces \cite{Miyamoto}; the hypothesis studied in this paper is precisely that the relevant complete ordinary one-point functions already close under the required modular transformation.
\end{remark}

\subsection{Complex-analytic conventions}\label{sec2.4}
We collect the elementary complex-analytic facts and conventions from \cite{A,C,H,W} that will be used in the following sections. In particular, we allow complex exponents throughout this paper.

Let \(U\subset \C^\times\) be a simply connected domain. A
holomorphic branch $\Log_U:U\longrightarrow\C$ of the logarithm determines a
holomorphic function $z^\lambda=\exp\bigl(\lambda\Log_U z\bigr)$ for every \(\lambda\in\C\). Whenever a branch has been fixed, we simply write $\log z=\Log_U z.$ If \(z\) is analytically continued once counterclockwise around the origin, then $\log z$ becomes $\log z+2\pi i$
and consequently $z^\lambda$ becomes $e^{2\pi i\lambda}z^\lambda.$

More generally, let \(f_k\) be single-valued holomorphic functions on an
annulus, then the expression $z^\lambda\sum_{k=0}^{K}(\log z)^k f_k(z)$ has finite-dimensional local monodromy. If
\(\mathsf M\) denotes counterclockwise local monodromy around
\(z=0\), then
\[
    \mathsf M
    \left(
        z^\lambda(\log z)^k f_k(z)
    \right)
    =
    e^{2\pi i\lambda}
    z^\lambda
    (\log z+2\pi i)^k f_k(z).
\]
In particular, $\bigl(\mathsf M-e^{2\pi i\lambda}\bigr)^{K+1}=0$
on the span of the terms with exponent \(\lambda\) and logarithmic
degree at most \(K\).

Two exponents \(\lambda,\mu\in\C\) determine the same monodromy
eigenvalue precisely when $\lambda-\mu\in\Z.$
Hence exponent classes naturally occur modulo \(\Z\). If $\lambda_a-\lambda_b\notin\Z$, then $e^{2\pi i\lambda_a}\neq e^{2\pi i\lambda_b},$
and polynomial spectral projections in \(\mathsf M\) separate the
corresponding exponent classes.

After an exponent class has been isolated, logarithmic powers can be
separated by finite differences. Let $p(\log z)$ be a polynomial in $\log z$, define
\[
    \Delta p(\log z)
    =
    p(\log z+2\pi i)-p(\log z).
\]
Then \(\Delta\) lowers the logarithmic degree by one, and $\Delta^K(\log z)^K=K!(2\pi i)^K.$
Consequently, if
\[
    \sum_{a=1}^{J}
    \sum_{k=0}^{K}
    z^{\lambda_a}(\log z)^k f_{a,k}(z)
    =0,
    \qquad
    \lambda_a-\lambda_b\notin\Z
    \quad(a\neq b),
\]
where every \(f_{a,k}\) is a single-valued holomorphic function on an
annulus, then polynomial projections in the monodromy first isolate
each \(a\), and successive finite differences isolate the highest
remaining logarithmic degree.

In particular, if all \(f_{a,k}\) are represented by convergent
Laurent series, then
\[
    \sum_{a=1}^{J}
    \sum_{k=0}^{K}
    z^{\lambda_a}(\log z)^k f_{a,k}(z)
    =0  \qquad
\Rightarrow   \qquad
    f_{a,k}=0
    \qquad
    \text{for all }a,k.
\]

Let $A(r,R)=\{z\in\C:r<|z|<R\},0\le r<R\le\infty$ be an annulus.
If \(f\) is holomorphic on \(A(r,R)\), then it has a unique Laurent
expansion $f(z)=\sum_{n\in\Z}a_nz^n$
which converges absolutely and locally uniformly on \(A(r,R)\).
For every \(\rho\) satisfying $r<\rho<R,$ the coefficients are given by
\[
    a_n
    =
    \frac{1}{2\pi i}
    \int_{|\zeta|=\rho}
    \frac{f(\zeta)}{\zeta^{n+1}}\,d\zeta.
\]
In particular, the Laurent series converges uniformly on every compact
subannulus.

Suppose that \(P\subset\C^m\) is a domain and $F:A(r,R)\times P\longrightarrow\C$
is jointly holomorphic. Then
\[
    a_n(p)
    =
    \frac{1}{2\pi i}
    \int_{|\zeta|=\rho}
    \frac{F(\zeta,p)}{\zeta^{n+1}}\,d\zeta
\]
is holomorphic in \(p\), and $F(z,p)=\sum_{n\in\Z}a_n(p)z^n$
converges locally uniformly on \(A(r,R)\times P\).

\begin{remark}
This fact will be used when a genus-one correlation function is
simultaneously expanded in a collision variable and in the modular
parameter. Once the relevant expansions are identified with Laurent
or power-series expansions of the same jointly holomorphic function,
absolute and locally uniform convergence on compact subsets permits
the corresponding orders of summation to be interchanged.
\end{remark}

Let \(X\) be a complex manifold and \(D\subset X\) an analytic
divisor. A germ of a holomorphic function on \(X\setminus D\) may be
analytically continued along paths avoiding \(D\). On a simply
connected open subset of \(X\setminus D\), a chosen continuation is
single valued. Continuation around a closed loop may change the
branch; this defines the local monodromy.

In the genus-one applications below, for fixed
\(\tau\in\Hh\), the collision divisors in the insertion variables are $z_i-z_j\in\Z+\Z\tau.$ Thus, in a collision coordinate $z=z_i-z_j,$
the possible collision points form the lattice $\Z+\Z\tau.$ If $\rho_\tau=\min\left\{|a+b\tau|\mid (a,b)\in\Z^2\setminus\{(0,0)\}\right\},$
then the punctured disk $0<|z|<\rho_\tau$ contains no lattice point other than \(z=0\).

All branches used below are obtained by analytic continuation from
the initially specified convergence region along paths avoiding the
collision divisors.

Let \(P\subset\C^m\) be a domain and suppose $A(p,q)$ is a matrix-valued function holomorphic on $P\times\{q\mid |q|<\epsilon\}.$
Consider $q\frac{\partial}{\partial q}\Phi(p,q)=A(p,q)\Phi(p,q),$
equivalently,
\[
    \frac{\partial}{\partial q}\Phi(p,q)
    =
    \frac{1}{q}A(p,q)\Phi(p,q),
\]
so \(q=0\) is a regular singular point.

After choosing a branch of \(\log q\), local solutions have
Frobenius-type expansions of the form
\[
    \Phi(p,q)
    =
    \sum_{a=1}^{J}
    \sum_{k=0}^{K}
    \sum_{n\ge0}
    \phi_{a,k,n}(p)
    q^{h_a+n}
    (\log q)^k,
\]
where $h_a\in\C.$ The exponents may be chosen distinct modulo \(\Z\), and need not be real or rational.

\begin{remark}
The regular-singular convergence results used below imply absolute
and locally uniform convergence on sufficiently small punctured
neighborhoods. If the coefficient matrix depends holomorphically on
auxiliary parameters, the corresponding estimates are locally uniform
in those parameters.
\end{remark}

If two holomorphic functions on a connected domain agree on a
nonempty open subset, then they agree everywhere. The same principle
applies to two branches of multivalued analytic functions once they
have been analytically continued to the same simply connected branch.

\section{Categorical foundations}\label{sec3}
In this section, we first establish an abstract criterion for dualizability and
invertibility in an abelian monoidal category with internal Homs.
We then specialize to the module category of the vertex operator
algebra $V$.  Under $C_2$-cofiniteness, this category is a finite
abelian braided tensor category with contragredient
Grothendieck--Verdier duality and dualizing object $K=DV.$
This categorical structure, together with the preceding
dualizability criterion, will be used later to prove rigidity.

\subsection{A criterion for dualizability}\label{sec3.1}
Recall from Definition \ref{def2.13} the definition of internal-Hom object, for \(W,Z\in\mathcal C\), an \emph{internal Hom} \([W,Z]\) is an object such that
\begin{eqnarray}\label{id5.1}
    \Hom_{\mathcal C}(X,[W,Z])
    \cong
    \Hom_{\mathcal C}(X\boxtimes W,Z).
\end{eqnarray}
The associated evaluation morphism $a_{W,Z}$ has source $[W,Z]\boxtimes W$
and target \(Z\). Now, for any $f\in\Hom_{\mathcal C}(X\boxtimes W,Z),$ the adjoint morphism $\Phi(f)\in \Hom_{\mathcal C}(X, [W,Z])$ satisfying $$a_{W,Z}(\Phi(f)\boxtimes \Id_W)=f.$$ In the following, we will omit the associativity and unit isomorphisms for simplicity.

\begin{lemma}\label{lm5.1}
Let $\mathcal C$ be an abelian monoidal category admitting internal
Homs, and assume that the tensor product is right exact in each
variable.  Fix an object $W\in\mathcal C$ and put
\[
B=[W,\mathbf1],
\qquad
A=[W,W],
\qquad
Q=W\boxtimes B.
\]

Let $e:B\boxtimes W\longrightarrow\mathbf1$
be the evaluation morphism, and $u:\mathbf1\longrightarrow A$
the internal unit adjoint to $\Id_W$ under identity (\ref{id5.1}).

Define $\rho:Q\longrightarrow A$
to be the adjoint of
$\Id_W\boxtimes e:
(W\boxtimes B)\boxtimes W
\longrightarrow W$ under (\ref{id5.1}).
If the internal unit $u$ factors through $\im\rho$,
then $B$ is a left dual of $W$.
In particular, there exists a coevaluation $i:\mathbf1\longrightarrow W\boxtimes B$
such that
\begin{eqnarray}\label{8.2}
(\Id_W\boxtimes e)
(i\boxtimes\Id_W)
=
\Id_W,\qquad
(e\boxtimes\Id_B)
(\Id_B\boxtimes i)
=
\Id_B.
\end{eqnarray}
Furthermore, if $\mathcal C$ is braided, then $W$ is rigid.
\end{lemma}
\proof (i) First, let $a:A\boxtimes W\longrightarrow W$ be the evaluation morphism. From definitions, we have $a(u\boxtimes \Id_W)=\Id_W.$ Then the internal endomorphism object $A=[W,W]$
carries a canonical unital algebra structure with
\[
\mu:A\boxtimes A\longrightarrow A, \qquad   (A\boxtimes A)\boxtimes W
\xrightarrow{\alpha_{A,A,W}}
A\boxtimes(A\boxtimes W)
\xrightarrow{\Id_A\boxtimes a}
A\boxtimes W
\xrightarrow{a}
W.
\]
and $u:\mathbf1\longrightarrow A$ is defined above.

Since $B=[W,\mathbf1],$ there is a right
$A$-action $r:B\boxtimes A\longrightarrow B$ on $B$ characterized by
\[
e(r\boxtimes\Id_W)
=
e(\Id_B\boxtimes a),\qquad B\boxtimes A\boxtimes W\longrightarrow \textbf{1}.
\]
For $Q=W\boxtimes B,$ the left and right $A$-actions on $Q$ are
\begin{eqnarray*}
L
&=&
a\boxtimes\Id_B:
A\boxtimes Q\longrightarrow Q,\\
R
&=&
\Id_W\boxtimes r:
Q\boxtimes A\longrightarrow Q.
\end{eqnarray*}
Define also the contraction
$m=\Id_W\boxtimes e\boxtimes\Id_B:
Q\boxtimes Q
\longrightarrow Q.$
Recall that $\rho:Q\longrightarrow A$
is adjoint to
$\Id_W\boxtimes e: Q\boxtimes W\longrightarrow W,$ i.e., $a(\rho\boxtimes\Id_W)=\Id_W\boxtimes e.$
The definition of $\rho$ and the compatibility relation
\ref{id5.1} imply
\begin{eqnarray}\label{id5.2}
L(\rho\boxtimes\Id_Q)=m=
R(\Id_Q\boxtimes\rho).
\end{eqnarray}

Indeed, the first equality is exactly the adjunction defining
$\rho$.  For the second equality, we have
\begin{eqnarray*}
&&e\Bigl(\bigl(r(\Id_B\boxtimes\rho)\bigr)\boxtimes\Id_W\Bigr)
=e(r\boxtimes\Id_W)(\Id_B\boxtimes\rho\boxtimes\Id_W)\\
&=&e(\Id_B\boxtimes a)(\Id_B\boxtimes\rho\boxtimes\Id_W)
=e\Bigl(\Id_B\boxtimes a(\rho\boxtimes\Id_W)\Bigr)\\
&=&e(\Id_B\boxtimes\Id_W\boxtimes e)=
e\Bigl((e\boxtimes\Id_B)\boxtimes\Id_W\Bigr).
\end{eqnarray*}
Hence, by the injectivity of
\[
\Hom(B\boxtimes W\boxtimes B,B)
\longrightarrow
\Hom(B\boxtimes W\boxtimes B\boxtimes W,\mathbf1),
\qquad
g\longmapsto e(g\boxtimes\Id_W),
\]
we have
$
r(\Id_B\boxtimes\rho)=e\boxtimes\Id_B.
$
Therefore, we get
\begin{eqnarray*}
R(\Id_Q\boxtimes\rho)=
(\Id_W\boxtimes r)
(\Id_W\boxtimes\Id_B\boxtimes\rho)
=
\Id_W\boxtimes
r(\Id_B\boxtimes\rho)
=
\Id_W\boxtimes e\boxtimes\Id_B
=
m.
\end{eqnarray*}

(ii) Let $k:N\hookrightarrow Q$
be the kernel of $\rho$. Factor $\rho$ through its image, we have
\[
Q\xrightarrow{\,p\,}I\xrightarrow{\,j\,}A,
\qquad
\rho=jp,
\]
where $p:Q\twoheadrightarrow I$ is epic and $j:I\hookrightarrow A$ is monic. Define $F=L(\Id_A\boxtimes k):A\boxtimes N\longrightarrow Q.$
Using identity (\ref{id5.2}), we obtain
\[
\begin{aligned}
F(\rho\boxtimes\Id_N)
=
L(\rho\boxtimes k)
=
m(\Id_Q\boxtimes k)
=
R(\Id_Q\boxtimes\rho k)=0.
\end{aligned}
\]
Substituting $\rho=jp$ gives $F(j\boxtimes\Id_N)
(p\boxtimes\Id_N)=0.$ Thus $F(j\boxtimes\Id_N)=0$ since $p$ is epic.

If $u$ factors through the image
$I=\im\rho$, there exists $v:\mathbf1\longrightarrow I$
such that $jv=u.$
Using the unit law for the left action $L$, we obtain
\[
\begin{aligned}
k=
L(u\boxtimes k)
=
L(jv\boxtimes k)
=
F(j\boxtimes\Id_N)
(v\boxtimes \Id_N)
=0.
\end{aligned}
\]
Hence $N=0$ and $\ker\rho=0.$ This implies $\rho$ is monic and $p$ is isomorphic.

(iii) Now, define $i=p^{-1}v:\mathbf1\longrightarrow Q=W\boxtimes B.$
Then, we have
$\rho i=jp\,p^{-1}v=jv=u.$
By identity (\ref{id5.1}), this implies
\[
(\Id_W\boxtimes e)
(i\boxtimes\Id_W)
=a(\rho\boxtimes\Id_W)(i\boxtimes\Id_W)=a(u\boxtimes\Id_W)=
\Id_W.
\]

Similarly, define $s=
(e\boxtimes\Id_B)
(\Id_B\boxtimes i):
B\longrightarrow B.$ Using the interchange law together with above equation, we have
\[
\begin{aligned}
e(s\boxtimes\Id_W)
&=
e
\bigl(
(e\boxtimes\Id_B)
(\Id_B\boxtimes i)
\boxtimes\Id_W
\bigr)
\\
&=
e
\bigl(
\Id_B\boxtimes
(\Id_W\boxtimes e)
(i\boxtimes\Id_W)
\bigr)
\\
&=
e(\Id_B\boxtimes\Id_W)=e.
\end{aligned}
\]
Identity (\ref{id5.1}) gives $\Hom_{\mathcal C}(B,B)
\cong
\Hom_{\mathcal C}
(B\boxtimes W,\mathbf1),$
which sends a morphism $g:B\to B$ to $e(g\boxtimes\Id_W).$
This implies $s=\Id_B.$
Thus, we get
\[
(e\boxtimes\Id_B)
(\Id_B\boxtimes i)
=\Id_B.
\]

(iv) Now, equations (\ref{8.2}) are exactly the two triangle
identities for left dual object. Hence $B$ is a left dual of $W$.
If $\mathcal C$ is braided, the braiding converts a left dual into a
right dual as well.  Thus $W$ is rigid.  Thus we complete the proof.          $\hfill\Box$

\begin{remark}
We also call $\rho$ the rank-one morphism.
\end{remark}

\begin{remark}
The hypothesis above concerns the image of a canonical map. It does not require a lift of $u$ through any epimorphism, and does not require the unit to be projective.
\end{remark}

\begin{proposition}\label{Prop3.11}
Let $\mathcal{C}$ be an abelian braided Grothendieck--Verdier
category with simple tensor unit $\mathbf{1}$ and right-exact tensor
product. Assume that the duality
$D:\mathcal{C}^{\mathrm{op}}\longrightarrow\mathcal{C}$
is exact and satisfies
$D^2\cong \Id_{\mathcal{C}}.$ Let $K$ be the dualizing object.  We identify $K\simeq D\mathbf{1}$
via the canonical isomorphism induced by the
Grothendieck--Verdier adjunction.
Then $\mathcal C$ admits internal Homs.  More precisely, for any
$X,Z\in\mathcal C$, the object
\[
[X,Z]=D(X\boxtimes DZ).
\]
\end{proposition}
\proof Fix $X,Z\in\mathcal C$ and set
\[
H_{X,Z}=D(X\boxtimes DZ).
\]
Using Grothendieck-Verdier duality and Theorem \ref{thm3.4}, we obtain
\begin{eqnarray*}
\Hom_{\mathcal C}
\bigl(W,D(X\boxtimes DZ)\bigr)
&\cong&
\Hom_{\mathcal C}
\bigl(W\boxtimes(X\boxtimes DZ),K\bigr)\\
&\cong&
\Hom_{\mathcal C}\bigl((W\boxtimes X)\boxtimes DZ,K\bigr)\\
&\cong&
\Hom_{\mathcal C}\bigl(W\boxtimes X, D^2Z\bigr)\\
&\cong&
\Hom_{\mathcal C}\bigl(W\boxtimes X, Z\bigr).
\end{eqnarray*}
All the isomorphisms above are natural in $W$.  Hence the object $D(X\boxtimes DZ)$ represents the functor
\[
W\longmapsto
\Hom_{\mathcal C}(W\boxtimes X,Z).
\]
This means
\[
[X,Z]\simeq D(X\boxtimes DZ).
\]
This completes the proof.              $\hfill\Box$

\begin{remark}\label{rem:internal-hom-criterion-applies}
The proposition verifies the internal-Hom hypothesis used in the
dualizability criterion established earlier in this section.
Indeed, for every $W\in\mathcal C$,
\[
[W,V]\simeq D(W\boxtimes DV)
\]
and
\[
[W,W]\simeq D(W\boxtimes DW)
\]
exist as objects of $\mathcal C$.  Thus all the morphisms appearing in the dualizability criterion are well-defined in the present module category.
\end{remark}

\begin{proposition}\label{prop5.3}
Let $\mathcal{C}$ be an abelian braided Grothendieck--Verdier
category with simple tensor unit $\mathbf{1}$ and right-exact tensor
product. Assume that the duality
$D:\mathcal{C}^{\mathrm{op}}\longrightarrow\mathcal{C}$
is exact and satisfies
$D^2\cong \Id_{\mathcal{C}}.$ Let $K$ be the dualizing object.  We identify $K\simeq D\mathbf{1}$
via the canonical isomorphism induced by the
Grothendieck--Verdier adjunction.
If $K\boxtimes K\neq 0,$ then $K$ is invertible.
\end{proposition}
\proof (i) From Proposition \ref{Prop3.11}, we know $\mathcal C$ admits internal Homs, and for $X,Z\in\mathcal C$, we have
\begin{eqnarray}\label{id5.4}
[X,Z]\simeq D(X\boxtimes DZ).
\end{eqnarray}
This implies
\begin{eqnarray}\label{id5.5}
[K,K]
\simeq
D(K\boxtimes DK)
\simeq
D(K\boxtimes\mathbf{1})
\simeq
DK
\simeq\mathbf{1}.
\end{eqnarray}

Let $u_K:\mathbf{1}\longrightarrow[K,K]$ be the internal unit.
Since it is the internal-Hom adjunction of $\Id_K$, it is nonzero. By equation (\ref{id5.5}), we know $u_K$ is an isomorphism.

(ii) Let $B=[K,\mathbf{1}].$ Using identity (\ref{id5.4}), we have
$$
B=
[K,\mathbf{1}]\simeq
D(K\boxtimes D\mathbf{1})=
D(K\boxtimes K).
$$
Since $K\boxtimes K\neq0,$ and $D^2\cong\Id_{\mathcal{C}},$ we have $B\neq0.$

Now by definition and identity (\ref{id5.1}), there is a nonzero morphism
\[
e:B\boxtimes K\longrightarrow\mathbf{1},
\]which is the adjunction of $\Id_B$.
Since $\mathbf{1}$ is simple, we have $e$
is an epimorphism. By right exactness of tensor product, we have another nonzero epimorphism
\[
\Id_K\boxtimes e:
K\boxtimes B\boxtimes K
\longrightarrow
K\boxtimes\mathbf{1}
\simeq K.
\]

Define $\rho_K:K\boxtimes B\longrightarrow[K,K]$
to be the morphism adjoint to $\Id_K\boxtimes e:
K\boxtimes B\boxtimes K\longrightarrow K$ under (\ref{id5.1}), it is nonzero.
Let $a:[K,K]\boxtimes K\longrightarrow K$
denote the canonical evaluation morphism, by Lemma \ref{lm5.1}, we have
\[
a(\rho_K\boxtimes\Id_K)
=
\Id_K\boxtimes e.
\]
By identity (\ref{id5.5}) and assumptions, we have $\rho_K$
is an epimorphism and $\im \rho_K=[K,K].$

From above discussion, we can see $u_K$ factors through $\im\rho_K$, by Lemma \ref{lm5.1}, we have $B$ is a left dual of $K$, and $e$ is the evaluation. Following the proof of Lemma \ref{lm5.1}, we have $\rho_K$ is monic. Hence $\rho_K$ is isomorphic since $\mathcal C$ is abelian. This proves
\[
K\boxtimes B\simeq\mathbf{1}.
\]
Since $\mathcal{C}$ is braided, using braiding isomorphism, we also obtain
\[
B\boxtimes K\simeq\mathbf{1}.
\]
Thus $K$ is invertible and the inverse object is $B.$      $\hfill\Box$

\begin{corollary}\label{co5.4}
Let $\mathcal{C}$ be an abelian braided Grothendieck--Verdier
category with simple tensor unit $\mathbf{1}$ and right-exact tensor
product. Assume that the duality
$D:\mathcal{C}^{\mathrm{op}}\longrightarrow\mathcal{C}$
is exact and satisfies
$D^2\cong \Id_{\mathcal{C}}.$ Let $K$ be the dualizing object.  We identify $K\simeq D\mathbf{1}$
via the canonical isomorphism induced by the
Grothendieck--Verdier adjunction. If there is an object $B$ such that $K\boxtimes B\rightarrow \mathbf 1$ is nonzero, then $K$ is invertible.
\end{corollary}
\proof Since $\mathcal C$ admits internal Homs, using identity (\ref{id5.1}), there is a nonzero morphism $B\rightarrow[K,\mathbf 1]=D(K\boxtimes K)$. Thus $K\boxtimes K\neq0$ and $K$ is invertible.       $\hfill\Box$

\subsection{Category $\mathcal C$} \label{sec3.2}
Let $V$ be a simple, $C_2$-cofinite and lower truncated vertex operator algebra, and $\mathcal C$ the category of grading-restricted generalized $V$-modules.

Since $V$ is $C_2$-cofinite, it is also $C_{s+2}$-cofinite, where $-s$ is the lowest weight of $V$. Thus there is a homogeneous subspace $A$ such that $V=A+C_{s+2}(V).$ Let $W$ be a weak $V$-module generated by $w$. By lower-truncation property and finite-dimensional property of $A$, there is an integer $N\in\mathbb{N}$ such that $a^j_mw=0$ for $m\geq N,~j\in J,$ where $A=\Span\{a^j,j\in J\}.$

\begin{lemma}\label{lm3.1}
Suppose $V$ is $C_2$-cofinite and $W$ is a weak $V$-module generated by $w\in W.$ Then, $W$ is spanned by $$a^1_{i_1}\cdots a^k_{i_k}w,\qquad a^j\in A,\qquad k\geq0, \qquad i_1<\cdots<i_k<N.$$

Let $\de(a^1_{i_1}\cdots a^k_{i_k}w)=\sum_{j=1}^k(\wt a^j-i_j-1),$ then $$W(d)=\{a^1_{i_1}\cdots a^k_{i_k}w\mid \de(a^1_{i_1}\cdots a^k_{i_k}w)=d\}$$ is a finite-dimensional subspace of $W$ such that $W(d)=0$ for $d<<0.$
\end{lemma}
\proof First define $p(u)=\wt u+s$ for homogeneous $u\in V$, we have $p(u)\geq0.$ Let $F_P(W)=\{a^1_{i_1}\cdots a^k_{i_k}w\mid \sum p(a^j)\leq P\}.$ Then we have a filtration $$0=F_{-1}(W)\subseteq F_0(W)\subseteq\cdots\subseteq F_P(W)\subseteq\cdots,$$ and $W=\cup F_P(W).$ Since $W$ is generated by $w$, we know $W$ is spanned by
\begin{eqnarray}\label{id3.1}
u^1_{i_1}\cdots u^k_{i_k}w,\qquad i_j\in\mathbb{Z},
\end{eqnarray}
with $u^j\in V$ homogeneous.

(i) We prove $u^j$ in elements (\ref{id3.1}) can be chosen from $A$.

Induction on $P$, if $P=0$, then we have $p(u^j)=0$. This implies $u^j\in V_{-s}.$ Thus $u^j\in A.$ Suppose it is true for $F_{P-1}(W).$ Now let $u^1_{i_1}\cdots u^k_{i_k}w\in F_P(W)$ and $u^j\notin A$ for some $u^j$, then there are $x^i,y^i$ such that $u^j=\sum_ix^i_{-s-2}y^i.$ From Jacobi identity, we have
$$(x^i_{-s-2}y^i)_{i_j}=\sum_{t\geq0}(-1)^t\binom{-s-2}{t}\left(x^i_{-s-2-t}y^i_{i_j+t}-(-1)^{-s-2}y^i_{-s-2+i_j-t}x^i_t\right).$$
Since $p(x^i)+p(y^i)<p(x^i)+p(y^i)+1=p(x^i_{-s-2}y^i),$ this implies $u^1_{i_1}\cdots u^k_{i_k}w$ can be written as a linear combination of elements of $F_{P-1}(W).$ Hence by induction, we prove our assertion.

(ii) We prove $i_1\leq i_2\leq\cdots\leq i_k<N.$

 From Jacobi identity, we have $x_py_q=y_qx_p+\sum_{t\geq0}\binom{p}{t}(x_ty)_{p+q-t},$ and $p(x_ty)=p(a)+p(b)-(s+t+1).$ Now induction on $P$, it is trivial to prove $i_1\leq i_2\leq\cdots\leq i_k.$ Furthermore, from the finite-dimensional property of $A$, we get $i_k<N.$

(iii) We prove that the subscripts in (\ref{id3.1}) can be chosen unequal.

Prove by contradiction. Let $P$ be the smallest integer such that there exists an element of $F_P(W)$ which cannot be written in the desired form. Induction on the length $k$. $k=0$ is trivial, suppose the assertion is true for $k-1$. Now consider the length is equal to $k$. Assume $x_iy_i$ appears. If $i<0$, we have $$x_iy_i=(x_{-1}y)_{2i+1}-y_{3i+1}x_{-i-1}-\sum_{j\neq -i-1}x_{-1-j}y_{2i+1+j}+y_{2i-j}x_j.$$
If $i\geq0$, we have $$x_iy_i=(y_{-1}x)_{2i+1}-y_{-1-i}x_{3i+1}-\sum_{j\neq i}y_{-1-j}x_{2i+1+j}+x_{2i-j}y_j.$$
Since $p(x)+p(y)>p(x_{-1}y)=p(y_{-1}x)$ for $s>0$, the first terms in the right hand side of both identities belong to $F_{P-1}(W)$. For $s=0$, the first terms in the right hand side of both two identities have length $k-1$. Thus in both cases, we are done for the first terms in the right hand side of both identities.

For the rest terms, all of them have form $x_{i+j}y_{i-j}$ or $y_{i-j}x_{i+j}$ with $j\in \mathbb{Z}\backslash\{0\}$. For fixed $d$, there are only finitely many linearly independent elements in $F_P(W)$ with length $k$ since $i_j<N$ for all $j$ and $\dim A<\infty$. Thus we can define a total order on the subscript by lexicographic order for those elements. For $j>0$, we have $(i-j,i+j)<(i,i)$. Using induction on the lexicographic order, we can see the terms of $x_{i-j}y_{i+j}$ and $y_{i-j}x_{i+j}$ can be written in desired forms. For $j<0$, we have $$x_{i-j}y_{i+j}=y_{i+j}x_{i-j}+\sum_{t\geq0}\binom{i-j}{t}(x_ty)_{2i-t}.$$
The first term becomes $(i+j,i-j)<(i,i)$, and the second terms have length $k-1$. Using the same discussion for $y_{i-j}x_{i+j}$, we can see that in any case, we can rewrite $x_iy_i$ to be a combination of elements of the desired forms.

(iv) For $a^1_{i_1}\cdots a^k_{i_k}w$, we have $$\de(a^1_{i_1}\cdots a^k_{i_k}w)=\sum (\wt a^j-i_j-1)\geq-k(s+1)-\frac{k(2N-k+1)}{2}.$$
Since $k\geq0$, we have $W(d)=0$ for $d<<0.$ Since $\dim A<\infty$ and $i_j<N$ for all $j$, thus for fixed $k$, there are finitely linearly independent elements satisfying $\de(a^1_{i_1}\cdots a^k_{i_k}w)=d$. Now by above identity, $k$ cannot be infinite. Hence $\dim W(d)<\infty$.

This completes the proof.                               $\hfill\Box$

\begin{proposition}\label{prop3.2}
Every finitely generated weak $V$-module is a $C_m$-cofinite grading-restricted generalized $V$-module for $m\geq1$. Its generalized weights belong to $\cup_{i=1}^k(\lambda_i+\Z).$ Furthermore, there is $N_1\in\mathbb{N},$ such that$$(L(0)-\lambda_i-n)^{N_1}w=0$$ for all $w\in W_{[\lambda_i+n]}$ and $\lambda_i,n\in\Z.$
\end{proposition}
\proof It is enough to prove the case $W$ is generated by one element $w$. By Lemma \ref{lm3.1}, we have $W(d)$ is finite dimensional and lower truncated. It is obvious that $W(d)$ is stable under $L(0)$. In particular, $\span\{L(0)^jw,j\geq0\}$ is finite dimensional and $L(0)$-stable. Let $w=\sum_{\lambda_i}w_{\lambda_i}$ be the generalized-eigenvector decomposition of $w$. It is a finite sum and there is $N_1\in \mathbb{N}$ such that $(L(0)-\lambda_i)^{N_1}w_{\lambda_i}=0$ for all $\lambda_i.$

Now let $W[\lambda_i]=\oplus_{n\in\Z} W_{[\lambda_i+n]}$. Using $[L(0),u_n]=(\wt u-n-1)u_n$ for all $u\in V,n\in\Z$, we can see that $W[\lambda_i]$ is lower truncated, $\dim W_{[\lambda_i+n]}<\infty$ for all $n\in\Z,$ and $(L(0)-\lambda_i-n)^{N_1}w_{\lambda_i+n}=0$ for all $w_{\lambda_i+n}\in W_{[\lambda_i+n]}.$ Thus $W[\lambda_i]$ is grading-restricted generalized $V$-module which satisfies the uniformly nilpotent condition. Now the same statement is also true for $W$ since $W=\sum_iW[\lambda_i].$

Let $m>1,$ in Lemma \ref{lm3.1}, if $i_1\leq-m,$ then the element belongs to $C_m(W)$. Thus modulo $C_m(W)$, we have $-m<i_1<\cdots i_k<N.$ This means $W$ is $C_m$-cofinite. Now since $$u_{-s-2}w=\frac{1}{(s+1)!}(L(-1)^{s+1}u)_{-1}w,$$ we have $C_{s+2}(W)\subset C_1^+(W)$. This implies $W$ is also $C_1$-cofinite. and we complete the proof.             $\hfill\Box$

Now we recall several results about Zhu algebras \cite{DLM-Assoc,DLM2,DLM3,Zhu}. Recall $A_n(V)=V/O_n(V)$, where $$O_n(V)=\{(L(-1)+L(0))u,u\circ_nv=\Res_z\frac{(1+z)^{\wt u+n}}{z^{2n+2}}Y(u,z)v\mid u,v\in V\}.$$Then $A_n(V)$ is an associative algebra with identity $[\textbf{1}]$ and central element $[\omega]$, and $$u\ast_nv=\sum_{j=0}^n(-1)^j\binom{n+j}{n}\Res_z\frac{(1+z)^{\wt u+n}}{z^{n+j+1}}Y(u,z)v.$$ For $n=0$, it is the usual Zhu algebra $A(V)$. The lowest space of an irreducible ordinary $V$-module is a simple $A(V)$-module. This gives a bijection between irreducible ordinary $V$-modules and simple $A(V)$-modules.

\begin{lemma}\label{lm3.3}
Suppose $V$ is $C_2$-cofinite. Then $V$ is $C_1$-cofinite and $A_n(V)$ are finite-dimensional. More precisely, we have $$C_{2s+3}(V)\subset C_1^a(V),\qquad \dim A_n(V)\leq\dim V/C_{2n+2}(V).$$
\end{lemma}
\proof Using $L(-1)$-derivation property, we have $$(2s+1)!u_{-2s-3}v=(L(-1)^{2s+2}u)_{-1}v\equiv(-1)^{s+1}(L(-1)^{s+1}u)_{-1}(L(-1)^{s+1}v) ~~\mod L(-1)V.$$
From lower truncated properties of $V$, we can see that $\wt L(-1)^{s+1}u>0$ for all $u\in V$. This proves the first relation.

Since $C_{2n+2}(V)$ is homogeneous, let $X$ be its homogeneous complement in $V$. From definition, we have $$u_{-2n-2}v\equiv-\sum _{j>0}\binom{\wt u+n}{j}u_{j-2n-2}v~~\mod O_n(V).$$
It is obvious that $\wt u^i_{j-2n-2}v^i<\wt u^i_{-2n-2}v^i$ for $j>0$.

Now we prove $V_n\subset X+O_n(V).$ This implies $\dim A_n(V)<\dim X=\dim V/C_{2n+2}(V)$.
Induction on $n$. If $n=-s$, for any $w\in V_{-s}$, let $w=x+\sum_iu^i_{-2n-2}v^i$ with $\wt u^i_{-2n-2}v^i=-s$. Hence $u^i_{-2n-2}v^i\in O_n(V)$ from above equation. Suppose the assertion is true for all $V_k$ with $k<n$. Now let $w\in V_n,$ and $w=x+\sum_iu^i_{-2n-2}v^i.$
Then we have $$w=x+\sum _{j>0}\binom{\wt u+n}{j}u_{j-2n-2}v~~\mod O_n(V).$$Since $\wt u^i_{j-2n-2}v^i<\wt u^i_{-2n-2}v^i$, by induction, we have $u^i_{j-2n-2}v^i=x^i+o_n^i$ with $x^i\in X,o^i_n\in O_n(V)$. Therefore, we get $w\in X+O_n(V).$ This completes the proof.         $\hfill\Box$

\begin{theorem}\label{thm3.4}
Let $V$ be $C_2$-cofinite and $\mathcal {C}$ the category of grading-restricted generalized $V$-modules. Then $\mathcal {C}$ is a finite abelian category, and the objects are exactly the finitely generated weak modules. Every object is $C_1$-cofinite, and $\mathcal {C}$ has a natural vertex tensor category structure.

Furthermore, let $D$ be the contragredient functor on modules. Then $\mathcal {C}$ is a ribbon Grothendieck-Verdier category with $$K=DV,\qquad \Hom_\mathcal {C}(X\boxtimes W,K)\cong\Hom_\mathcal {C}(X,DW),\qquad \theta_X=e^{2\pi iL_X(0)},$$and the tensor product is right exact. If $V$ is simple, so is $K.$
\end{theorem}
\proof (i) From Lemma \ref{lm3.3}, we know Zhu algebra $A_n(V)$ is finite dimensional. Hence it has only finitely many simple modules. From \cite{Zhu,DLM-Assoc}, we know $V$ has only finitely many irreducible weak modules.

Since an irreducible weak $V$-module $W$ can be generated by one element, from Proposition \ref{prop3.2}, we know $W$ is a grading-restricted generalized $V$-module and $W=\sum_i W[\lambda_i]$. It is obvious that each $W[\lambda_i]$ is a weak $V$-module. Thus there is only one $\lambda_i$, denoted by $\lambda.$

Let $N_1$ be defined in Proposition \ref{prop3.2}. If $N_1>1$, then any eigenvector of $L(0)$ can generate a weak $V$-submodule which contradicts the irreducibility of $W$. Thus $N_1=1$. This implies $W$ is ordinary.

Let $V=W^0,W^1,...,W^r$ be all the irreducible ordinary $V$-modules, and $h_0,h_1,...,h_r$ the conformal weights respectively.

Now let \(W\) be any weak \(V\)-module. Let \(0\neq w\in W_{[\lambda]}\) and $U=V\cdot w$. Then $U$ is grading-restricted generalized by Proposition \ref{prop3.2}. Let \(N\subset U\) be the maximal submodule among those not
containing \(w\). The existence of $N$ follows from Zorn's Lemma. Then \(U/N\) is irreducible and \(w+N\neq0\). Hence $U/N\cong W^i$ for some irreducible ordinary module \(W^i\) with lowest conformal weight \(h_i\). Since \(w+N\) still has generalized weight \(\lambda\), we obtain $\lambda\in h_i+ \mathbb{N}.$ Therefore all the generalized weights of $W$ are a subset of $\cup_ih_i+\Z.$

(ii) Let $B=\max \{\Re h_i\}$ and $\mathcal {E}_B(W)=\oplus_{\Re\lambda\leq B}W_{[\lambda]}$ for $W\in\mathcal {C}$. From lower truncated property and finite dimensional property of generalized weight space, it is obvious that $\dim \mathcal {E}_B(W)<\infty.$ From definition, we can see that $\mathcal {E}_B(W)\neq0.$ This implies $W$ can be generated by $\mathcal {E}_B(W).$ Therefore $W$ is finitely generated.

Let $0\rightarrow U\rightarrow W\rightarrow Q\rightarrow 0$ be a short exact sequence in $\mathcal {C}$. By $L(0)$ generalized eigenvalue property, above sequence induces a short exact sequence $0\rightarrow U_{[\lambda]}\rightarrow W_{[\lambda]}\rightarrow Q_{[\lambda]}\rightarrow 0$ for all $\lambda.$ This gives the short exact sequence $$0\rightarrow \mathcal {E}_B(U)\rightarrow \mathcal {E}_B(W)\rightarrow \mathcal {E}_B(Q)\rightarrow 0,$$this means $\mathcal {E}_B$ is an exact functor from $\mathcal {C}$ to the category of finite dimensional vector spaces.

Hence, for any descending chain of sub objects of $W$, we get a descending chain of subspaces of $\mathcal {E}_B(W)$. Since $\dim\mathcal {E}_B(W)<\infty$, the chain must have finite length. Thus $W$ has finite length.

If $f:U^1\rightarrow U^2$ is a morphism in $\mathcal {C}$, we get a linear map $\mathcal {E}_B(f):\mathcal {E}_B(U^1)\rightarrow \mathcal {E}_B(U^2)$. Since $\mathcal {E}_B$ is exact and $\dim\mathcal {E}_B(U^1)<\infty,\dim\mathcal {E}_B(U^2)<\infty$, we have $\Hom_\mathcal {C}(U^1,U^2)$ is finite dimensional.

From \cite{Huang-Projective}, using Lemma \ref{lm3.3} and Proposition \ref{prop3.2}, we can see that every simple object has a projective cover. This proves that $\mathcal {C}$ is a finite abelian category, every object is finitely generated and $C_1$-cofinite. Now using Proposition \ref{prop2.15}, we see that $\mathcal {C}$ is a vertex tensor category with $P(z)$-tensor product, thus it induces a braided monoidal structure on $\mathcal {C}$.

(iii) Let $W\in \mathcal {C}$, define $DW$ to be the contragredient module of $W$, see Definition \ref{def2.9}. Since $V$ is lower truncated, we can see that $L(1)$ is locally nilpotent. This implies $e^{zL(1)}v$ is well-defined.
Now it is trivial that $DW\in\mathcal {C}$ and $D^2=\Id_\mathcal {C}$. Thus $\mathcal {C}$ is closed under $D$, and $D:\mathcal {C}^{op}\rightarrow\mathcal {C}$ is a contravariant equivalence.

Now let $K=DV$, from the contragredient properties of intertwining operators, we have $$\Hom_\mathcal {C}(X\boxtimes W,K)\cong\Hom_\mathcal {C}(X\boxtimes W,DV)\cong\Hom_\mathcal {C}(X,DW).$$Thus $K=DV$ is a dualizing object and $D$ is its Grothendieck-Verdier duality.
Since $\theta_X=e^{2\pi iL_X(0)}$ and $L(0)$ acts on the contragredient module by transpose, we have
$$\theta_{DX}=e^{2\pi i L_{DX}(0)}=e^{2\pi i L_X(0)^{t}}=\left(e^{2\pi i L_X(0)}\right)^{t}= D\!\left(e^{2\pi i L_X(0)}\right)=D(\theta_X).$$Hence $\mathcal {C}$ is a ribbon Grothendieck-Verdier category.

Now from
 \begin{eqnarray*}
 &&\Hom_\mathcal {C}(X\boxtimes W,Z)\\
 &\cong&\Hom_\mathcal {C}(X\boxtimes W,D(DZ))\cong\Hom_\mathcal {C}((X\boxtimes W)\boxtimes DZ,K)\\
 &\cong& \Hom_\mathcal {C}(X\boxtimes (W\boxtimes DZ),K)\cong \Hom_\mathcal {C}(X,D(W\boxtimes DZ)),
 \end{eqnarray*}
we have $-\boxtimes W$ is the left adjoint of $D(W\boxtimes D(-))$. Thus $-\boxtimes W$ is exact. Now using the braiding isomorphism, we also get the exactness of $W\boxtimes-.$

(iv) If $V$ is simple, then $K=DV$ is also simple since $D$ is exact and $D^2=\Id_\mathcal {C}$.

This completes the proof.                        $\hfill\Box$

\begin{remark}
(i) In Theorem \ref{thm3.4}, we just prove the Grothendieck-Verdier duality, not the rigidity of $\mathcal {C}$.

(ii) In Theorem \ref{thm3.4}, we assume $V$ can have negative weights. If $V$ is nonnegative graded, several results have been proven in literature.

(iii) In Theorem \ref{thm3.4}, if $\mathcal {C}$ is semisimple, then $V$ is regular. Thus regularity is equivalent to $C_2$-cofiniteness and rationality without restriction of the nonnegative gradation of $V$.
\end{remark}

\begin{corollary}\label{co3.6}
Let $V$ be $C_2$-cofinite and $W$ an irreducible weak $V$-module such that it is lower truncated with integer conformal grading. Then $W$ is ordinary.
\end{corollary}
\proof Since $W$ is simple, we have $W=V\cdot w$. Thus from Proposition \ref{prop3.2}, $W$ is grading restricted. From the proof of Theorem \ref{thm3.4}, $L(0)$ is semisimple on $W$. Thus $W$ is admissible. From irreducibility, it is obvious ordinary.            $\hfill\Box$

\begin{remark}
The preceding results ensure that the module category $\mathcal C$ precisely
satisfies all the conditions considered at the beginning of this
section.  In particular, $\mathcal C$ is a finite abelian braided
monoidal category with right-exact tensor product which admits internal Homs, and the
contragredient functor $D$ gives an exact Grothendieck--Verdier
duality with dualizing object
\[
K=DV.
\]
When $V$ is simple, both the tensor unit $V$ and the dualizing object
$K$ are simple.

Consequently, the invertibility criterion established above applies
to $\mathcal C$.  In view of Corollary \ref{co5.4}, it will be enough in the
subsequent argument to construct an object $B\in\mathcal C$ together
with a nonzero morphism
\[
K\boxtimes B\longrightarrow V.
\]
Such a morphism will imply that $K$ is invertible.  Notice that
rigidity of $\mathcal C$ has not yet been established; the purpose of
the later genus-one arguments is precisely to supply the additional
information needed to pass from Grothendieck--Verdier duality to
rigidity.
\end{remark}

\section{Genus one foundations}\label{sec4}
This section develops the genus-one analytic tools needed for the
subsequent rigidity argument.  We first establish convergence,
analytic continuation, collision, and translation properties for
geometrically modified genus-one trace functions in the lower-truncated setting.
We then apply these results to suitable two-point functions and derive
the genus-one comparison identities relating the modular transformation of
the vacuum trace to trace functions associated with tensor-product
morphisms.  These identities provide the analytic bridge between
modular invariance and the categorical arguments used later.

\subsection{Genus one trace functions}\label{sec4.1}
 First, we work on the trace functions. Let $V$ be $C_2$-cofinite and $\mathcal {C}$ the category of grading restricted generalized $V$-modules. Set $q=e^{2\pi i\tau},~q_{z_i}=e^{2\pi iz_i}.$
\begin{theorem}\label{thm3.7}
Fix $n\in\{1,2\}.$ Let $W_1,\ldots,W_n\in\mathcal C,A_1,...,A_{n}\in\mathcal C$, and $\mathcal Y_i\in\binom{W_{i-1}}{A_{i}~~W_{i}}$. Set $W_0=W_n$. For $z_1,\ldots,z_n\in\mathbb C,\tau\in\mathbb H,$ define the corresponding genus-one trace function by
\[
\Phi(z_1,\ldots,z_n;\tau)
=
\tr_{W_0}
\left(
\mathcal Y_1(\mathcal {U}(q_{z_1})a_1,z_1)
\cdots
\mathcal Y_n(\mathcal {U}(q_{z_n})a_n,z_n)
q^{L(0)-c/24}
\right),
\]
where $a_i\in A_i$. Then the following properties hold:

(i) There exists an ordered genus-one product region $$\mathcal R=\{(z_1,...,z_n;\tau)\mid |q_{z_1}|>...>|q_{z_n}|>|q|,~z_i-z_j\notin \mathbb{Z}+\mathbb{Z}\tau\}\subset\mathbb C^n\times\mathbb H,$$ such that $\Phi(z_1,\ldots,z_n;\tau)$ converges absolutely and locally uniformly on \(\mathcal R\).

(ii) Let the collision divisor be $D_\tau=\bigcup_{i<j}\{(z_1,...,z_n;\tau)\mid z_i-z_j\in\mathbb Z+\mathbb Z\tau\}.$ Then \(\Phi\) admits analytic continuation to a multivalued analytic function on $(\mathbb C^n\times\mathbb H)\setminus D_\tau .$ For every simply connected domain $U\subset(\mathbb C^n\times\mathbb H)\setminus D_\tau ,$
a choice of analytic continuation determines a single-valued
holomorphic branch $\Phi_U:U\longrightarrow\mathbb C .$

Analytic continuation around closed loops defines the local
monodromy action on the space of branches.

(iii) Fix a collision variable $z=z_i-z_j$ and let
$\rho_\tau=\min\{|a+b\tau|\mid(a,b)\in\mathbb Z^2\setminus\{(0,0)\}\}.$
Then on the punctured disk $0<|z|<\rho_\tau,$ the analytically continued trace admits a convergent expansion
\[
\Phi(z,\tau)
=
\sum_{\alpha=1}^{N}
\sum_{k=0}^{K_\alpha}
z^{\lambda_\alpha}
(\log z)^k
F_{\alpha,k}(z,\tau),
\]
where $\lambda_\alpha\in\mathbb C,$ and each coefficient $F_{\alpha,k}(z,\tau)$ is single-valued and holomorphic in \(z\) on the punctured disk. Moreover,
\[
F_{\alpha,k}(z,\tau)
=
\sum_{m\in\mathbb{Z}}
a_{\alpha,k,m}(\tau)z^m
\]
converges absolutely and locally uniformly on compact subannuli.

(iv) Now let $n=2$. From Theorem \ref{thm3.4}, let $\mathcal Y_3,\mathcal Y_4$ be determined by $(\mathcal Y_1,\mathcal Y_2)$, define
\begin{eqnarray*}
\Phi_{\mathrm{prod}}&=&\operatorname{tr}_{W_0}
\left(
\mathcal Y_1
\left(
\mathcal U(q_{z_1})a_1,q_{z_1}
\right)
\mathcal Y_2
\left(
\mathcal U(q_{z_2})a_2,q_{z_2}
\right)
q^{L(0)-c/24}
\right),
\\
\Phi_{\mathrm{iter}}&=&\operatorname{tr}_{W_0}
\left(
\mathcal Y_3
\left(
\mathcal U(q_{z_2})
\mathcal Y_4
\left(
a_1,e^{\log(z_1-z_2)}
\right)a_2,
q_{z_2}
\right)
q^{L(0)-c/24}
\right).
\end{eqnarray*}
Then on every nonempty common convergence domain, $$\Phi_{\mathrm{prod}}=\Phi_{\mathrm{iter}}.$$
Hence their analytic continuations coincide on every common branch.

(v) Choosing finitely many generating trace functions $\Phi_1,\ldots,\Phi_N,$ there exists a finite-dimensional vector-valued function $\boldsymbol{\Phi}(q,z)=(\Phi_1(q,z),\ldots,\Phi_N(q,z))^T$ satisfying a first-order differential system
\[
q\frac{\partial}{\partial q}\boldsymbol{\Phi}(q,z)
=
A(z,q)\boldsymbol{\Phi}(q,z),
\]
where $A(z,q)\in \mathbb{M}(N,\mathcal {O})$ is a matrix whose entries are generated by the elliptic functions and holomorphic near $q=0.$

Therefore \(q=0\) is a regular singular point and the solutions have
Frobenius-type expansions
\[
\boldsymbol{\Phi(q,z)}
=
\sum_{\alpha=1}^{N_0}
\sum_{k=0}^{K}
\sum_{m\ge0}
\boldsymbol{\phi}_{\alpha,k,m}(z)
q^{h_\alpha+m}
(\log q)^k ,
\]
where $h_\alpha\in\mathbb C,\boldsymbol{\phi}_{\alpha,k,m}(z)\in\mathbb C^N . .$
\end{theorem}
\proof We will prove this Theorem in Appendix \ref{seca}.     $\hfill\Box$

\begin{remark}
Consequently, genus-one trace functions in the category
\(\mathcal C\) possess the required convergence, analytic
continuation, collision expansion, and modular-parameter expansion
properties used in the later arguments.
\end{remark}

Now we prove several identities for the genus-one
trace functions that will be used in the later rigidity argument.

Let $V$ be a $C_2$-cofinite vertex operator algebra and
$\mathcal C$ the category of grading-restricted generalized
$V$-modules. We assume the analytic properties of genus-one
correlation functions established in Theorem \ref{thm3.7}. Recall the conventions we made in subsection \ref{subsec2.2}.

\begin{proposition}\label{prop3.9}
Let $R,W,X\in\mathcal C$ and $e:R\boxtimes W\longrightarrow V$ a morphism in $\mathcal C$. Define
\[
e_X:
R\boxtimes(W\boxtimes X)\longrightarrow X,\qquad e_X=l_X\circ(e\boxtimes\Id_X)\circ \alpha^{-1}_{R,W,X}.
\]
Let $\mathcal Y_{e_X}=e_X\mathcal {Y}^{R,W\boxtimes X}$
be the corresponding intertwining operator. For $r\in R$, $w\in W$, define
\[
\begin{aligned}
Z_{W,X}(r,w;z,\tau)
=
\tr_X
\Big(
&\mathcal Y_{e_X}(\mathcal {U}(q_z)r,q_z)
\mathcal {Y}^{W,X}(\mathcal {U}(1)w,1)
q^{L(0)-c/24}
\Big).
\end{aligned}
\]
Then, $Z_{W,X}$ converges absolutely in the region $-\operatorname{Im}\tau<\operatorname{Im}z<0,$ and admits analytic continuation away from the collision
divisors $z\in\mathbb Z+\mathbb Z\tau .$ Furthermore, in the collision region $z\rightarrow0$, we have
\[
Z_{W,X}(r,w;z,\tau)
=
F_X
\big(
e\mathcal {Y}^{R,W}(r,e^{\log z})w,\tau
\big),
\]
where the right hand side is interpreted coefficientwise in the
iterate expansion of the intertwining operator.
\end{proposition}
\proof The convergence and analytic continuation follow directly from the
genus-one convergence theorem for products and iterates of
intertwining operators. Indeed, $\mathcal Y_{e_X}$ and $Y^{W,X}$ are two intertwining operators whose
source and target objects belong to $\mathcal C$. Therefore Theorem
\ref{thm3.7} applies.

When $z$ approaches $0$, the two insertion points collide, using Theorem \ref{thm3.7} (iv), we have
\begin{eqnarray*}
&&
Z_{W,X}(r,w;z,\tau)
\\
&=&
\operatorname{tr}_{X}
\left(
\mathcal Y_{e_X}
\left(
\mathcal U(q_z)r,q_z
\right)
\mathcal Y^{W,X}
\left(
\mathcal U(1)w,1
\right)
q^{L_X(0)-c/24}
\right)
\\
&=&
\operatorname{tr}_{X}
\left(
\mathcal Y_{ e_X}
\left(
\mathcal U(1)
\mathcal Y^{R,W}
\left(
r,e^{\log z}
\right)w,
1
\right)
q^{L_X(0)-c/24}
\right)
\\
&=&
\operatorname{tr}_{X}
\left( e_X
\mathcal Y^{R\boxtimes W,X}
\left(
\mathcal U(1)
\mathcal Y^{R,W}
\left(
r,e^{\log z}
\right)w,
1
\right)
q^{L_X(0)-c/24}
\right)
\\
&=&
\operatorname{tr}_{X}
\left(
Y_X
\left(
\mathcal U(1)\,
e\,
\mathcal Y^{R,W}
\left(
r,e^{\log z}
\right)w,
1
\right)
q^{L_X(0)-c/24}
\right).
\end{eqnarray*}
This completes the proof.           $\hfill\Box$

\begin{proposition}\label{prop3.10}
Define
\begin{eqnarray*}
&&d:(W\boxtimes R)\boxtimes W\longrightarrow W,
\qquad
d=
r_W\circ
(\Id_W\boxtimes e)\circ
\alpha_{W,R,W},\\
&&d_X=(d\boxtimes\Id_X)\circ \alpha^{-1}_{W\boxtimes R,W,X},\qquad \mathcal Y_{d_X}=d_X\mathcal {Y}^{W\boxtimes R,W\boxtimes X}.
\end{eqnarray*}
Then the analytic continuation of the genus-one trace after moving
the first insertion around the $\tau$-cycle satisfies

\[
\begin{aligned}
&Z_{W,X}(r,w;-z-\tau,\tau)
=
\tr_{W\boxtimes X}
\left(
\mathcal Y_{d_X}
\big(
\mathcal {U}(1)\mathcal {Y}^{W,R}(w,e^{\log z})r,1
\big)
q^{L(0)-c/24}
\right).
\end{aligned}
\]
In particular, moving an insertion around the $\tau$-cycle changes
the trace space from $X$ to $W\boxtimes X$.
\end{proposition}
\proof Using $L(0)$-conjugation identity $q^{L(0)}\mathcal Y(u,z)q^{-L(0)}=\mathcal Y(q^{L(0)}u,qz)$, properties of trace $\tr_U(AB)=\tr_{U'}(BA)$, and genus-zero associativity identifies of  composition of intertwining operators with the morphism $\mathcal Y_{d_X}$, we have
\begin{equation*}
\begin{aligned}
&
Z_{W,X}(r,w;-z-\tau,\tau)
\\
&=
\tr_{X}
\left(
\mathcal Y_{e_X}
\left(
\mathcal U(q_{-z-\tau})r,
q_{-z-\tau}
\right)
\mathcal Y^{W,X}
\left(
\mathcal U(1)w,1
\right)
q^{L_X(0)-c/24}
\right)
\\
&=
\tr_{X}
\left(
(q^{-1}q_{-z})^{L_X(0)}
\mathcal Y_{e_X}
\left(
\mathcal U(1)r,1
\right)
(q^{-1}q_{-z})^{-L_{W\boxtimes X}(0)}
\mathcal Y^{W,X}
\left(
\mathcal U(1)w,1
\right)
q^{L_X(0)-c/24}
\right)
\\
&=
\tr_{W\boxtimes X}
\left(
(q^{-1}q_{-z})^{-L_{W\boxtimes X}(0)}
\mathcal Y^{W,X}
\left(
\mathcal U(1)w,1
\right)
q^{L_X(0)-c/24}
(q^{-1}q_{-z})^{L_X(0)}
\mathcal Y_{e_X}
\left(
\mathcal U(1)r,1
\right)
\right)
\\
&=
\tr_{W\boxtimes X}
\left(
q^{L_{W\boxtimes X}(0)}
q_z^{L_{W\boxtimes X}(0)}
\mathcal Y^{W,X}
\left(
\mathcal U(1)w,1
\right)
q_z^{-L_X(0)}
\mathcal Y_{e_X}
\left(
\mathcal U(1)r,1
\right)
q^{-c/24}
\right)
\\
&=
\tr_{W\boxtimes X}
\left(
q_z^{L_{W\boxtimes X}(0)}
\mathcal Y^{W,X}
\left(
\mathcal U(1)w,1
\right)
q_z^{-L_X(0)}
\mathcal Y_{e_X}
\left(
\mathcal U(1)r,1
\right)
q^{L_{W\boxtimes X}(0)-c/24}
\right)
\\
&=
\tr_{W\boxtimes X}
\left(
\mathcal Y^{W,X}
\left(
\mathcal U(q_z)w,q_z
\right)
\mathcal Y_{e_X}
\left(
\mathcal U(1)r,1
\right)
q^{L_{W\boxtimes X}(0)-c/24}
\right)
\\
&=
\tr_{W\boxtimes X}
\left(
d_X\,
\mathcal Y^{W\boxtimes R,W\boxtimes X}
\left(
\mathcal U(1)
\mathcal Y^{W,R}
\left(
w,e^{\log z}
\right)r,
1
\right)
q^{L_{W\boxtimes X}(0)-c/24}
\right)
\\
&=
\tr_{W\boxtimes X}
\left(
\mathcal Y_{d_X}
\left(
\mathcal U(1)
\mathcal Y^{W,R}
\left(
w,e^{\log z}
\right)r,
1
\right)
q^{L_{W\boxtimes X}(0)-c/24}
\right).
\end{aligned}
\end{equation*}
Therefore the analytically continued trace is precisely the stated
trace on $W\boxtimes X$.             $\hfill\Box$

\begin{remark}\label{re3.11}
Let $f_1:W_1\boxtimes M\longrightarrow X,
f_2:W_2\boxtimes X\longrightarrow M$
be morphisms in the module category.
We consider the genus-one two-point trace function with geometric
coordinate modifications:
\begin{eqnarray*}
&&Z(z,\tau)\\
&=&
\tr_{X}
\Big(
(f_1\circ
(\Id_{W_1}\boxtimes f_2)
\circ \alpha_{W_1,W_2,X})
\mathcal Y^{W_1,W_2}
\left(
\mathcal U(q_z)w_1,q_z
\right)
\mathcal Y^{W_2,X}
\left(
\mathcal U(1)w_2,1
\right)
q^{L(0)-c/24}
\Big).
\end{eqnarray*}
Consider the integer translation $z\longmapsto z-n,n\in\mathbb Z .$
This operation should be understood as the analytic continuation of
the multivalued genus-one correlation function along an integral
cycle of the torus.

Since $q_{z-n}=e^{2\pi i(z-n)}=q_z,\log q_{z-n}=\log q_z-2\pi in .$
Therefore, by the $L(0)$-conjugation property of intertwining
operators, $w_2$ becomes the factor $e^{2\pi inL(0)}w_2 .$
Furthermore, the analytic continuation of the genus-one correlation functions along the closed cycle produces a monodromy operator $\mathsf M$ on $W_2\boxtimes X,$ which is $\mathsf M^{W_2,X}=c_{X,W_2}c_{W_2,X}$ by \cite{HL4}. After $n$ translations the monodromy becomes $(\mathsf M^{W_2,X})^{\,n}.$

Hence the evaluation morphism is replaced by
\[
f^{(n)}
=
f_1
\circ
(\Id_{W_1}\boxtimes f_2)
\circ
(\Id_{W_1}\boxtimes (\mathsf M^{W_2,X})^{\,n})
\circ
a_{W_1,W_2,X}.
\]
Consequently, the analytically continued trace is
\[
\begin{aligned}
\tr_{X}
\Big(
(f^{(n)}
\mathcal Y^{W_1\boxtimes W_2,X})
\left(
\mathcal U(1)
\mathcal Y^{W_1,W_2}
(w_1,e^{\log z})
e^{2\pi inL(0)}w_2,
z
\right)
q^{L(0)-c/24}
\Big).
\end{aligned}
\]

The above integer translation formula is independent of the collision
expansion.
Indeed, when the two insertion points collide,
Theorem~\ref{thm3.7} (iv) is applied to the same genus-one two-point
trace function. It identifies the product expansion and the iterate
expansion of the correlation functions.
\end{remark}

\begin{lemma}\label{le3.12}
For every object $X\in\mathcal C$ and every $v\in V$, we have $F_X(L(-1)v,\tau)=0 $
\end{lemma}
\proof Since the geometric coordinate change satisfies
$\mathcal U(1)L(-1)=2\pi i(L(-1)+L(0))\mathcal U(1),$
 let $a\in V$ be homogeneous with $\wt a=k,$ we have $o((L(-1)+L(0))a)=0 .$
Hence

\[
F_X(L(-1)a,\tau)=0 .
\]
By linearity, we have $F_X(L(-1)v,\tau)=0$ for all $v\in V$.             $\hfill\Box$

\begin{remark}
The results above convert the analytic genus-one theory into explicit
trace manipulation rules. They are the tools used to compare
the modular $S$-transformation of one-point functions with traces obtained
from tensor products and evaluation morphisms.
\end{remark}

\subsection{Comparison lemma}\label{sec4.2}
Now let $V$ be a simple lower truncated $C_2$-cofinite vertex operator algebra, $\mathcal C$ the category of grading-restricted generalized $V$-modules. Let $D$ be the contragredient functor on $\mathcal C$ and the twist $\theta_X=e^{2\pi iL(0)}$. Let $W^0=V,W^1,...,W^r$ be all the inequivalent simple $V$-modules and $F_{W^i}$ the modified geometric trace functions. Suppose there exist constants $s_j\in\C$ such that
\begin{eqnarray}\label{id5.6}
SF_V=\sum_{j=0}^rs_jF_{W^j}.
\end{eqnarray}

Fix $W,R\in\mathcal C$ and a morphism $e:R\boxtimes W\rightarrow V.$ Let $Q=W\boxtimes R,$ and define
\begin{eqnarray}
\alpha&=&\Id_W\boxtimes e:Q\boxtimes W\rightarrow W,\label{id5.7}\\
t&=&ec_{W,R}(\theta_W\boxtimes\Id_R):Q\rightarrow V,\label{id5.8}\\
\beta_X&=&(e\boxtimes\Id_X)(\Id_R\boxtimes \mathsf M^{W,X})(c_{W,R}(\theta_W\boxtimes\Id_R)\boxtimes\Id_X):Q\boxtimes X\rightarrow X,\label{id5.9}
\end{eqnarray}
where $\mathsf M^{W,X}$ is the monodromy $c_{X,W}c_{W,X}$. For a morphism $b:P\boxtimes X\rightarrow X$, denote by $\mathcal Y_b=b\mathcal Y^{P,X}$ the corresponding intertwining operator, for $p\in P$, set
\begin{eqnarray}\label{G}
G_{b,X}(p,\tau)=\tr_X\mathcal Y_b(\mathcal U(1)p,1)q^{L(0)-\frac{c}{24}},
\end{eqnarray}
where $q=e^{2\pi i\tau}$. By Theorem \ref{thm3.7}, $G$ is convergent. Furthermore, $$SG(p,\tau)=G((-\tau^{-1})^{L(0)}p,-\tau^{-1}).$$

\begin{lemma}\label{Lm4.8}
With above notations, we have the following statements:

(i) For every $b\in Q$,
\begin{eqnarray}\label{ID4.5}
(SG_{\alpha,W})(b,\tau)
=
\sum_{j=0}^{r}
s_jG_{\beta_{W^j},W^j}(b,\tau).
\end{eqnarray}

(ii) Suppose that $W$ and $R$ are simple. Let $\theta_W=\zeta_W\Id_W,\theta_R=\zeta_R\Id_R$ with $\zeta_W,\zeta_R\in \C,$
and $\mu=\frac{\zeta_R}{\zeta_W}.$
For $X_j=W\boxtimes W^j$, define
\[
\alpha_j=
\alpha\boxtimes\Id_{W^j}:
Q\boxtimes X_j\longrightarrow X_j.
\]
Then, for every $b\in Q$, we have
\begin{eqnarray}\label{ID4.6}
\mu(SF_V)(t(b),\tau)
=
\sum_{j=0}^{r}s_jG_{\alpha_j,X_j}(b,\tau).
\end{eqnarray}
\end{lemma}
\proof Recall $q_z=e^{2\pi iz},$ using Proposition \ref{prop3.9}, consider the two-point function
\begin{eqnarray*}
Z_{W,X}(r,w;z,\tau)
&=&\tr_X\Bigl(\mathcal Y_{e_X}^{R,W\boxtimes X}(\mathcal U(q_z)r,q_z)\mathcal Y^{W,X}(\mathcal U(1)w,1)q^{L_X(0)-c/24}\Bigr)\\
&=&F_X\Bigl(e\mathcal Y^{R,W}(r,e^{\log z})w,\tau\Bigr).
\end{eqnarray*}
Using $L(0)$-conjugation and identity (\ref{id5.6}), we obtain
\begin{eqnarray}\label{ID4.7}
Z_{W,V}
\left(
(-\tau^{-1})^{L_R(0)}r,
(-\tau^{-1})^{L_W(0)}w;
-\frac{z}{\tau},
-\frac{1}{\tau}
\right)
=
\sum_j
s_j
Z_{W,W^j}(r,w;z,\tau).
\end{eqnarray}

(i) Fix $\tau$, suppose $z\rightarrow 0$ in the region $z\in P_\tau=\{a-1+b\tau\mid 0<a,b<1\},$ and consider
\begin{eqnarray}\label{ID4.8}
Z_{W,V}
\left(
(-\tau^{-1})^{L_R(0)}r,
(-\tau^{-1})^{L_W(0)}w;
\frac{z+1}{\tau},
-\frac{1}{\tau}
\right).
\end{eqnarray}
On one hand, using Proposition \ref{prop3.10}, the logarithmic branch
$\log\left(-\frac{z}{\tau}\right)-\log\left(-\frac{1}{\tau}\right)=\log z-2\pi i$ and $W\boxtimes V\cong W$, $L(0)$-conjugation, and $q'=e^{\frac{-2\pi i}{\tau}}$, $S$-transformation on $G$, we have
\begin{eqnarray}
&&
Z_{W,V}
\left(
(-\tau^{-1})^{L_R(0)}r,
(-\tau^{-1})^{L_W(0)}w;
\frac{z+1}{\tau},
-\frac{1}{\tau}
\right)
\nonumber\\
&=&
\tr_{W\boxtimes V}
\left(
\mathcal Y_\alpha
\left(
\mathcal U(1)
\mathcal Y^{W,R}
\left(
(-\tau^{-1})^{L_W(0)}w,
e^{\log(-z/\tau)}
\right)
(-\tau^{-1})^{L_R(0)}r,
1
\right)
q'^{\,L(0)-c/24}
\right)
\nonumber\\
&=&
G_{\alpha,W}
\left(
\mathcal Y^{W,R}
\left(
(-\tau^{-1})^{L_W(0)}w,
e^{\log(-z/\tau)}
\right)
(-\tau^{-1})^{L_R(0)}r,
-\frac{1}{\tau}
\right)
\nonumber\\
&=&
G_{\alpha,W}
\left(
\mathcal Y^{W,R}
\left(
(-\tau^{-1})^{L_W(0)}w,
(-\tau^{-1})\,e^{\log z-2\pi i}
\right)
(-\tau^{-1})^{L_R(0)}r,
-\frac{1}{\tau}
\right)
\nonumber\\
&=&
G_{\alpha,W}
\left(
(-\tau^{-1})^{L_Q(0)}
\mathcal Y^{W,R}
\left(
w,
e^{\log z-2\pi i}
\right)r,
-\frac{1}{\tau}
\right)\nonumber\\
&=&(SG_{\alpha,W})
\left(
\mathcal Y^{W,R}
(w,e^{\log z-2\pi i})r,
\tau
\right).\label{ID4.9}
\end{eqnarray}
On the other hand, define $\gamma=c_{W,R}(\theta_W\boxtimes\Id_R):
Q\longrightarrow R\boxtimes W,$ and $f_X^{(1)}$ by Remark \ref{re3.11} with $f_1=e,f_2=\Id_{W\boxtimes X}$ since $V\boxtimes X\cong X.$ Then, we have
\begin{eqnarray*}
f_X^{(1)}(\gamma\boxtimes\Id_X)=
(e\boxtimes\Id_X)
\bigl(\Id_R
\boxtimes
(c_{X,W}c_{W,X})
\bigr)
\bigl(
c_{W,R}
(\theta_W\boxtimes\Id_R)
\boxtimes\Id_X
\bigr)=
\beta_X.
\end{eqnarray*}
Using braiding isomorphism, Remark \ref{re3.11}, we obtain
\begin{eqnarray}\label{ID4.10}
&&
Z_{W,X}(r,w;-z-1,\tau)
\nonumber\\
&=&
\tr_{X}
\left[
\left(
f_X^{(1)}
\mathcal Y^{R\boxtimes W,X}
\right)
\left(
\mathcal U(1)
\mathcal Y^{R,W}
\left(
r,e^{\log(-z)}
\right)
\theta_W w,
1
\right)
q^{L_X(0)-c/24}
\right]
\nonumber\\
&=&
\tr_{X}
\left[
\left(
f_X^{(1)}
\mathcal Y^{R\boxtimes W,X}
\right)
\left(
\mathcal U(1)
\gamma
e^{-zL_Q(-1)}
\mathcal Y^{W,R}
\left(
w,e^{\log z-2\pi i}
\right)r,
1
\right)
q^{L_X(0)-c/24}
\right]
\nonumber\\
&=&
\tr_{X}
\left[
\left(
f_X^{(1)}
(\gamma\boxtimes\Id_X)
\mathcal Y^{Q,X}
\right)
\left(
\mathcal U(1)
e^{-zL_Q(-1)}
\mathcal Y^{W,R}
\left(
w,e^{\log z-2\pi i}
\right)r,
1
\right)
q^{L_X(0)-c/24}
\right]
\nonumber\\
&=&
\tr_{X}
\left[
\mathcal Y_{\beta_X}
\left(
\mathcal U(1)
e^{-zL_Q(-1)}
\mathcal Y^{W,R}
\left(
w,e^{\log z-2\pi i}
\right)r,
1
\right)
q^{L_X(0)-c/24}
\right]
\nonumber\\
&=&
G_{\beta_X,X}
\left(
e^{-zL_Q(-1)}\mathcal Y^{W,R}
\left(
w,e^{\log z-2\pi i}
\right)r,
\tau
\right).
\end{eqnarray}
For any $p\in Q$, since $\mathcal U(1)L_Q(-1)=2\pi i(L_Q(-1)+L_Q(0))\mathcal U(1)$, we have
\begin{eqnarray*}
\tr_{X}
\mathcal Y_{\beta_X}
\left(
\mathcal U(1)L_Q(-1)p,1
\right)
q^{L_X(0)-c/24}
=
2\pi i\,
\tr_{X}
\left[
L_X(0),
\mathcal Y_{\beta_X}
\left(
\mathcal U(1)p,1
\right)
q^{L_X(0)-c/24}
\right]
=
0.
\end{eqnarray*}
Thus, $G_{\beta_X,X}\left(L_Q(-1)^n p,\tau\right)=0$ for every $n\geq 1$.
Consequently, we have
\begin{eqnarray*}
G_{\beta_X,X}
\left(
e^{-zL_Q(-1)}p,\tau
\right)
=
\sum_{n\geq 0}
\frac{(-z)^n}{n!}
G_{\beta_X,X}
\left(
L_Q(-1)^n p,\tau
\right)
=
G_{\beta_X,X}(p,\tau).
\end{eqnarray*}
Putting those together, and let $p=\mathcal Y^{W,R}\left(w,e^{\log z-2\pi i}\right)r$, we have
\begin{eqnarray}\label{ID4.11}
Z_{W,X}(r,w;-z-1,\tau)
=
G_{\beta_X,X}
\left(
\mathcal Y^{W,R}
\left(
w,e^{\log z-2\pi i}
\right)r,
\tau
\right).
\end{eqnarray}
Thus, using identities (\ref{ID4.7}), (\ref{ID4.10}), (\ref{ID4.11}), we get
\begin{eqnarray}\label{ID4.12}
&&
Z_{W,V}
\left(
(-\tau^{-1})^{L_R(0)}r,
(-\tau^{-1})^{L_W(0)}w;
\frac{z+1}{\tau},
-\frac{1}{\tau}
\right)\nonumber\\
&=&
\sum_j s_j
G_{\beta_{W^j},W^j}
\left(
\mathcal Y^{W,R}
\left(
w,e^{\log z-2\pi i}
\right)r,
\tau
\right).
\end{eqnarray}
With identities (\ref{ID4.9}) and (\ref{ID4.12}), we eventually get
\begin{eqnarray*}
&&
(SG_{\alpha,W})
\left(
\mathcal Y^{W,R}
\left(
w,e^{\log z-2\pi i}
\right)r,
\tau
\right)
=
\sum_j s_j
G_{\beta_{W^j},W^j}
\left(
\mathcal Y^{W,R}
\left(
w,e^{\log z-2\pi i}
\right)r,
\tau
\right).
\end{eqnarray*}
Continuing once around $z=0$ replaces
$\log z-2\pi i$ by $\log z$ on both sides.
By coefficient uniqueness and the fact that the coefficients of
$\mathcal Y^{W,R}(w,z)r$ span
$Q=W\boxtimes R$, it follows that identity (\ref{ID4.5}) holds for all $b\in Q.$

(ii) If $W$ and $R$ are simple, there exist $\zeta_W,\zeta_R\in \C$, such that $\theta_W=\zeta_W\Id_W,
\theta_R=\zeta_R\Id_R.$
Using Proposition \ref{prop3.10}, consider
\begin{eqnarray}\label{ID4.13}
H(z,\tau)&=&
\sum_j
s_j
Z_{W,W^j}(r,w;-z-\tau,\tau)\nonumber\\
&=&\sum_j
s_j
G_{\alpha_j,X_j}
\left(
\mathcal Y^{W,R}
(w,e^{\log z})r,
\tau
\right).
\end{eqnarray}
On the other hand, from definitions and braiding-twist balancing relation, we first have
\begin{eqnarray}\label{ID4.14}
&&
t\,c_{R,W}
(\theta_R\boxtimes\Id_W)
\nonumber\\
&=&
e\,c_{W,R}
(\theta_W\boxtimes\Id_R)
c_{R,W}
(\theta_R\boxtimes\Id_W)
\nonumber\\
&=&
e\,c_{W,R}c_{R,W}
(\theta_R\boxtimes\theta_W)
=
e\,\theta_{R\boxtimes W}\nonumber\\
&=&
\theta_V e
=
e.
\end{eqnarray}
Now since $W,R$ are simple, using identity (\ref{ID4.7}), Remark \ref{re3.11} with $\mathsf M^{W,V}=\Id$, the chosen branch $\log(-z)-\pi i=\log z$, $L(0)$-conjugation, $S$-transformation, identity (\ref{ID4.14}), braiding isomorphism and Lemma \ref{le3.12}, we have

\begin{eqnarray}\label{ID4.15}
H(z,\tau)
&=&
Z_{W,V}
\left(
(-\tau^{-1})^{L_R(0)}r,
(-\tau^{-1})^{L_W(0)}w;
\frac{z}{\tau}+1,
-\frac{1}{\tau}
\right)
\nonumber\\
&=&
F_V
\left(
e\,
\mathcal Y^{R,W}
\left(
(-\tau^{-1})^{L_R(0)}r,
e^{\log(z/\tau)}
\right)
\theta_W^{-1}
(-\tau^{-1})^{L_W(0)}w,
-\frac{1}{\tau}
\right)
\nonumber\\
&=&
F_V
\left(
e\,
\mathcal Y^{R,W}
\left(
(-\tau^{-1})^{L_R(0)}r,
(-\tau^{-1})
e^{\log(-z)-2\pi i}
\right)
(-\tau^{-1})^{L_W(0)}
\theta_W^{-1}w,
-\frac{1}{\tau}
\right)
\nonumber\\
&=&
F_V
\left(
(-\tau^{-1})^{L_V(0)}
e\,
\mathcal Y^{R,W}
\left(
r,e^{\log(-z)-2\pi i}
\right)
\theta_W^{-1}w,
-\frac{1}{\tau}
\right)
\nonumber\\
&=&
(SF_V)
\left(
e\,
\mathcal Y^{R,W}
\left(
r,e^{\log(-z)-2\pi i}
\right)
\theta_W^{-1}w,
\tau
\right)
\nonumber\\
&=&
\mu\,
(SF_V)
\left(
t\,c_{R,W}
\mathcal Y^{R,W}
\left(
r,e^{\log(-z)-2\pi i}
\right)
w,
\tau
\right)
\nonumber\\
&=&
\mu\,
(SF_V)
\left(
e^{-zL_V(-1)}
t\,
\mathcal Y^{W,R}
\left(
w,e^{\log z}
\right)r,
\tau
\right)
\nonumber\\
&=&
\mu\,
(SF_V)
\left(
t\,
\mathcal Y^{W,R}
\left(
w,e^{\log z}
\right)r,
\tau
\right).
\end{eqnarray}
Since the coefficients of $\mathcal Y^{W,R}(w,z)r$ span $Q=W\boxtimes R$, by identities (\ref{ID4.13}) and (\ref{ID4.15}), for all $b\in Q$, we have
$$\mu\,
(SF_V)
\left(
t(b),
\tau
\right)=\sum_j
s_j
G_{\alpha_j,X_j}
\left(
b,
\tau
\right).$$
This proves (ii) and we complete the proof.       $\hfill\Box$

\section{Rigidity and semisimplicity of $\mathcal C$}\label{sec5}
In this section, $\mathcal C$ is the category of grading restricted generalized modules of a simple $C_2$-cofinite lower truncated vertex operator algebra. Assuming that the Grothendieck--Verdier dualizing object is
invertible, we first normalize the duality and combine the
internal-Hom criterion with the genus-one comparison identities to
prove that the module category $\mathcal C$ is rigid.  We then use
complete one-point functions and the resulting natural
endomorphisms of the identity functor to control extensions between
simple objects.  This yields the semisimplicity of $\mathcal C$.

\subsection{Rigidity}\label{sec5.1}
Suppose from now on that the Grothendieck--Verdier dualizing
object $K=DV$
is invertible.
\begin{definition}\label{def6.1}
For $X\in\mathcal C$, define \[
D_0X=K^{-1}\boxtimes DX.
\]
\end{definition}
Since tensoring with the invertible object $K^{-1}$ is an exact
equivalence and $D$ is an exact contravariant equivalence, $D_0$
is again an exact contravariant equivalence.

\begin{lemma}\label{lm6.2}
$D_0$ is another Grothendieck--Verdier duality whose
dualizing object is $V$, we call $D_0$ the normalized dual.
\end{lemma}
\proof For $X,Z\in\mathcal C$, tensoring
a morphism $X\boxtimes Z\to V$ with $K$, applying the original
Grothendieck--Verdier adjunction, and finally tensoring with
$K^{-1}$, we have the following natural isomorphisms
\begin{eqnarray}\label{id6.1}
\Hom_{\mathcal C}(X\boxtimes Z,V)
&\simeq&
\Hom_{\mathcal C}
(K\boxtimes X\boxtimes Z,K)
\nonumber\\
&\simeq&
\Hom_{\mathcal C}
(K\boxtimes X,DZ)
\nonumber\\
&\simeq&
\Hom_{\mathcal C}
(X,K^{-1}\boxtimes DZ)
\nonumber\\
&\simeq&
\Hom_{\mathcal C}(X,D_0Z).
\end{eqnarray}
This completes the proof.               $\hfill\Box$

For every $X\in\mathcal C$, let $e_X:D_0X\boxtimes X\longrightarrow V$
be the evaluation morphism adjunction to
$\Id_{D_0X}$ under
\eqref{id6.1}.  Note that this is an evaluation for
the normalized duality $D_0$, not yet a rigid evaluation.

Recall the original twist $\theta_X=e^{2\pi iL_X(0)}$,
and define
\begin{equation}\label{id6.2}
t_X
=
e_X
c_{X,D_0X}
(\theta_X\boxtimes\Id_{D_0X})
:
X\boxtimes D_0X\longrightarrow V.
\end{equation}
The braiding and the adjunction \eqref{id6.1} show that
$(X,t_X)$ represents the functor
\begin{eqnarray}\label{id6.3}
Z\longmapsto
\Hom_{\mathcal C}
(Z\boxtimes D_0X,V), \qquad \text{equivalently,}\qquad
\Hom_{\mathcal C}(Z,X)=\Hom_{\mathcal C}(Z\boxtimes D_0X,V).
\end{eqnarray}
On the other hand, applying the $D_0$-adjunction to
$Z\boxtimes D_0X$ gives a natural isomorphism
\[
\Hom_{\mathcal C}
(Z\boxtimes D_0X,V)
\simeq
\Hom_{\mathcal C}
(Z,D_0^2X).
\]
Hence, by the uniqueness of representing objects, there is a natural isomorphism $j_X:X\xrightarrow{\sim}D_0^2X$
characterized by
\begin{equation}\label{id6.4}
e_{D_0X}
(j_X\boxtimes\Id_{D_0X})
=
t_X.
\end{equation}

\begin{remark}\label{re6.3}
Note that we do not identify $j_X$ with the usual graded bidual
map $X\to D^2X$, we also do not assume any monoidality property of $j$.
\end{remark}

We now fix an object $W\in\mathcal C$ and set
\[
R=D_0W,
\qquad
Q=W\boxtimes R,
\qquad
e=e_W,
\qquad
t=t_W.
\]
Since $D_0$ has dualizing object $V$ and $\mathcal C$ admits internal Homs, we have $R=D_0W\simeq[W,V].$
Thus $Q=W\boxtimes[W,V]$
is precisely the object occurring in the dualizability criterion.

\begin{lemma}\label{lm6.4}
With above notation, we have $[W,W]\simeq D_0Q$.
\end{lemma}
\proof For every $X\in\mathcal C$, and a morphism $f:X\boxtimes W\longrightarrow W$, we get a morphism
\begin{eqnarray}\label{id6.5}
t(f\boxtimes\Id_R):
X\boxtimes W\boxtimes R\longrightarrow V.
\end{eqnarray}Thus we have a natural isomorphism
$\Hom_{\mathcal C}(X\boxtimes W,W)
\simeq\Hom_{\mathcal C}
(X\boxtimes W\boxtimes R,V).
$ Hence
\begin{eqnarray}\label{id0.6}
\Hom_{\mathcal C}(X\boxtimes W,W)
\simeq
\Hom_{\mathcal C}(X\boxtimes Q,V)
\simeq
\Hom_{\mathcal C}(X,D_0Q).
\end{eqnarray}
Consequently, $[W,W]\simeq D_0Q.$             $\hfill\Box$

Let $a:D_0Q\boxtimes W\longrightarrow W$
denote the internal evaluation under this identification.
Since the identity morphism of $D_0Q$ corresponds to
$e_Q:D_0Q\boxtimes Q\longrightarrow V$ under \eqref{id6.1},
using \eqref{id6.5}, the evaluation $a$ is
characterized by
\begin{equation}\label{id6.6}
t(a\boxtimes\Id_R)
=
e_Q.
\end{equation}

Since $[W,V]\simeq R,[W,W]\simeq D_0Q,$ the rank one morphism $\rho$ defined in Lemma \ref{lm5.1} becomes a map
\[
\Phi:Q=W\boxtimes R\longrightarrow D_0Q,
\]
which is the internal-Hom adjoint of
$
\alpha
=
\Id_W\boxtimes e
:
Q\boxtimes W\longrightarrow W,
$
equivalently, under the $D_0$-adjunction, $\Phi$ is characterized by
\begin{equation}\label{id6.7}
e_Q
(\Phi\boxtimes\Id_Q)
=
t
(\Id_W
 \boxtimes e
 \boxtimes\Id_R)
=\beta
:
Q\boxtimes Q\longrightarrow V.
\end{equation}
By the defining property of the internal evaluation $a$, the fact
that $\Phi$ is adjoint to $\alpha$ is equivalently expressed as
\begin{equation}\label{id6.8}
\alpha
=
a
(\Phi\boxtimes\Id_W).
\end{equation}

Finally, let $u:V\longrightarrow[W,W]\simeq D_0Q$
be the internal unit, namely the morphism adjoint to
$\Id_W$ under \eqref{id0.6}.  Let $u_0:V\xrightarrow{\sim}D_0V$
denote the canonical identification.  Under the first
isomorphism of \eqref{id0.6}, $\Id_W$ corresponds to $t:Q=W\boxtimes R\longrightarrow V.$
Under the $D_0$-adjunction, the morphism $t$ corresponds to
$D_0(t)u_0$.  Hence
\begin{equation}\label{id6.10}
u
=
D_0(t)u_0
:
V\longrightarrow D_0Q.
\end{equation}

We have therefore identified all the objects and morphisms entering
the dualizability criterion:
\[
[W,V]\simeq R,
\qquad
[W,W]\simeq D_0Q,
\qquad
\rho=\Phi,
\qquad
u=D_0(t)u_0.
\]
Consequently, in order to prove that $W$ is rigid, it remains to
show that the internal unit $u$ lies in the image of $\Phi$.

\begin{lemma}\label{lm6.5}
With above notations, we have $\Phi=D_0(\Phi)j_Q.$
Consequently, $\im D_0(\Phi)=\im\Phi$
which are sub objects of $D_0Q$.
\end{lemma}
\proof First, we know $D_0(\Phi):D_0^2Q\longrightarrow D_0Q.$
Thus, $D_0(\Phi)j_Q$ has the same source and target as $\Phi$.
We shall prove that these two morphisms have the same image under
the evaluation adjunction
$\Hom_{\mathcal C}(Q,D_0Q)
\simeq
\Hom_{\mathcal C}(Q\boxtimes Q,V),$
which sends a morphism $f:Q\to D_0Q$ to $e_Q(f\boxtimes\Id_Q).$

Using naturality of the evaluation defined by $D_0$, identities \eqref{id6.4} and \eqref{id6.2}, naturality of braiding, identity \eqref{id6.7}, $Q=W\boxtimes R$ and $R=D_0W$, hexagon axioms and braiding twist balancing identity, we have
\begin{eqnarray*}
&&e_Q\left((D_0(\Phi)j_Q)\boxtimes\Id_Q\right)
=e_{D_0Q}(j_Q\boxtimes\Phi)\\
&=&e_{D_0Q}(j_Q\boxtimes\Id_{D_0Q})(\Id_Q\boxtimes\Phi)
=t_Q(\Id_Q\boxtimes\Phi)\\
&=&e_Qc_{Q,D_0Q}(\theta_Q\boxtimes\Id_{D_0Q})(\Id_Q\boxtimes\Phi)
=e_Qc_{Q,D_0Q}(\theta_Q\boxtimes\Phi)\\
&=&e_Q(\Phi\boxtimes\theta_Q)c_{Q,Q}
=e_Q(\Phi\boxtimes\Id_Q)(\Id_Q\boxtimes\theta_Q)c_{Q,Q}\\
&=&e_Q(\Phi\boxtimes\Id_Q)c_{Q,Q}(\theta_Q\boxtimes\Id_Q)
=\beta c_{Q,Q}(\theta_Q\boxtimes\Id_Q)\\
&=&t(\Id_W\boxtimes e\boxtimes\Id_R) c_{Q,Q}(\theta_Q\boxtimes\Id_Q)\\
&=&t\left(\Id_W\boxtimes(t\,c_{R,W}(\theta_R\boxtimes\Id_W))\boxtimes\Id_R\right)\\
&=&t\left(\Id_W\boxtimes(e\,c_{W,R}(\theta_W\boxtimes\Id_R)
c_{R,W}(\theta_R\boxtimes\Id_W))\boxtimes\Id_R\right)\\
&=&t\left(\Id_W\boxtimes(e\,c_{W,R}c_{R,W}(\theta_R\boxtimes\theta_W))\boxtimes\Id_R\right)\\
&=&t\left(\Id_W\boxtimes(e\,\theta_{R\boxtimes W})\boxtimes\Id_R\right)=t\left(\Id_W\boxtimes e\boxtimes\Id_R\right)=\beta\\
&=&e_Q(\Phi\boxtimes\Id_Q).
\end{eqnarray*}
Hence, we get
\[
D_0(\Phi)j_Q
=
\Phi.
\]

Finally, since $j_Q$ is an isomorphism,
\begin{eqnarray*}
\im\Phi
=
\im\bigl(D_0(\Phi)j_Q\bigr)
=
\im D_0(\Phi).
\end{eqnarray*}
This proves the Lemma.             $\hfill\Box$

\begin{lemma}\cite{EtingofPenneys}\label{LM}
In category $\mathcal C$, rigid objects are closed under extensions.
\end{lemma}

\begin{proposition}\label{prop6.6}
Suppose that the Grothendieck--Verdier dualizing object $K=DV$
is invertible and the complete trace functions satisfy
\[
SF_V=\sum_{j=0}^r s_jF_{W^j}.
\]
Then the category $\mathcal C$ is rigid.
\end{proposition}
\proof (i) Let $W$ be simple and put
\[
R=D_0W,
\qquad
Q=W\boxtimes R,
\qquad
e=e_W,
\qquad
t=t_W.
\]
Since $D_0$ is a contravariant equivalence, $R$ is also simple.
Hence, by Schur's Lemma, the twists act by nonzero scalars:
$\theta_W=\zeta_W\Id_W,\theta_R=\zeta_R\Id_R.$
Set $\mu_W=\frac{\zeta_R}{\zeta_W}\neq0.$

Recall that $[W,W]\simeq D_0Q$
and let $a,\alpha,\Phi$ be the corresponding morphisms defined as above. Let $X_j=W\boxtimes W^j,$
define
\[
a_j^{\mathrm{act}}
=
a\boxtimes\Id_{W^j}
:
D_0Q\boxtimes X_j\longrightarrow X_j.
\]
Recall $\alpha_j=\alpha\boxtimes\Id_{W^j}$ defined in Lemma \ref{Lm4.8}, we have
\begin{eqnarray*}
\alpha_j
=
\left(
a(\Phi\boxtimes\Id_W)
\right)
\boxtimes\Id_{W^j}
=
(a\boxtimes\Id_{W^j})
(\Phi\boxtimes\Id_{X_j})
=
a_j^{\mathrm{act}}
(\Phi\boxtimes\Id_{X_j}).
\end{eqnarray*}
For $b\in Q$, using the naturality of tensor-product intertwining operators, we have
\begin{eqnarray*}
&&\mathcal Y_{\alpha_j}(b,z)\\
&=&
\alpha_j\mathcal Y^{Q,X_j}(b,z)
=
a_j^{\mathrm{act}}
(\Phi\boxtimes\operatorname{Id}_{X_j})
\mathcal Y^{Q,X_j}(b,z)\\
&=&
a_j^{\mathrm{act}}
\mathcal Y^{D_0Q,X_j}(\Phi(b),z)
=
\mathcal Y_{a_j^{\mathrm{act}}}(\Phi(b),z).
\end{eqnarray*}
Since $\Phi$ is a $V$-module morphism, we have $\Phi\,\mathcal U_Q(1)=
\mathcal U_{D_0Q}(1)\Phi.$ Consequently, we get
\[
G_{\alpha_j,X_j}(b,\tau)
=
\tr_{X_j}
\mathcal Y_{a_j^{\mathrm{act}}}
\left(
\mathcal U(1)\Phi(b),1
\right)
q^{L_{X_j}(0)-c/24}.
\]
Applying Lemma~\ref{Lm4.8} (ii), we therefore obtain,
for every $b\in Q$,
\begin{eqnarray}\label{ID5.11}
\mu_W(SF_V)(t(b),\tau)
=
\sum_j
s_j
\tr_{X_j}
\mathcal Y_{a_j^{\mathrm{act}}}
\left(
\mathcal U(1)\Phi(b),1
\right)
q^{L_{X_j}(0)-c/24}.
\end{eqnarray}

Let $N=\ker\Phi.$ If $b\in N$, then $\Phi(b)=0$, which means $(SF_V)(t(b),\tau)=0$ by identity \eqref{ID5.11}.
Equivalently, the complete one-point function $SF_V$ vanishes on
the submodule $t(N)\subseteq V.$
Since $V$ is simple, we have $t(N)=0$ or $t(N)=V.$

If $t(N)=V$, then $(SF_V)(v,\tau)=0$ for every $v\in V$. This is absurd since $SF_V(\mathbf 1,\tau)$
is a nonzero graded character and the $S$-transformation is an
invertible change of variables on complete one-point functions which means $SF_V(\mathbf 1,\tau)\neq0.$
Therefore $t(N)=0.$ Thus, we conclude that
\begin{equation}\label{ID5.12}
\ker\Phi\subseteq\ker t.
\end{equation}

Now, let $p:Q\twoheadrightarrow Q/\ker\Phi$
be the quotient map. From \eqref{ID5.12},
there is a unique morphism
\[
\overline t:Q/\ker\Phi\longrightarrow V
\]
such that $t=\overline t\,p.$
Hence, by contravariance of $D_0$, we have $D_0(t)=D_0(p)D_0(\overline t),$
and therefore
\begin{equation}\label{ID5.13}
\im D_0(t)
\subseteq
\im D_0(p).
\end{equation}

On the other hand, the image-coimage factorization of $\Phi$ has
the form
\[
Q
\xrightarrow{\,p\,}
Q/\ker\Phi
\xrightarrow[\sim]{\,\overline\Phi\,}
\im\Phi
\xrightarrow{\,i\,}
D_0Q.
\]
Using the exactness and contravariance, we have $\im D_0(\Phi)=\im D_0(p).$
Together with \eqref{ID5.13} and Lemma \ref{lm6.5}, we arrive at
\[
\im D_0(t)
\subseteq
\im D_0(\Phi)=\im \Phi.
\]
From identity \eqref{id6.10}, the internal unit $u=D_0(t)u_0$ and $u_0$ is the canonical identification, we have
\[
\im u
\subseteq
\im D_0(t)
\subseteq
\im\Phi.
\]
Thus the internal unit $u$ factors through $\im\Phi$, by Lemma \ref{lm5.1}, we have $R=D_0W\simeq[W,V]$
is a left dual of $W$. Since $\mathcal C$ is braided, left dual determines a right dual. Thus $W$ is rigid.
We have therefore proved that every simple object of $\mathcal C$
is rigid.

(ii) Since every object of $\mathcal C$ has finite length by Theorem \ref{thm3.4}. Induction
on composition length.  Objects of length one are simple and hence
rigid.  Suppose that every object of length strictly less than $n$
is rigid, and let $X$ have length $n$.  Choose a simple quotient
$S$ of $X$ and write
\[
0\longrightarrow Z
\longrightarrow X
\longrightarrow S
\longrightarrow0.
\]
Then
\[
\ell(Z)=n-1.
\]
By induction $Z$ is rigid, while $S$ is rigid since it is simple.
Therefore $X$ is rigid  by Lemma \ref{LM}. This proves every object of $\mathcal C$ is rigid, thus $\mathcal C$ is rigid and we complete the proof.       $\hfill\Box$

\subsection{Semisimplicity}\label{sec5.2}
Now we use the rigidity to prove the semisimplicity.
\begin{lemma}\label{LM5.8}
The complete one-point geometric modified trace functions $F_{W^0},...,F_{W^r}$ of pairwise inequivalent irreducible ordinary $V$-modules are linearly independent.
\end{lemma}
\proof Since the geometric modified operator $\mathcal U(1)$ is invertible, it is enough to show the ordinary complete one point trace functions are linearly independent. Proof by contradiction, suppose there exist $s_j\in\C$, such that
\begin{eqnarray}\label{ID5.14}
\sum_js_j\tr_{W^j}(v,\tau)=0
\end{eqnarray}
for all $v\in V,\tau\in\mathbb{H}$.

Let $h_j$ be the lowest conformal weight of $W^j$. Since $W^j$
is ordinary, we have $W^j=\bigoplus_{n\geq0}W^j(h_j+n)$ with $\dim W^j(h_j+n)<\infty$. Then the $j$-th trace function is
\begin{eqnarray*}
\tr_{W^j}
o_{W^j}(v)q^{L(0)-c/24}
=
\sum_{n\geq0}
\tr_{W^j(h_j+n)}
o_{W^j}(v)
q^{h_j+n-c/24}.
\end{eqnarray*}
Thus all exponents occurring in the $j$-th trace belong to $h_j-\frac{c}{24}+\mathbb Z_{\geq0}.$
Applying $T$-transformation to identity \eqref{ID5.14}, we get
\begin{eqnarray*}
0=\sum_js_j\tr_{W^j}(v,\tau+1)&=&\sum_js_j\left(\sum_{n\geq0}
\tr_{W^j(h_j+n)}
o_{W^j}(v)
e^{2\pi i(\tau+1)(h_j+n-c/24)}\right)\\
&=&\sum_js_je^{2\pi i(h_j-c/24)}\tr_{W^j}(v,\tau).
\end{eqnarray*}
Hence, applying $T$-transformation several times, using Vandermonde determinant, we can assume all $h_j$ lie in same coset modulo $\Z.$

Since $V$ has only finitely many inequivalent irreducible ordinary modules, let $$\lambda=\min\{h_j,j=0,...,r\},$$ we have $h_j-\lambda\in\mathbb Z_{\geq0}$
for all $j$.

Extracting the coefficient of $q^{\lambda-c/24}$
from identity \eqref{ID5.14}, comparing the coefficients of degree $q^{\lambda-c/24}$, we can assume all $h_j$ are equal.
Now, the top levels $W^j(\lambda)$ are pairwise inequivalent simple
$A_0(V)$-modules, where $A_0(V)$ is the Zhu algebra.
Since $V$ is $C_2$-cofinite, we have $A_0(V)$ is finite dimensional.
Let $J=\Rad A_0(V)$ be the Jacobson radical.
Every simple $A_0(V)$-module is annihilated by $J$ and hence is naturally
a simple module over the semisimple algebra $A_0(V)/J$.  By the
Wedderburn theorem,
\[
A_0(V)/J
\simeq
\bigoplus_{j=0}^r
\operatorname{Mat}_{d_j}(\mathbb C).
\]
The pairwise inequivalent simple modules $W^j(\lambda)$ correspond to
distinct matrix summands.  For each $j$, let
$\overline e_j\in A_0(V)/J$ be the identity of the matrix summand
corresponding to $W^j(\lambda)$, and choose a lift $e_j\in A_0(V).$
Then, we have
\[
e_j\mid_{W^j(\lambda)}
=
\Id_{W^j(\lambda)},
\qquad
e_j|_{W^k(\lambda)}
=
0,~~
k\neq j.
\]

Now choosing $v^j\in V$ such that its image in $A_0(V)$ satisfy $[v^j]=e_j.$
Substituting $v=v^j$ into
\eqref{ID5.14}, comparing the lowest degree, we have
\begin{eqnarray*}
0=
\sum_js_j
\tr_{W^j(\lambda)}
o_{W^j}(v^j)
=
s_j
\tr_{W^j(\lambda)}
e_j\mid_{W^j(\lambda)}
=
s_j\dim W^j(\lambda).
\end{eqnarray*}
This means $s_j=0.$

From the whole discussion, we can see that $F_{W^0},...,F_{W^r}$ are linearly independent, and we complete the proof.                   $\hfill\Box$

\begin{lemma}\label{LM5.9}
Let $\mathcal D$ be a $\mathbb C$-linear abelian category, and let
$A,B$ be simple objects. Suppose that $\eta$ is a natural endomorphism of the identity functor such that $\eta_A=\alpha\Id_A,
\eta_B=\beta\Id_B$
for different scalars $\alpha,\beta\in\mathbb C$.
Then
\[
\Ext_{\mathcal D}^1(A,B)=0.
\]
\end{lemma}
\proof Consider an arbitrary extension
\begin{equation}\label{ID5.15}
0
\longrightarrow
B
\xrightarrow{\ \iota\ }
E
\xrightarrow{\ p\ }
A
\longrightarrow
0.
\end{equation}
Since $\eta$ is natural, applying $\eta$ to $p:E\rightarrow A$, we have $p(\eta_E-\alpha\Id_E)=0.$
Therefore
\[
\im(\eta_E-\alpha\Id_E)
\subseteq
\ker p
=
\im\iota.
\]
Similarly, applying $\eta$ to $\iota:B\to E$, we have
$(\eta_E-\beta\Id_E)\iota=0$.
Thus $\eta_E-\beta\Id_E$ vanishes on $\im\iota$.
Therefore, we obtain
\begin{equation}\label{ID5.16}
(\eta_E-\beta\Id_E)
(\eta_E-\alpha\Id_E)
=
0.
\end{equation}

Since $\alpha\neq\beta$, we define
\begin{equation}\label{ID5.17}
e
=
\frac{\eta_E-\beta\Id_E}{\alpha-\beta}
\in
\operatorname{End}_{\mathcal D}(E).
\end{equation}
 Using identity
\eqref{ID5.16}, we have
\begin{eqnarray*}
e^2-e
=
\frac{
(\eta_E-\beta\Id_E)
\left(
\eta_E-\alpha\Id_E
\right)}
{(\alpha-\beta)^2}
=
0.
\end{eqnarray*}
Hence $e^2=e.$ Furthermore, we have
\begin{eqnarray}\label{ID5.18}
e\iota&=&\frac{1}{\alpha-\beta}(\eta_E-\beta\Id_E)\iota=0,\label{ID5.18}\\
pe&=&\frac{1}{\alpha-\beta}(p\eta_E-\beta p)=
\frac{1}{\alpha-\beta}(\eta_Ap-\beta p)=p.\label{ID5.19}
\end{eqnarray}

Let $I=\im e.$
Since $e$ is idempotent and $\mathcal C$ is abelian, $e$ splits.
Thus there are morphisms $r:E\twoheadrightarrow I,
j:I\hookrightarrow E$
such that
\[
e=jr,
\qquad
rj=\Id_I.
\]
Consider the restriction $\overline p=pj:I\longrightarrow A.$
First, by \eqref{ID5.19}, we have $p=\overline p\,r.$
Since $p$ is epic, $\overline p$ is epic.
On the other hand, let
$x:X\to I$ satisfy $\overline p\,x=0.$
Then $pjx=0.$
Since $\ker p=\im\iota,$
there exists a morphism $y:X\to B$ such that $jx=\iota y.$
Since $jx$ lies in $\im e$, the idempotent $e$ acts as the
identity on it, we have $e\,jx=jx.$
But, by identity \eqref{ID5.18}, $e\,jx=e\iota y=0.$
Hence, we get
\[
jx=0.
\]
Since $j$ is monic, it follows that $x=0.$
Therefore $\overline p$ is monic and isomorphic.

Define $s=j\,\overline p^{-1}:A\longrightarrow E$, we can see $ps=\Id_A$, i.e., $s$ is a section of $p$.
Thus the short exact sequence
\eqref{ID5.15} splits.
Therefore $\Ext_{\mathcal D}^1(A,B)=0$ and we complete the proof.
$\hfill\Box$

\begin{theorem}\label{THM5.10}
If $\mathcal C$ is rigid, and the complete trace functions satisfy
\[
SF_V=\sum_{j=0}^r s_jF_{W^j}.
\]Then $\mathcal C$ is semisimple.
\end{theorem}
\proof (i) Let $W\in\mathcal C$ be rigid, and choose a left
dual $W^*$ with evaluation and coevaluation
\[
e:W^*\boxtimes W\longrightarrow V,
\qquad
i:V\longrightarrow W\boxtimes W^*.
\]
Then we have
\begin{eqnarray}\label{ID5.20}
(\Id_W\boxtimes e)
    (i\boxtimes\Id_W)
    =
    \Id_W,
\qquad
 (e\boxtimes\Id_{W^*})
    (\Id_{W^*}\boxtimes i)
    =
    \Id_{W^*}.
\end{eqnarray}
Note that $W^*$ need not be equal to the contragredient module $DW=W'$.
Put
\[
R=W^*,
\qquad
Q=W\boxtimes W^*.
\]
Using the morphism $\beta_X:Q\boxtimes X\to X$ defined by identity \eqref{id5.9}, define
\begin{equation}\label{ID5.21}
h_{W,X}
=
\beta_X
(i\boxtimes\Id_X)
:
X\longrightarrow X.
\end{equation}
Using the naturality of $\beta_-$, if $f:X\to Z$ is a morphism in $\mathcal C$, we have
\begin{eqnarray*}
fh_{W,X}
=
\beta_Z
(\Id_Q\boxtimes f)
(i\boxtimes\Id_X)
=
\beta_Z
(i\boxtimes\Id_Z)f
=
h_{W,Z}f.
\end{eqnarray*}
Thus $h_{W,-}$
is a natural endomorphism of the identity functor.

Recall from identity \eqref{id5.7} that $\alpha=
\operatorname{Id}_W\boxtimes e:Q\boxtimes W\longrightarrow W.$ Using the naturality of tensor-product intertwining operators and identities \eqref{ID5.20}, we have
\begin{eqnarray*}
\mathcal Y_\alpha(i(v),z)
=
\alpha\mathcal Y^{Q,W}(i(v),z)
=
\alpha
(i\boxtimes\operatorname{Id}_W)
\mathcal Y^{V,W}(v,z)
=
Y_W(v,z).
\end{eqnarray*}
Similarly, we have
\begin{eqnarray*}
\mathcal Y_{\beta_{W^j}}(i(v),z)
=
\beta_{W^j}\mathcal Y^{Q,X}(i(v),z)
=
\beta_{W^j}
(i\boxtimes\Id_{W^j})
Y_{W^j}(v,z)
=
h_{W,{W^j}}Y_{W^j}(v,z).
\end{eqnarray*}
Since $i$ is a $V$-module morphism, it commutes with all Virasoro
operators and hence with the geometric coordinate operator
$\mathcal U(1)$.
It follows that
\begin{eqnarray*}
(SG_{\alpha,W})(i(v),\tau)
&=&
G_{\alpha,W}
\left(
(-\tau^{-1})^{L_Q(0)}i(v),
-\tau^{-1}
\right)
\\
&=&
G_{\alpha,W}
\left(
i((-\tau^{-1})^{L_V(0)}v),
-\tau^{-1}
\right)
\\
&=&
F_W
\left(
(-\tau^{-1})^{L_V(0)}v,
-\tau^{-1}
\right)
\\
&=&
(SF_W)(v,\tau),\\
G_{\beta_{W^j},{W^j}}(i(v),\tau)
&=&
\tr_{W^j}
\left(
h_{W,{W^j}}
Y_{W^j}(\mathcal U(1)v,1)
q^{L_{W^j}(0)-c/24}
\right)
\\
&=&
\tr_{W^j}
\left(
h_{W,{W^j}}
o_{W^j}(\mathcal U(1)v)
q^{L_{W^j}(0)-c/24}
\right).
\end{eqnarray*}
Using Lemma \ref{Lm4.8} (i), we get
\begin{equation}\label{ID5.22}
(SF_W)(v,\tau)
=
\sum_j s_j
\tr_{W^j}
\left(
h_{W,W^j}
o_{W^j}(\mathcal U(1)v)
q^{L_{W^j}(0)-c/24}
\right).
\end{equation}

(ii) We now take $W=W^l.$
Since $W^j$ is simple, Schur's lemma gives a scalar
$\lambda_{lj}\in\mathbb C$ such that $h_{W^l,W^j}=\lambda_{lj}\Id_{W^j}.$
Equation \eqref{ID5.22} therefore becomes
\begin{equation}\label{ID5.23}
SF_{W^l}
=
\sum_j
s_j\lambda_{lj}F_{W^j}.
\end{equation}

Let $\mathcal H=
\operatorname{span}_{\mathbb C}
\{F_{W^0},\ldots,F_{W^r}\}.$
Equation \eqref{ID5.23} shows that $S(\mathcal H)\subseteq\mathcal H.$
By Lemma \ref{LM5.8}, the complete functions $F_{W^0},\ldots,F_{W^r}$
form a basis of $\mathcal H$.
Since the $S$-transformation and the geometrically modified operation are
invertible. Hence we have
\[
S\mid_{\mathcal H}:\mathcal H\longrightarrow\mathcal H
\]
is injective.  Since $\mathcal H$ is finite dimensional,
$S\mid_{\mathcal H}$ is therefore an isomorphism.

Set $\Lambda=(\lambda_{lj})_{l,j},
\Gamma=\operatorname{diag}(s_0,\ldots,s_r).$
With the coefficient convention of
\eqref{ID5.23}, we have the matrix of $S\mid_{\mathcal H}$ is $\Lambda \Gamma.$ Hence $\Lambda \Gamma$ is invertible. In particular, $\Lambda$ and $\Gamma$
are also invertible.
Consequently, for every pair $j\neq k$, there exists an index $l$
such that $\lambda_{lj}\neq\lambda_{lk}.$
For this $l$, the natural endomorphism $h_{W^l,-}$
acts on $W^j$ and $W^k$ by the distinct scalars
$\lambda_{lj}$ and $\lambda_{lk}$, respectively.  By Lemma \ref{LM5.9}, for all $j\neq k,$ we have
\begin{equation}\label{ID5.24}
\operatorname{Ext}_{\mathcal C}^1(W^j,W^k)=0
\end{equation}

By Theorem \ref{thm3.4}, $\mathcal C$ is a finite abelian monoidal category with simple unit object. If $\mathcal C$ is rigid, it is a finite tensor category.  By \cite{EGNO}, the tensor unit of a
finite tensor category has no self-extension, i.e., $ \operatorname{Ext}_{\mathcal C}^1(V,V)=0.$
Together with \eqref{ID5.24}, this implies
\begin{equation}\label{ID5.25}
\operatorname{Ext}_{\mathcal C}^1(V,W^j)=0
\end{equation}
for all $j=0,1,...,r.$

(iii) For any object $X\in\mathcal C$, it has finite length $\ell(X)$.  Induction on $\ell(X)$. If $\ell(X)=1$, $X$ is simple, then identity \eqref{ID5.25} implies $\operatorname{Ext}_{\mathcal C}^1(V,X)=0$.  Suppose $\operatorname{Ext}_{\mathcal C}^1(V,Z)=0$ for all objects $Z\in \mathcal C$ with $\ell(Z)<\ell(X)$. Now
choose a simple quotient $W$ of $X$ and write
\[
0
\longrightarrow
Z
\longrightarrow
X
\longrightarrow
W
\longrightarrow
0.
\]
Applying $\operatorname{Hom}_{\mathcal C}(V,-)$ to above sequence, we obtain the long exact sequence containing
\[
\operatorname{Ext}_{\mathcal C}^1(V,Z)
\longrightarrow
\operatorname{Ext}_{\mathcal C}^1(V,X)
\longrightarrow
\operatorname{Ext}_{\mathcal C}^1(V,W).
\]
Now induction hypothesis and identity \eqref{ID5.25} force
$\operatorname{Ext}_{\mathcal C}^1(V,X)=0.$ This completes the induction. Thus for any $X\in \mathcal C,$ we have
\[
\operatorname{Ext}_{\mathcal C}^1(V,X)=0.
\]
A basic knowledge of homological algebra \cite{HS} tells us that above equation is equivalent to saying the tensor unit $V$ is projective.

For any $X\in\mathcal C$, since $\mathcal C$ is rigid, so is $X$.  For a suitable dual
$X^*$, using the properties of evaluation and coevaluation morphisms, there is a natural adjunction
\[
\operatorname{Hom}_{\mathcal C}
(V\boxtimes X,-)
\simeq
\operatorname{Hom}_{\mathcal C}
(V,-\boxtimes X^*),\qquad f\mapsto (f\boxtimes \Id_{X^*})(\Id_V\boxtimes i_X).
\]
Since $\mathcal C$ is rigid, so is $X^*$. Thus the tensor functor $-\boxtimes X^*$
has both left adjoint functor and right adjoint functor which means $-\boxtimes X^*$ is exact.
Since $V$ is projective, $\operatorname{Hom}_{\mathcal C}(V,-)$
is also exact.  Hence, we get
\[
\operatorname{Hom}_{\mathcal C}
(V\boxtimes X,-)
\]
is exact, and therefore, we have $X\simeq V\boxtimes X$
projective for every $X\in\mathcal C$.  Hence every short exact
sequence in $\mathcal C$ splits, and therefore $\mathcal C$ is semisimple. This completes the proof.     $\hfill\Box$

\begin{proposition}\label{PROP5.10}
If $\mathcal C$ is semisimple, then $V$ is regular.
\end{proposition}
\proof Let $W$ be an arbitrary weak $V$-module.  For every $w\in W$, the cyclic submodule $V\cdot w$ generated by $w$ is finitely
generated.  By Theorem \ref{thm3.4}, $V\cdot w$ belongs to $\mathcal C$. Since $\mathcal C$ is semisimple, we have $V\cdot w$ is a finite direct sum of simple ordinary $V$-modules.
Thus $W$ is the sum of its simple ordinary submodules.

Let $\mathcal S$ be the set of all submodules of $W$ which is a direct sum of simple submodules. Then $\mathcal S$ is a partially ordered set. By Zorn's Lemma, let $N$ be a maximal element of $\mathcal S$, and write
\[
N
=
\bigoplus_{\lambda\in\Lambda}S_\lambda
\subseteq W.
\]
Suppose that $N\neq W$.  Since $W$ is the
sum of its simple submodules, there exists a simple submodule
$X\subseteq W$ which is not contained in $N$.  Since $X$ is simple, if $X\cap N\neq0$, we have $X\subset N$. This contradicts the assumption of $X$. Thus we have $X\cap N=0$. Now $X\oplus N$
is a strictly larger direct sum of simple submodules which  contradicts
the maximality of $N$.  Thus
\[
N=W,
\]
and every weak $V$-module $W$ is a direct sum of irreducible
ordinary $V$-modules.  Hence $V$ is regular, and we complete the proof.          $\hfill\Box$

\begin{theorem}\label{THM5.11}
Let $V$ be a simple, $C_2$-cofinite lower truncated vertex operator algebra, and $\mathcal C$ the category of grading restricted generalized $V$-modules. Suppose the complete trace functions satisfy
\begin{eqnarray}\label{Last}
SF_V=\sum_{j=0}^r s_jF_{W^j}.
\end{eqnarray} If there is $B\in \mathcal C$, and a nonzero morphism $DV\boxtimes B\to V$, then $V$ is rational.
\end{theorem}
\proof By Theorem \ref{thm3.4}, we can see that $\mathcal C$ satisfies all the assumptions of Corollary \ref{co5.4} with $K=DV$. Thus if there is $B\in \mathcal C$, and a nonzero morphism $DV\boxtimes B\to V$, we have $K$ is invertible. In this case, using Proposition \ref{prop6.6}, under identity \eqref{Last}, we have $\mathcal C$ is rigid. Again under identity \eqref{Last} and Theorem \ref{THM5.10}, we know $\mathcal C$ is semisimple. Using Proposition \ref{PROP5.10}, we have $V$ is regular. Then, from \cite{ABD}, we have $V$ is rational.                    $\hfill\Box$

\begin{remark}\label{RE5.12}
By Theorem \ref{THM5.11}, to prove the rationality of $V$, we just need to find a $V$-module $B$ in $\mathcal C$ such that there is a nonzero intertwining operator of type $\binom{V}{K~~B}$, where $K=DV$ is the contragredient module of $V$.
\end{remark}

\section{Rationality of $V$}\label{sec6}
The purpose of this section is to prove the main Theorem of this paper, i.e., we will show $V$ is rational under the assumptions that $V$ is simple, $C_2$-cofinite, lower truncated and the space spanned by trace functions of irreducible $V$-modules is invariant under the modular group. The main difficulty is that the original conformal grading
may contain a finite negative part and the original contragredient
module need not be the vacuum module.
We overcome this difficulty in three steps.

First, using the trace
on the lowest homogeneous subspace, we construct a linear functional $f$ which detects the obstruction to modifying the conformal structure.
The modular invariance hypothesis implies that this obstruction
vanishes, yielding a correction vector $h\in V_1$
satisfying the correction equations.

Second, the correction vector determines a modified conformal
structure $\Omega=\omega+L(-1)h,$
whose zero-eigenvalue subalgebra provides the appropriate nonnegative
grading and M\"{o}bius structure needed for the subsequent construction.

Finally, using the relative inversion associated with the corrected
M\"{o}bius structure, we construct a Laurent-valued pseudoautomorphism
and obtain a simple module $B$ together with a nonzero morphism $DV\boxtimes B\longrightarrow V,$ where $DV=K.$

After those preparations, we use the invertibility criterion to show that $K$ is invertible. This implies $\mathcal C$ is rigid under the modular invariance assumption. Using modular invariance again, we show $V$ is regular, thus rational.

\subsection{First step: corrections}\label{sec6.1}
Recall that $-s$ is the lowest weight of the simple $C_2$-cofinite vertex operator algebra $V.$ Set $W=V_{-s}$, for $v\in V,$ define $$f(v)=\tr_{W}v_{-1},$$ it is obvious that $f(v)=0$ for $v\notin V_0.$
\begin{lemma}\label{lm4.1}
Let $f$ be defined as above. Then we have $$f(\textbf{1})=\dim W>0,\qquad f\circ L(-1)=0,\qquad f(a_0v)=0,a\in V_1,v\in V,\qquad f(u_{-2}v)=0,u,v\in V.$$
\end{lemma}
\proof First identity is trivial. For the second identity, we just need to show $f(L(-1)v)=0$ for $v\in V_{-1}$ since $L(-1)$ is homogeneous of degree $1$. Now let $v\in V_{-1},$ we have $(L(-1)v)_{-1}=v_{-2}.$ This implies $f(L(-1)v)=0$.

For the third identity, it is enough to show the identity holds for $v\in V_0.$ In this case, it is obvious that $f(a_ov)=0$ since $(a_0v)_{-1}=[a_0,v_{-1}]$ and $\tr AB=\tr BA$.

For the last identity, by linearity and definition, we can assume $u,v$ are homogeneous and $\wt u_{-2}v=0.$ Since $W$ is an $A(V)$-module, we have $u\circ v$ acts on $W$ as zero. Thus $f(u\circ v)=0.$ On the other hand, $$u\circ v=\sum_{j\geq0}\binom{\wt u}{j}u_{j-2}v.$$ Since $u_{j-2}v\notin V_0$ for $j>0$, we have $f(u_{-2}v)=0.$                       $\hfill\Box$

\begin{corollary}\label{co4.2}
For $n<-1$, we have $f(u_nv)=0$.
\end{corollary}
\proof Since $u_{-k-2}v=\frac{1}{(k+1)!}(L(-1)^{k}u)_{-2}v$ for all $k\geq0.$          $\hfill\Box$

\begin{lemma}\label{lm4.3}
For any $g\in V',0\neq u\in V$, there is a homogeneous $v\in V$ such that $gY(u,z)v\neq0.$
\end{lemma}
\proof Let $J=\{u\in V\mid gY(u,z)v=0,\forall v\in V_n,n\in\Z\}.$ It is obvious that $J$ is stable under $L(-1)$ and $\textbf{1}\notin J$. Now let $u\in J,v\in V$, by weak commutativity, there is $N\in \mathbb{N}$, such that $$(z_1-z_2)^NgY(v,z_2)Y(u,z_1)=(z_1-z_2)^NgY(u,z_1)Y(v,z_2)=0.$$ Now the coefficients belong to $\mathbb{C}((z_1,z_2))$ which is a field. Thus the coefficients are invertible. This implies $gv_mu_n=0$ for all $m,n$. Then using Jacobi identity, we can see that $J$ is stable under $v_m$. Thus $J$ is an ideal of $V$. Now from the simplicity of $V$, we have $J=0$.

Hence if $u\neq0$, there is $v\in V$ such that $gY(u,v)\neq0.$                  $\hfill\Box$

Let $B_f(a,b)=f(a_1b)$ on $V_1$. By skew symmetry and Lemma \ref{lm4.1}, we can see that $B_f$ is a symmetric bilinear form on $V_1$. Let $R_f=\Rad B_f=\{u\in V_1\mid B_f(u,v)=0,~\forall v\in V_1\}$ be the radical of $B_f.$ Then it is obvious that $L(-1)V_0\subset R_f$ by Lemma \ref{lm4.1}.

\begin{lemma}\label{lm4.4}
For all $a\in R_f$, we have $a_0$ is locally nilpotent on $V.$
\end{lemma}
\proof Using skew symmetry, we have the following relations $$2a_0a=\sum_{j>0}\frac{(-1)^{j+1}}{j!}L(-1)^ja_ja=L(-1)\left(\sum_{j>1}\frac{(-1)^{j+1}}{j!}L(-1)^ja_ja\right)\in L(-1)V_0.$$
Since $a_0$ is homogeneous of degree $0$, from $L(-1)$-bracket property, we have $L(-1)V_0$ is $a_0$-stable.

Since $V_1$ is finite-dimensional, we have
the generalized eigenspace decomposition of $V_1$ under linear map $a_0$ by
$
V_1
=
V_1^{[0]}
\oplus
\bigoplus_{\mu\neq 0}V_1^{[\mu]}.
$
Let
$$a
=
a'
+
\sum_{\mu\neq0}a^\mu,\qquad
a'\in V_1^{[0]},
a^\mu\in V_1^{[\mu]}
$$ be the decomposition of $a$ under above direct sum.
From the first paragraph, for every $\mu\neq0$, we have $a_0 a^\mu
\in
L(-1)V_0\cap V_1^{[\mu]}.
$
Since $a_0$ is invertible on $V_1^{[\mu]}$ for $\mu\neq0$, we have
$
a^\mu
\in L(-1)V_0.
$
Consequently,
\[
a-a'
=
\sum_{\mu\neq0}a^\mu
\in L(-1)V_0.
\]

Suppose $a-a'=L(-1)b$ for some $b\in V_0$, we have $a'
=
a-(a-a')
\in R_f,$ since $L(-1)V_0\subseteq R_f$
and $a\in R_f$. Furthermore, we have $(a-a')_0=(L(-1)b)_0=0,$
i.e., $a'_0=a_0.$
On the other hand, since $a'\in V_1^{[0]}$, there exists an integer $N>0$
such that $a_0^N a'=0.$ Now for each $j\in\Z,$ using Jacobi identity, we have $a_0^N(a'_j)
=
(a_0^N a')_j
=
0.$

Let $V^{[\lambda]}$
be the generalized $\lambda$-eigenspace of $a_0$. For
$x\in V^{[\lambda]}$, we have $(a_0-\lambda)^M x=0$
for some $M\in\mathbb{N}$. Thus, for $k\geq N+M$, we have $(a_0-\lambda)^k(a'_j x)=0.$
Therefore,
$
a'_jV^{[\lambda]}
\subseteq
V^{[\lambda]}.
$

Suppose $a_0$ has a nonzero eigenvalue $\lambda\neq0.$
Since $V$ is lower truncated, there exists a
least integer $p$ such that $V_p^{[\lambda]}
=
V_p\cap V^{[\lambda]}$
is a finite-dimensional nonzero space, we may choose an
actual eigenvector $0\neq u\in V_p^{[\lambda]}$
such that $a_0 u=\lambda u.$ Now
for every $j\geq1$, we have $a'_j u\in V_{p-j}^{[\lambda]}.$
By the choice of $p,$ we have $a'_j u=0$ for all $j>0.$

By Lemma \ref{lm4.3}, there exists a homogeneous vector $v\in V_q$ such that $fY(u,z)v\neq0$ with $q$ minimal.
From definition of $f$, we have $f(u_mv)\neq0$ only if $m=p+q-1.$

Now set $n=p+q-2,$ we have $u_n v\in V_1,$ and
$
0
=
B_f(a',u_n v)
=
f(a'_1u_n v)
$ since we have shown $a'\in R_f.$
Using Jacobi identity and above discussion, we have $[a'_1,u_n]
=
\lambda u_{n+1}.
$
Therefore,
\[
\begin{aligned}
0
&=
f(a'_1u_n v)
\\
&=
f\bigl([a'_1,u_n]v+u_n a'_1v\bigr)
\\
&=
f\bigl(\lambda u_{n+1}v+u_n a'_1v\bigr)
\\
&=
\lambda f(u_{n+1}v)
+
f(u_n a'_1v).
\end{aligned}
\]
For the second term, since  $a'\in V_1,v\in V_q$, we have $a'_1v\in V_{q-1}$ is homogeneous. From the choice of $q$, we get $f(u_n a'_1v)=0.$
Consequently, $\lambda f(u_{n+1}v)=0,$ i.e.,  $f(u_{p+q-1}v)=0.$ But this contradicts the choice of $v$ and $q$.
Therefore $a_0$ has no nonzero eigenvalues on $V$.

Finally, since $a_0$ preserves each finite dimensional homogeneous subspace $V_t$, and every vector of $V$ is a finite sum of homogeneous vectors,
we have $a_0$ is locally nilpotent on $V$.                                    $\hfill\Box$

\begin{lemma}\label{lm4.5}
Let $T;g(\tau)\rightarrow g(\tau+1)$ be the translation and $\mathcal H$ a finite-dimensional space of holomorphic functions
on $\mathbb H$ which is invariant under $T$. Suppose $T$ is invertible and semisimple on $\mathcal H$. If $\sum_{k=r}^{d}\tau^k g_k(\tau)=0,
g_k\in\mathcal H,$
then for all $k$, we have $g_k=0$.
\end{lemma}
\proof Multiplying the identity by a suitable power of $\tau$, we may assume
that $r\geq0$. Since $T$ is semisimple on $\mathcal H$, let
\[
p(x)=\prod_{\lambda\in\operatorname{Spec}(\mathsf T)}(x-\lambda).
\]
Then $p$ is the minimal polynomial of $T$ and squarefree, thus $p'(\lambda)\neq0.$
Consequently, $T p'(T)$ is invertible on $\mathcal H$.

For $g\in\mathcal H$, write $p(x)=\sum_{\ell=0}^{N}c_\ell x^\ell.$
Then
\[
\begin{aligned}
p( T)(\tau^k g)
&=
\sum_{\ell=0}^{N}
c_\ell\,T^\ell(\tau^k g)
=
\sum_{\ell=0}^{N}
c_\ell(\tau+\ell)^k T^\ell g\\
&=
\tau^k p(T)g
+
k\tau^{k-1}
\sum_{\ell=0}^{N}c_\ell\ell T^\ell g
+
\text{lower powers of }\tau
\\
&=
k\tau^{k-1}
 T p'( T)g
+
\text{lower powers of }\tau.
\end{aligned}
\]
Hence, applying $p( T)$ to $\sum_{k=0}^{d}\tau^kg_k=0,$
the coefficient of the highest remaining power $\tau^{d-1}$ is $d\, T p'( T)g_d.$
Hence $d\, T p'( T)g_d=0.$
Since $\mathsf T p'(\mathsf T)$ is invertible, we obtain $g_d=0.$

Induction on $d$, for every $k$, we get $g_k=0.$ This completes the proof.                                  $\hfill\Box$

Now we are ready to finish our first step.
\begin{proposition}\label{prop4.6}
Suppose $V$ is simple, $C_2$-cofinite and lower truncated. Let $V=W^0, ..., W^r$ be all the irreducible $V$-modules. If there exist constants $s_j\in\C$, such that $$SF_V=\sum_{j=0}^r s_jF_{W^j}.$$ Then there is $h\in V_1$ such that $$f(L(1)a)=2f(h_1a), \forall a\in V_1,\qquad f\circ (L(1)-2h_1)=0.$$Above two equations are called the correction equations and $h$ is called a corrector.
\end{proposition}
\proof Let $\mathcal H=\Span\{F_{W^j}\mid j=0,1,...,r\}$, then it trivial that $\mathcal H$ is $T$-stable. For $a\in R_f,$
set $w=\mathcal U(1)^{-1}a,$ then $w$ has a finite homogeneous decomposition $w=\sum_k w^k,
w^k\in V_k$ with $k\leq 1.$ Furthermore, from the definition, we have
\[
\mathcal U(1)^{-1}
=
\exp\left(\sum_{j\geq1}A_jL(j)\right)
(2\pi i)^{-L(0)}.
\]
Thus $w^1=(2\pi i)^{-1}a$ and
\[
\begin{aligned}
w^0
=
(2\pi i)^{-1}A_1L(1)a
=
(2\pi i)^{-1}(-\pi i)L(1)a
=
-\frac12L(1)a.
\end{aligned}
\]

Since $a\in R_f$, we have $a_0$ is locally nilpotent by Lemma \ref{lm4.4}. Hence, we get $\tr_{V_n}(a_0)=0.$
It follows that
\[
\begin{aligned}
F_V(w,\tau)
&=
\tr_{V}
o(\mathcal U(1)w)
q^{L_0-c/24}
\\
&=
\tr_{V}
o(a)q^{L_0-c/24}
\\
&=
\sum_n
q^{n-c/24}\tr_{V_n}(a_0)=0.
\end{aligned}
\]

For every homogeneous component $w^k$, define $\mathcal G_k(\tau)
=\sum_{j=0}^{r}s_jF_{W^j}(w^k,\tau).$
The modular hypothesis $SF_V=\sum_{j=0}^{r}s_jF_{W^j}$
applied to $w^k\in V_k$ gives
\[
\begin{aligned}
\mathcal G_k(\tau)
=
(SF_V)(w^k,\tau)
=
F_V\left(
(-\tau^{-1})^{L(0)}w^k,-\tau^{-1}
\right)
=
(-\tau^{-1})^k
F_V(w^k,-\tau^{-1}).
\end{aligned}
\]
On the other hand, from $F_V(w,\tau)=0$,
we also have $F_V(w,-\tau^{-1})=0.$
Putting those together, we get
\[
\begin{aligned}
0
=
F_V(w,-\tau^{-1})
=
\sum_kF_V(w^k,-\tau^{-1})
=
\sum_k(-\tau)^k\mathcal G_k(\tau).
\end{aligned}
\]
Since $\mathcal G(\tau)\in\mathcal H$, from Lemma \ref{lm4.5}, we have $\mathcal G_k=0$ for all $k.$
In particular,
\[
\begin{aligned}
0=
\mathcal G_0(\tau)
=
(SF_V)(w^0,\tau)
=
F_V(w^0,-\tau^{-1}).
\end{aligned}
\]
Thus, we get $F_V(w^0,\tau)=0.$
Moreover, we have $F_V(L(1)a,\tau)=0$ since $L(1)a=-2w^0.$

Recall that $h_V=-s$ and $W=V_{-s}$, we now examine the coefficient of the lowest power $q^{-s-c/24}$
of $F_V(L(1)a,\tau)$.
Since $L(1)a\in V_0,$
we can write $\mathcal U(1)L(1)a
=L(1)a+\sum_{\ell<0}x_\ell,$ where $x_\ell\in V_\ell$.
Consequently, the coefficient of $q^{-s-c/24}$ in
$F_V(L(1)a,\tau)$ is
\[
\tr_{W}
o(L(1)a)
+
\sum_{\ell<0}
\tr_{W}o(x_\ell).
\]
Now by definition of $f$ and Lemma \ref{lm4.1}, we have $[q^{-s-c/24}]F_V(L(1)a,\tau)
=
f(L(1)a).$
This implies $f(L(1)a)=0$ for all $a\in R_f$.
Thus the linear functional
\[
\ell:V_1\longrightarrow\mathbb C,
\qquad
\ell(a)=f(L(1)a),
\]
annihilates the radical $R_f$, i.e., $\ell|_{R_f}=0.$

Now consider the linear map $\Psi:V_1\longrightarrow V_1^*,
\Psi(x)=B_f(x,-).$
Since $B_f$ is symmetric, we have $\ker\Psi=R_f.$
Moreover, if $y\in R_f$, then $B_f(x,y)=B_f(y,x)=0,$
so $\im\Psi
\subseteq
\Ann(R_f).$
On the other hand, from rank-null duality, we have
\[
\begin{aligned}
\dim\im\Psi
&=
\dim V_1-\dim R_f
=
\dim\Ann(R_f).
\end{aligned}
\]
Hence $\im\Psi=\Ann(R_f).$
This implies $\ell\in\Ann(R_f)=\im\Psi.$
Let $h'\in V_1$ such that $\ell=\Psi(h')$, i.e., $\ell(a)=B_f(h',a)$ for all $a\in V_1$.
Define
\[
h=\frac12h'.
\]
For $a\in V_1,$ we obtain
$
f(L(1)a)
=
2B_f(h,a)
=
2f(h_1a).
$ This implies $f\circ(L(1)-2h_1)=0$ on $V_1$.
On the other hand, both $L(1)$ and $h_1$ are homogeneous of degree $-1$, using definition of $f$, for $n\neq1$, we have
\[
f(L(1)V_n)=0,
\qquad
f(h_1V_n)=0.
\]
Therefore, we have
\[
f\circ(L(1)-2h_1)=0
\]
on all of $V$.
This completes the proof.                                     $\hfill\Box$

\begin{remark}\label{re4.7}
If $V_1=0$, then the corrector in Proposition~\ref{prop4.6}
is necessarily $h=0$. In this case the correction equation is
automatic and $f\circ L(1)=0.$
\end{remark}

\subsection{Second step: modified conformal structure}\label{sec6.2}
Now we start our second step.
\begin{lemma}\label{lm4.8}
Let $h$ be defined by Proposition \ref{prop4.6}, and $\Omega=\omega+L(-1)h\in V_2$, and $\Omega^0$ the generalized $h_0$-eigenvalue zero component of $\Omega$. Then there is $b\in V_0$ such that $$\Omega^0=\omega+L(-1)\overline{h},\qquad \overline{h}_0=h_0,$$ where $\overline{h}=h+L(-1)b.$ Moreover, let $A=\Omega^0_2=L(1)-2\overline{h}_1,H=\Omega_1^0=L(0)-h_0$, then $f\circ A=0$, and $A$ preserves each generalized $h_0$-eigenspace, for all $u\in V_t,n\in\Z$, we have $$[A,u_n]=(2t-n-2)u_{n+1}+(Au)_n-2(h_0u)_{n+1}.$$
\end{lemma}
\proof (i) First, since $(L(-1)b)_0=0$, we have $\overline{h}_0=h_0.$ From $h\in V_1$, we have $\overline{h}_0=h_0$ preserves each homogeneous subspace $V_t.$
Moreover, we have $$[h_0,v_n]=(h_0v)_n,\qquad h_0L(-1)^k=L(-1)^kh_0$$ for all $k\geq0,n\in\Z.$

Using skew symmetry, we have
\[
\begin{aligned}
h_0\omega
&=
\sum_{j\geq0}
\frac{(-1)^{j+1}}{j!}
L(-1)^j(\omega_jh)\\
&=
-\omega_0h
+
L(-1)(\omega_1h)
+
\sum_{j\geq2}
\frac{(-1)^{j+1}}{j!}
L(-1)^j(\omega_jh)\\
&=
\sum_{j\geq2}
\frac{(-1)^{j+1}}{j!}
L(-1)^jL(j-1)h\\
&=L(-1)^2\sum_{j\geq0}\frac{(-1)^{j+3}}{(j+2)!}L(-1)^jL(j+1)h\in L(-1)^2V_0.
\end{aligned}
\]
Similarly we have
$
2h_0 h
=
\sum_{j\geq1}
\frac{(-1)^{j+1}}{j!}
L(-1)^j(h_jh)\in L(-1)V_0.
$
Putting those together, we have $$h_0\Omega=h_0\omega+h_0L(-1)h=h_0\omega+L(-1)h_0h\in L(-1)^2V_0.$$

Let $V_2=V_2^{[0]}\oplus\bigoplus_{\lambda\neq0}V_2^{[\lambda]}$ be the generalized $h_0$-eigenspace decomposition of $V_2$, and
\[
\Omega
=
\Omega^0
+
\sum_{\lambda\neq0}\Omega^\lambda,
\] the corresponding decomposition of $\Omega$.
Since $L(-1)^2V_0$ is $h_0$-stable, we also have the decomposition
\[
L(-1)^2V_0
=
\bigoplus_{\lambda}
\left(L(-1)^2V_0\cap V_2^{[\lambda]}\right).
\]
Thus for $\lambda\neq0$, we have $h_0\Omega^\lambda
\in
L(-1)^2V_0\cap V_2^{[\lambda]}$. Since $h_0$ is invertible on $V_2^{[\lambda]}$, we have
$
\Omega^\lambda\in L(-1)^2V_0.
$
Therefore
\[
\Omega-\Omega^0
=
\sum_{\lambda\neq0}\Omega^\lambda
\in L(-1)^2V_0.
\]
Hence there exists some $b\in V_0$ such that $\Omega^0-\Omega=L(-1)^2b.$ Now
\[
\begin{aligned}
\Omega^0
=
\omega+L(-1)h+L(-1)^2b=
\omega+L(-1)(h+L(-1)b)=\omega+L(-1)\bar h,
\end{aligned}
\]
where $\bar h=h+L(-1)b.$ Thus we get the desired relation.

(ii) Since
$
\Omega^0_n
=
\omega_n+(L(-1)\bar h)_n,$ we obtain $\Omega^0_0=L(-1)$,
$
\Omega^0_1=L_0-\bar h_0=
L_0-h_0,
$
and
$
\Omega^0_2
=
L_1-2\bar h_1.
$
Therefore, we have
$$
H=L(0)-h_0,\qquad
A=L(1)-2\bar h_1.
$$

(iii) If $a\in V_1$, we have $f((L(-1)b)_1a)=-f(b_0a).$ Using skew symmetry, we get
\[
b_0a
=
-a_0b
+
\sum_{j\geq1}
\frac{(-1)^{j+1}}{j!}
L(-1)^j(a_jb).
\]
From Lemma \ref{lm4.1},
we obtain $f(b_0a)
=
-f(a_0b)=0.$ This implies $f((L(-1)b)_1a)=0.$
Consequently,
$f(Aa)
=
f\bigl((L(1)-2\bar h_1)a\bigr)=
f\bigl((L(1)-2h_1)a\bigr)
-
2f((L(-1)b)_1a)=
0.
$

Since $A$ is homogeneous of degree $-1$,
from definition of $f$, we have $f(Aa)=0$ for $a\notin V_1.$ Hence, we get
\[
f\circ A=0.
\]

(iv) It is trivial that $(\operatorname{ad}h_0)A=[h_0,\Omega_2^0]=(h_0\Omega^0)_2$. By definition of $\Omega^0$, there exists $N>0$ such that $h_0^N\Omega^0=0.$ This implies $(\operatorname{ad}h_0)^NA=0.$

For $x\in V^{[\lambda]}$,
there exists $M>0$ such that $(h_0-\lambda)^Mx=0.$ Thus we have
\[
(h_0-\lambda)^{N+M}Ax
=
\sum_{j=0}^{N+M}
\binom{N+M}{j}
(\operatorname{ad}h_0)^j(A)
(h_0-\lambda)^{N+M-j}x=0.
\]
Hence $A$ preserves every generalized $h_0$-eigenspace.

(v) Let $u\in V_t$ be homogeneous. From Jacobi identity, we obtain
\[\begin{aligned}
[L(1),u_n]&=
(2t-n-2)u_{n+1}
+
(L(1)u)_n,\\
[\bar h_1,u_n]
&=
(h_0u)_{n+1}
+
(\bar h_1u)_n.
\end{aligned}\]
Therefore
\[
[A,u_n]
=
(2t-n-2)u_{n+1}
+
(Au)_n
-
2(h_0 u)_{n+1},
\]
This completes the proof.                           $\hfill\Box$

\begin{lemma}\label{lm4.9}
Let $U$ be a nonnegatively graded vertex operator algebra. For any $u,v\in U_0$, define $u\cdot v=u_{-1}v$. Then $U_0$ is a commutative associative algebra with identity $\textbf{1}.$
\end{lemma}
\proof It is obvious that $"\cdot"$ is well-defined and $\textbf{1}$ is the identity. Using skew symmetry, we have
\[
a_{-1}b
=
\sum_{j\ge0}
\frac{(-1)^j}{j!}
L(-1)^j(b_{j-1}a)=b_{-1}a.
\]
Thus $U_0$ is commutative.

Using Jacobi identity, we have
\[
(a_{-1}b)_{-1}c
=
\sum_{j\ge0}
a_{-1-j}b_{-1+j}c
+
\sum_{j\ge0}
b_{-2-j}a_jc=a_{-1}(b_{-1}c).
\]
Thus $U_0$ is associative.                              $\hfill\Box$

\begin{theorem}\label{thm4.10}
Let $h$ be defined by Proposition \ref{prop4.6}, set $U=V^{[0]}$, i.e., $U$ is the generalized $h_0$-eigenspace of generalized eigenvalue $0$. Then $U$ is a nonnegatively graded vertex subalgebra of $V$ and $U_0=\mathbb{C}\textbf{1}$.
\end{theorem}
\proof (i) Suppose $u\in V^{[\lambda]},v\in V^{[\mu]}$, there are $M,N\in\mathbb{N}$ such that
$
(h_0-\lambda)^N u=(h_0-\mu)^M v=0.
$ Hence we have
\[
(h_0-\lambda-\mu)^{M+N}(u_nv)
=
\sum_{j=0}^{N+M}
\binom{N+M}{j}
\bigl((h_0-\lambda)^{(N+M-j)} u\bigr)_n
\bigl((h_0-\mu)^{j}v\bigr)=0.
\]
Thus we get
$
V^{[\lambda]}\cdot V^{[\mu]}
\subseteq
V^{[\lambda+\mu]}$, where $V^{[\lambda]}\cdot V^{[\mu]}=\{u_nv\mid u\in V^{[\lambda]},v\in V^{[\mu]},n\in\mathbb N\}$.

Furthermore, it is obvious that $h_0\textbf{1}=0$, and $h_0L(-1)=L(-1)h_0.$ Thus,
we conclude that $U$ is a vertex subalgebra.

Since $h\in V_1$, we have $[L(0),h_0]=0.$
Thus $U$ has a natural gradation induced from $V.$
Since $V$ is lower truncated and every $V_n$ is finite-dimensional,
the same is true for $U$:
\[
U_n=0
\quad\text{for }n\ll0,
\qquad
\dim U_n<\infty.
\]

(ii) From Lemma \ref{lm4.1}, we can see $f\circ h_0=0.$ Since $h_0$ is invertible on $V_n^{[\lambda]}$, for any $x\in V_n^{[\lambda]}$, there is $y\in V_n^{[\lambda]}$ such that $x=h_0y.$ Hence we get $f(x)=0$. This means that $f(V_n^{[\lambda]})=0$ for all $\lambda\neq0.$ Now by Lemma \ref{lm4.3}, for $u\in U,$ there is $v\in U$ such that $fY(u,z)v\neq0.$

If there is $p<0$, such that $U_p\neq0$. Choose $p$ to be minimal and $u\in U_p$ such that $h_0u=0.$
From Lemma \ref{lm4.8}, we have $A(U)\subseteq U.$ Thus $A(U_p)\subset U_{p-1}=0.$ This implies $Au=0.$
From Lemma \ref{lm4.3}, choose $v\in U_q$ with $q$ minimal such that $fY(u,z)v\neq0.$ Then we get
\[
k=f(u_{p+q-1}v)\neq0.
\]
Set $n=p+q-2,$ from Lemma \ref{lm4.8}, using definition of $f$, the choice of $u,v$, minimality of $p,q$, and the degree of all linear maps, we have
\begin{eqnarray*}
0&=&f(A(u_nv))=f([A,u_n]v+u_nAv)\\
&=&(2p-n-2)f(u_{n+1}v)+f((Au)_nv)-2f((h_0 u)_{n+1}v)+f(u_nAv)\\
&=&k(p-q)+f(u_{p+q-2}Av).
\end{eqnarray*}
If $f(u_{p+q-2}Av)\neq0$, then $Av\in U_{q-1}$ has a lower weight which contradicts the choice of $q$. Hence we eventually arrived at $k(p-q)=0$. This means $p=q.$
In this case, we have $f(u_{2p-1}v)=0$ by Corollary \ref{co4.2}, since $p<0.$ Thus we get $k=0$ which is absurd. This contradiction shows that $p\geq0$, i.e., $U$ is nonnegatively graded.

(iii) By Lemma \ref{lm4.9}, $U_0$ is a commutative associative algebra. For $e\in U_0,$ if $n=0$, we have $e_0L(-1)e=L(-1)e_0e=0.$
If $n\ge1$, we have $e_nL(-1)e=0$ by gradation. Suppose $e$ is idempotent, i.e., $e_{-1}e=e,$ we get
\[
L(-1)e=
L(-1)(e_{-1}e)=
(L(-1)e)_{-1}e+e_{-1}L(-1)e.
\]
On the other hand, using skew symmetry, we have
\[
(L(-1)e)_{-1}e
=
\sum_{j\ge0}
\frac{(-1)^j}{j!}
L(-1)^j(e_{j-1}L(-1)e)=e_{-1}(L(-1)e).
\]
Hence we get $L(-1)e=2e_{-1}(L(-1)e)$. This implies $$e_{-1}L(-1)e=2e^2_{-1}(L(-1)e)=2e_{-1}(L(-1)e),$$Thus $e_{-1}(L(-1)e)=0$, and $L(-1)e=0.$

By $L(-1)$-bracket property, we have $Y(e,z)=e_{-1}.$ This implies for all $v\in V$,
\[
[e_{-1},Y(v,z)]=0.
\]
Thus $e_{-1}$ is an endomorphism of the adjoint $V$-module.
Moreover, $e_{-1}^2\mathbf1=e_{-1}e=e=e_{-1}\mathbf1.$
From the homomorphism, we have
$e_{-1}^2=e_{-1}$ on all of $V$.

Now let $V=\im e_{-1}\oplus\ker e_{-1}$, since $V$ is simple, we have $\im e_{-1}=0$ or $\ker e_{-1}=0.$ This is equivalent to $e_{-1}=0$ or $e_{-1}=\Id.$
Applying this to the vacuum, we get $e=0$, or $e=\mathbf1.$
Thus $U_0$ has no nontrivial idempotents.

If $U_0$ is decomposable, then each direct summand can produce a nontrivial idempotent element. Thus $U_0$ is indecomposable. This implies $U_0$ is a finite-dimensional commutative local algebra. Let $J=\Rad (U_0)$ be its unique maximal ideal. we have $J$ is nilpotent, see \cite{P}.

It is obvious that for any $t\in\C,$ the map $\exp{th_0}:U_0\rightarrow U_0$ is a well-defined algebra automorphism.
Since, the Jacobson radical is invariant under every algebra automorphism. This implies $J$ is stable under $h_0.$

If $J\neq0$, we have $\ker h_0\cap J\neq0$ since $h_0$ has an eigenvector with eigenvalue $0$ in $J$. Let $u\in J$ be nonzero
such that $h_0u=0.$ We deduce that $Au=0$ since $U$ is nonnegatively graded.
From Lemma \ref{lm4.3} and definition of $f$, there exists a
homogeneous $v\in U_q$ such that $f(u_{q-1}v)\neq0.$
Choose $q$ to be minimal, and put $k=f(u_{q-1}v)\neq0.$

It is obvious that $q\geq0.$ If $q=0$, for $u,v\in U_0$, we can see that $u\cdot v$ is exactly the Zhu's product $u*v$.
On the other hand, $W=V_{-s}$ is an $A(V)$-module with $v$ acts on $W$ by $o(v)=v_{\wt v-1}$. For $v\in U_0$, this is exactly $v_{-1}$.
Therefore, on $W$, we have $o(u_{-1}v)=o(u)o(v)$
on $W$. Hence
\[
f(u_{-1}v)
=
\tr_W\bigl(o(u)o(v)\bigr).
\]
Since $u\in J$ and $J$ is the radical of a finite-dimensional commutative algebra,
we have $u$ is nilpotent. Thus $o(u)o(v)$
is nilpotent since $U_0$ is commutative. This means $\tr_W(o(u)o(v))=0$. Hence, $f(u_{-1}v)=0.$ This contradicts that $f(u_{-1}v)=k\neq0.$
The contradiction means $q>0.$

Now consider $u_{q-2}v\in U_{1}$. From Lemma \ref{lm4.8},
we have
\begin{eqnarray*}
0&=&f\bigl(A(u_{q-2}v)\bigr)\\
&=&f([A,u_{q-2}]v)+f(u_{q-2}Av)\\
&=&-qf(u_{q-1}v)+f(u_{q-2}Av).
\end{eqnarray*}
If $f(u_{q-2}Av)\neq0,$ we get a nonzero vector $Av\in U_{q-1}$ which contradicts the minimality of $q$. Therefore, $f(u_{q-2}Av)=0,$ and we get
$0=-qf(u_{q-1}v)=-qk$. But this contradicts $q>0,k\neq0.$
This contradiction implies our assumption $J\neq0$ is false. Therefore $J=0$.

Now from above discussion, we have $U_0$ is a finite-dimensional commutative local $\mathbb C$-algebra and its unique maximal ideal is zero.
Hence $U_0$ is a field which is a finite extension of the algebraically closed field $\C$. Thus $U_0\cong\C$. This forces $U_0=\C \textbf{1}$,
and we complete the proof.                                  $\hfill\Box$

\begin{remark}\label{re4.11}
In Theorem \ref{thm4.10}, $U$ is not necessarily a vertex operator subalgebra, conformal or non-conformal. Its gradation is induced from the gradation of $V$.
\end{remark}

\begin{remark}\label{re4.12}
If $V_1=0$, by Remark \ref{re4.7}, we can see that $h$ defined in Proposition \ref{prop4.6} is zero. In this case, we have $U=V$. Hence by Theorem \ref{thm4.10}, we deduce that $V$ is nonnegatively graded and $V_0=\C \textbf{1}$. Thus if $V$ has negative gradation space, then $V_1\neq 0$ and $h$ defined in Proposition \ref{prop4.6} must exist under the modular invariance assumption.
\end{remark}

\begin{corollary}\label{co4.13}
Let $\Omega^0$ be defined in Lemma \ref{lm4.8}, we have $h_0\Omega^0=0.$ For any weak $V$-module $(M,Y_M)$, write $Y_M(v,z)=\sum v_n^Mz^{-n-1}$, define $$T^M=L(-1)^M=(\Omega^0_0)^M,~~ H^M=(\Omega^0_1)^M=L(0)^M-h_0^M,~~ A^M=(\Omega^0_2)^M=L(1)^M-2\overline{h}_1^M.$$ Then $\{T^M,H^M,A^M\}$ satisfy M\"{o}bius relations, and we have the corresponding fields covariance relations
\[
[T^M,Y_M(u,x)]
=
Y_M(L(-1)u,x)
=
\frac{d}{dx}Y_M(u,x),
\]
\[
[H^M,Y_M(u,x)]
=
Y_M(Hu,x)+xY_M(L(-1)u,x),
\]
\[
[A^M,Y_M(u,x)]
=
Y_M(Au,x)
+
2xY_M(Hu,x)
+
x^2Y_M(L(-1)u,x).
\]
Moreover $\Omega^0$ is a Virasoro element.
\end{corollary}
\proof First, we put what we have known.
By Lemma~\ref{lm4.8}, $\Omega=\omega+L(-1)h$ and $\Omega^0=\omega+L(-1)\bar h,\bar h=h+L(-1)b,b\in V_0,$ and $\bar h_0=h_0.$ Moreover, $\Omega^0\in U_2,U=V^{[0]}.$
Thus there exists some $N\gg0$ such that $h_0^N\Omega^0=0.$
Theorem~\ref{thm4.10} gives $U_n=0,n<0,U_0=\C \textbf{1}$.
Lemma~\ref{lm4.8} also gives $h_0\Omega^0\in L(-1)^2V_0.$
Indeed, from the proof of lemma \ref{lm4.8}, we already know $h_0(\omega+L(-1)h)\in L(-1)^2V_0.$ From definition, we can see $h_0\Omega^0\in L(-1)^2V_0.$
On the other hand,  $\Omega^0\in U_2$ and $U$ is $h_0$-stable, we have $h_0\Omega^0\in U_2.$
Putting those together, we have
\[
h_0\Omega^0
\in
(L(-1)^2V_0)\cap U_2.
\]

(i) Following the proof of Lemma \ref{lm4.8}, we can see that $\{T^M,H^M,A^M\}$ has the desired representations.

Now decompose $V_0$ into generalized $h_0$-eigenspaces $V_0=\bigoplus_{\lambda}V_0^{[\lambda]},$
Since $[h_0,L(-1)]=0,$ we have $L(-1)^2V_0^{[\lambda]}
\subseteq
V_2^{[\lambda]}.$
For $x\in(L(-1)^2V_0)\cap U_2,$
 write $x=L(-1)^2v$ with
$v=\sum_\lambda v^\lambda,
v^\lambda\in V_0^{[\lambda]},$
we have $L(-1)^2v^\lambda=0.$ Thus $x=L(-1)^2v^0$ with
$v^0\in V_0^{[0]}=U_0.$
This proves $(L(-1)^2V_0)\cap U_2\subseteq L(-1)^2U_0.$
The reverse inclusion is immediate, since $U$ is $L(-1)$-stable.
Hence $$(L(-1)^2V_0)\cap U_2= L(-1)^2U_0.$$Using $U_0=\C \textbf{1}$, we have $h_0\Omega^0=0.$

(ii) Since $h_0\Omega^0=0,$ this implies $h_0$ commute with $A^M,T^M,H^M.$ Thus
\begin{eqnarray*}
[H^M,A^M]&=&[L^M(0),A^M]-[h_0^M,A^M]\\
&=&[L^M(0),L^M(1)-2\bar h^M_1]\\
&=&-L^M(1)+2\bar h^M_1=-A^M.
\end{eqnarray*}
\begin{eqnarray*}
[H^M,T^M]&=&[L^M(0),T^M]-[h_0^M,T^M]\\
&=&[L^M(0),L^M(-1)]\\
&=&L^M(-1)=T^M.
\end{eqnarray*}
\begin{eqnarray*}
[A^M,T^M]&=&[L^M(1),L(-1)^M]-2[\bar h^M_1,L(-1)^M]\\
&=&2L^M(0)-2\overline{h}_0^M\\
&=&2H^M.
\end{eqnarray*}
Hence, $\{T^M,H^M,A^M\}$ satisfy M\"{o}bius relations.

(iii) Since $T^M=L(-1)^M$, the covariance relation for $T^M$ is trivial. The covariance relation for $H^M$ is
\begin{eqnarray*}
[H^M,Y_M(u,x)]&=&[L(0)^M-h_0^M,Y_M(u,x)]\\
&=&Y_M(L(0)u,x)+xY_M(L(-1)u,x)-Y_M(h_0u,x)\\
&=&
Y_M(Hu,x)+x\frac{d}{dx}Y_M(u,x).
\end{eqnarray*}
And the covariance relation for $A^M$ is
\begin{eqnarray*}
&&[A^M,Y_M(u,x)]=[L(1)^M,Y_M(u,x)]-2[\overline{h}_1^M,Y_M(u,x)]\\
&=&Y_M(L(1)u,x)+2xY_M(L(0)u,x)+x^2Y_M(L(-1)u,x)-2xY_M(\bar h_0u,x)-2Y_M(\bar h_1u,x).\\
&=&
Y_M(Au,x)
+
2xY_M(Hu,x)
+
x^2Y_M(L(-1)u,x).
\end{eqnarray*}

(iv) We prove $\Omega^0$ is a Virasoro element, since $\Omega^0\in U$ and $U$ is a vertex algebra, this is equivalent to show $$(\Omega^0)_j\Omega^0=0,j>3,~~ (\Omega^0)_3\Omega^0=\frac{c'}{2}\textbf{1},~~ (\Omega^0)_2\Omega^0=0,~~ (\Omega^0)_1\Omega^0=2\Omega^0,~~ (\Omega^0)_0\Omega^0=L(-1)\Omega^0.$$

Since $\Omega^0\in U_2$ and $h_0\Omega^0=0$, we have
$(\Omega^0)_1\Omega^0=H\Omega^0=(L(0)-h_0)\Omega^0=2\Omega^0.$ For $j>3$, we have
$(\Omega^0)_j\Omega^0\in U_{2+2-j-1}=0$ since $U$ is nonnegatively graded. Since $U_0=\C \textbf{1}$, we also have $(\Omega^0)_3\Omega^0=\frac{c'}{2}\textbf{1}$. Last relation is trivial. Hence we just show the relation for $(\Omega^0)_2\Omega^0.$ Using skew symmetry, we have
$$(\Omega^0)_2\Omega^0=\sum_{j\ge0}
\frac{(-1)^{j+3}}{j!}
L(-1)^j\bigl((\Omega^0)_{j+2}\Omega^0\bigr)=-(\Omega^0)_2\Omega^0.$$This implies $(\Omega^0)_2\Omega^0=0.$ Thus $\Omega^0$ is a Virasoro element.

This completes the proof.                          $\hfill\Box$

\subsection{Last step: relative inversion}\label{sec6.3}
Now, we start our last step.

\begin{definition}\label{def4.14}
Define $$C(z)=e^{zL(1)}e^{\pi iL(0)}z^{-2L(0)},\qquad C'(z)=e^{zA}e^{\pi iH}z^{-2H},\qquad E(z)=C(z)^{-1}C'(z).$$Since $A,L(1)$ are homogeneous of degree $-1$ and $H,L(0)$ are homogeneous of degree $0,$ we can see that $E^{\pm}$ are well-defined for any fixed $v\in V.$
\end{definition}

\begin{remark}
In the following, we use $C$ to denote the corresponding operators for M\"{o}bius triple $\{L(-1),L(0),L(1)\}$ and $C'$ for $\{L(-1),H,A\}$ on $V$. For a $V$-module $M$, the corresponding operators are $C_M$ for $\{L(-1),L(0),L(1)\}$ and $C'_M$ for $\{T^M,H^M,A^M\}.$
\end{remark}

\begin{lemma}\label{lm4.15}
Let $\{T^M,H^M,A^M\}$ denote the corrected M\"{o}bius triple obtained in Corollary \ref{co4.13}, and let $\{L(-1),L(0),L(1)\}$
denote the original M\"{o}bius triple.
For either of the two M\"{o}bius triples, write $C_{*,M}(z)$
for the corresponding operator on an ordinary $V$-module $M$.

Then
\begin{eqnarray}\label{id4.1}
C_{*,M}(z)Y_M(u,x)
=
Y_M\left(
C_{*,V}(z+x)u,
-\frac{x}{z(z+x)}
\right)
C_{*,M}(z).
\end{eqnarray}
Moreover, if $E_M(z)=C_M(z)^{-1}C'_M(z),$
then $E_M(z)Y_M(u,x)=Y_M(E_V(z+x)u,x)E_M(z),$
and
\[
\frac{d}{dz}E_M(z)
=
E_M(z)T^M-T^ME_M(z).
\]
\end{lemma}
\proof (i) Define $F(s)=
e^{sH^M}Y_M(u,x)e^{-sH^M}.$
Differentiating with respect to $s$, using Corollary \ref{co4.13}, we get
\begin{eqnarray*}
\frac{d}{ds}F(s)&=&e^{sH^M}[H^M,Y_M(u,x)]e^{-sH^M}\\
&=&e^{sH^M}\left(Y_M(Hu,x)+xY_M(L(-1)u,x)\right)e^{-sH^M}\\
&=&Y_M(e^{sH}Hu,e^sx)+xY_M(e^{sH}L(-1)u,e^sx)\\
&=&Y_M(e^{sH}Hu,e^sx)+xY_M(e^sL(-1)e^{sH}u,e^sx)
\end{eqnarray*}
On the other hand, define $G(s)=Y_M(e^{sH}u,e^sx).$
Differentiating, we get
\[
\begin{aligned}
\frac{d}{ds}G(s)
&=
Y_M(He^{sH}u,e^sx)+
e^sx
\frac{\partial}{\partial y}
Y_M(e^{sH}u,y)
\bigg|_{y=e^sx}\\
&=Y_M(He^{sH}u,e^sx)+
e^sx
Y_M(L(-1)e^{sH}u,e^sx)
\end{aligned}
\]
Thus $F$ and $G$ satisfy the
same differential equation and have the same initial value $F(0)=G(0)=Y_M(u,x).$
Therefore, $F=G$, i.e.,
\[
e^{sH^M}Y_M(u,x)e^{-sH^M}
=
Y_M(e^{sH}u,e^sx).
\]
In particular, we have
\begin{eqnarray}\label{id4.2}
e^{\pi iH^M}Y_M(u,x)e^{-\pi iH^M}=Y_M(e^{\pi iH}u,-x).
\end{eqnarray}
\begin{eqnarray}\label{id4.3}
z^{-2H^M}Y_M(u,x)z^{2H^M}
=
Y_M(z^{-2H}u,z^{-2}x).
\end{eqnarray}

(ii) Define $\phi_t(x)=\frac{x}{1-tx},B_t(x)=e^{t(1-tx)A}(1-tx)^{-2H}$, set $s=1-tx,
q=t(1-tx)=ts,$ we have $B_t=e^{qA}s^{-2H}.$ Furthermore, we get
\[
\frac{d}{dt}\phi_t(x)
=
\frac{x^2}{(1-tx)^2}
=
\phi_t(x)^2,
\]
\begin{eqnarray*}
\frac{d}{dt}B_t&=&q'Ae^{qA}s^{-2H}+e^{qA}\frac{2x}{s}Hs^{-2H}\\
&=&q'AB_t+\frac{2x}{s}(H+qA)B_t\\
&=&\left(q'+\frac{2xq}{s}\right)AB_t+2\frac{x}{s}HB_t\\
&=&AB_t+2\phi_t(x)HB_t.
\end{eqnarray*}

Now define $G(t)=Y_M(B_tu,\phi_t(x))$, from above discussion and Corollary \ref{co4.13}, we have
\[
\begin{aligned}
\frac{d}{dt}G(t)
&=
Y_M(AB_tu,\phi_t)
+
2\phi_tY_M(HB_tu,\phi_t)
+
\phi_t^2
\frac{\partial}{\partial y}
Y_M(B_tu,y)
\bigg|_{y=\phi_t}\\
&=Y_M(AB_tu,\phi_t)
+
2\phi_tY_M(HB_tu,\phi_t)
+
\phi_t^2Y_M(L(-1)B_tu,\phi_t)\\
&=[A^M,G(t)].
\end{aligned}
\]
On the other hand, define $F(t)=e^{tA^M}Y_M(u,x)e^{-tA^M}$, using Corollary \ref{co4.13}, we have $$\frac{d}{dt}F(t)=[A^M,F(t)].$$ Thus $F,G$ satisfy same differential equation and $F(0)=G(0)=Y_M(u,z).$ Hence we get $F=G$, i.e.,
\begin{eqnarray}\label{id4.4}
e^{tA^M}Y_M(u,x)e^{-tA^M}
=
Y_M\left(
e^{t(1-tx)A}(1-tx)^{-2H}u,
\frac{x}{1-tx}
\right).
\end{eqnarray}

(iii) For the chosen M\"{o}bius triple $\{A^M,H^M,T^M\},$ we have
$C_{*,M}(z)=e^{zA^M}e^{\pi iH^M}z^{-2H^M}.$ Using identities (\ref{id4.2}), (\ref{id4.3}), (\ref{id4.4}), we get
\begin{eqnarray*}
C_{*,M}(z)Y_M(u,x)C_{*,M}(z)^{-1}&=&e^{zA^M}Y_M\left(e^{\pi iH}z^{-2H}u,-z^{-2}x\right)e^{-zA^M}\\
&=&Y_M\left(C_{*,V}(z+x)u,-\frac{x}{z(z+x)}\right).
\end{eqnarray*}
This proves identity (\ref{id4.1}) for the M\"{o}bius triple $\{A^M,H^M,T^M\}.$ Similarly, we have the identity (\ref{id4.1}) for the M\"{o}bius triple $\{L(1),L(0),L(-1)\}.$

(iv) From above computation, let $E_M=C^{-1}_MC'_M$, using identity (\ref{id4.1}), we have
\[
\begin{aligned}
E_M(z)Y_M(u,x)
&=
C_M(z)^{-1}
Y_M(C'_V(z+x)u,-\frac{x}{z(z+x)})
C'_M(z)\\
&=Y_M(C^{-1}_V(z+x)C'_V(z+x)v,x)
C_M(z)^{-1}C'_M(z)\\
&=Y_M(E_V(z+x)u,x)E_M(z).
\end{aligned}
\]

(v) It is trivial that $\frac{d}{dz}C_M(z)=e^{zL(1)}
(L(1)-2z^{-1}L(0))
e^{\pi iL(0)}z^{-2L(0)}.$
On the other hand, using bracket relations and the Baker-Campbell-Hausdorff formula, we have
\[
\begin{aligned}
&C_M(z)L(-1)+z^{-2}L(-1)C_M(z)\\
=& -z^{-2}e^{zL(1)}L(-1) e^{\pi iL(0)}z^{-2L(0)} +z^{-2}e^{zL(1)}
(L(-1)-2zL(0)+z^2L(1))e^{\pi iL(0)}z^{-2L(0)}     \\
=&
e^{zL(1)}
(L(1)-2z^{-1}L(0))
e^{\pi iL(0)}z^{-2L(0)}.
\end{aligned}
\]This proves
\[
\frac{d}{dz}C_M(z)
=C_M(z)L(-1)+z^{-2}L(-1)C_M(z).
\]
Similarly, we have the same identity for $C'_M$
\[
\frac{d}{dz}C'_M(z)
=
C'_M(z)T^M+z^{-2}T^MC'_M(z).
\]

Putting those identities together, we get
\[
\begin{aligned}
\frac{d}{dz}E_M
&=-C_M^{-1}(\frac{d}{dz}C_M)C_M^{-1}C'_M+C_M^{-1}(\frac{d}{dz}C'_M)\\
&=-C_M^{-1}(C_ML(-1)+z^{-2}L(-1)C_M)E_M+C_M^{-1}(C'_MT^M+z^{-2}T^MC'_M)\\
&=E_M(z)T^M-T^ME_M(z).
\end{aligned}
\]Last identity holds since $T^M=L(-1)$ on $M$. This is equivalent to
\[
[T^M,E_M(z)]
=
-\frac{d}{dz}E_M(z).
\] Thus we complete the proof.                  $\hfill\Box$

\begin{remark}
We should notice that $H$ need not yet be semisimple.
Let $H=H_s+H_n$
be the Jordan decomposition, then
\[
z^{-2H}
=
z^{-2H_s}
e^{-2(\log z)H_n}
\]
is well-defined since $H_n$ is nilpotent.
Moreover, since $A$ is homogeneous of degree $-1$ and $V$ is lower truncated, $e^{zA}u$ is also well-defined.

Consequently, all expressions above are well-defined and all identities are valid before the semisimplicity of $h_0$ is established. \end{remark}

\begin{proposition}\label{prop4.18}
Let notations be defined in above subsections. Then the operator $2h_0$ is semisimple on $V$ and has integral eigenvalues.
Consequently, for every $v\in V$, we have $E(z)v,\ E(z)^{-1}v\in V[z,z^{-1}]$.

Recall $K=DV=V'$
denotes the restricted contragredient module, define $K^E=K$ and
\[
Y_{K^E}(u,z)
=
Y_K(E(z)u,z),
\]
then $(K^E,Y_{K^E})$ is a weak $V$-module and $K^E\cong V$ as $V$-modules.
\end{proposition}
\proof (i) Recall that $K=V'=\bigoplus_{n}V_n^*$, for $v,u\in V$ and $k\in K$, define
\[
( Y'(v,z)k,u)
=
(
k,
Y(C'(z)v,z^{-1})u
).
\]
The original contragredient structure on $K$ is characterized by
\[
(
Y_K(a,z)k,u
)
=
(
k,
Y(C(z)a,z^{-1})u
).
\]
We use $Y'$ and $Y_K$ respectively. Now we have
\[
\begin{aligned}
(
Y_K(E(z)v,z)k,u
)
=
(
k,
Y(C(z)E(z)v,z^{-1})u
)
=
(
k,
Y(C'(z)v,z^{-1})u
)
=
(
Y'(v,z)k,u
).
\end{aligned}
\]
Putting together, we get
\begin{eqnarray}\label{id4.6.1}
Y'(v,z)=Y_K(E(z)v,z),
\end{eqnarray}
This implies $Y'$ is precisely the original contragredient field twisted by
the relative inversion operator $E(z)$.

For $v\in V_n$, suppose that $v$ belongs to the
generalized $h_0$-eigenspace of eigenvalue $\lambda$, i.e., $(h_0-\lambda)^Nv=0$
for some $N$. Then $v$ belongs to the generalized $H$-eigenspace of
eigenvalue $n-\lambda.$ Therefore, there is a polynomial $p(x)$, such that
\[
z^{-2H}v
=
z^{-2(n-\lambda)}p(\log z)v.
\]

Let $k\in V_r^*\subset K,u\in V_d,$ we have
\begin{eqnarray}\label{id4.5}
(Y'(v,z)k,u)&=&(k,Y(C'(z)v,z^{-1})u)\nonumber\\
&=&(k,e^{n\pi i}\sum_{j\geq0,\ell\in\Z}\frac{1}{j!}v^{n-j}_\ell uz^{\ell+1+j-2(n-\lambda)}p(\log z))\nonumber\\
&=&(k,e^{n\pi i}\sum_{j\geq0}\frac{1}{j!}v^{n-j}_{d+n-j-r-1} uz^{d-n-r+2\lambda}p(\log z)),
\end{eqnarray}where $v^{n-j}\in V_{n-j}.$
Consequently, the component of $Y'(v,z)k$ lying in $V_d^*$ has the
form
\[
z^{\,d-n-r+2\lambda}
\,
p(\log z).
\]

Since $d\ge -s,$
where $-s$ is the lower bound of the grading of $V$, the powers of
$z$ are lower bounded.  Moreover, $h_0$ has only finitely
many generalized eigenvalues. Thus $Y'(v,z)k$ is
coefficientwise well defined and involves only finitely many exponent
cosets modulo $\mathbb Z$.

(ii) Let $P$ acting on $K$ be defined by $(Pk,u)=(k,Au)$, we can see $P$ is homogeneous of degree $1$. Now using bracket relations and Corollary \ref{co4.13}, we have
\begin{eqnarray*}
(\frac{d}{dz}Y'(v,z)k,u)&=&(k,\frac{d}{dz}Y(C'(z)v,z^{-1})u)\\
&=&(k,Y(\frac{d}{dz}C'(z)v,z^{-1})u)-z^{-2}(k,\frac{d}{dy}Y(C'(z)v,y)\mid_{y=z^{-1}}u)\\
&=&(k,Y(-AC'(z)-2z^{-1}HC'(z)v,z^{-1})u)-z^{-2}(k,Y(L(-1)C'(z)v,z^{-1})u)\\
&=&(k,-[A,Y(C'(z)v,z^{-1})]u)=(-[Y'(v,z),P]k,u)\\
&=&([P,Y'(v,z)]k,u).
\end{eqnarray*}
Furthermore, from Lemma \ref{lm4.8}, we have $Pf=0$. Hence,
\[
\begin{aligned}
\frac{d}{dz}
\left(
e^{-zP}Y'(v,z)f
\right)
&=
-Pe^{-zP}Y'(v,z)f
+
e^{-zP}\partial_zY'(v,z)f
\\
&=
-Pe^{-zP}Y'(v,z)f
+
e^{-zP}
(PY'(v,z)-Y'(v,z)P)f
\\
&=
0.
\end{aligned}
\]
A function of the form $z^\alpha p(\log z)$ has zero differentiation if and only if $\alpha=0$ and $p$ is constant. Thus differentiation on
$z^\alpha\mathbb C[\log z]$
has only constants in its kernel.
Thus $e^{-zP}Y'(v,z)f$ is a constant, denoted by $\phi(v).$
Therefore,
\begin{eqnarray}\label{id4.6}
 Y'(v,z)f=e^{zP}\phi(v).
\end{eqnarray}
Moreover, $e^{zP}$ is well-defined, above identity is a legitimate identity and $Y'(v,z)f$ contains no logarithms and only nonnegative integral powers of $z$. Furthermore, it is obvious that $f=\phi(\textbf{1})$.

Now assume again that $v\in V_n$
has generalized $h_0$-eigenvalue $\lambda$. Let $k=f\in V_0^*$, from identity (\ref{id4.5}), we know the component of $Y'(v,z)f$ lying in $V_d^*$ is $z^{\,d-n+2\lambda}
\,p(\log z).$
Comparing the constant term of identity (\ref{id4.6}), we get $d-n+2\lambda=0,$ i.e.,
\begin{eqnarray}\label{id4.7}
d=n-2\lambda.
\end{eqnarray}
This also means $\phi(v)$ is supported only in original dual degree $n-2\lambda$. In particular, $\phi(v)$ belongs to the restricted dual $K$. Hence we get a linear map $$\phi:V\rightarrow K.$$

(iii) Set $I_f=\left\{v\in V\mid Y_K(E(z)v,z)f=0\right\}=
\left\{v\in V\mid Y'(v,z)f=0\right\}.$ Using identity (\ref{id4.6}), we have
\[
I_f=\ker\phi.
\]

On the other hand, for $v\in I_f,u\in V,$ since $E$ is well-defined, let
\[
E(z_1)u
=
\sum_{i=1}^N a_i(z_1)u_i,
\qquad
E(z_2)v
=
\sum_{j=1}^M b_j(z_2)v_j,
\]
with $u_i,v_j\in V$.
Choosing one integer $N_0$ large enough to make sure the weak associativity of $Y_K$ holds, and another integer $N_1$ large enough such that $u_nv=0$ for $n>N_1-1$,  using Lemma \ref{lm4.15}, we have
\begin{eqnarray*}
0&=&x_0^{N_1}(x_0+x_2)^{N_0}Y_K(E(x_0+x_2)u,x_0+x_2)Y_K(E(x_2)v,x_2)f\\
&=&x_0^{N_1}(x_0+x_2)^{N_0}Y_K\left(Y(E(x_0+x_2)u,x_0)E(x_2)v,x_2\right)f\\
&=&x_0^{N_1}(x_0+x_2)^{N_0}Y_K(E(x_2)Y(u,x_0)v,x_2)f\\
&=&(x_0+x_2)^{N_0}\sum_nx_0^{N_1-n-1}Y_K(E(x_2)u_nv,x_2)f.
\end{eqnarray*}
Hence, comparing the coefficient of $x_0^k$ repeatedly, we have $u_nv\in I_f$ for all $u\in V,\ v\in I_f,\ n\in\mathbb Z.$ Thus $I_f$ is an ideal of $V$.

Notice also that $\mathbf1\notin I_f,$
since $Y_K(\mathbf1,z)f=f\neq0.$
Therefore $I_f$ is a proper ideal. Thus $I_f=0$ since $V$ is simple. This means $\phi$ is injective.

(iv) The only possible multivaluedness in $C'(z)$ comes from $z^{-2H}=z^{-2(L(0)-h_0)}.$
When $z$ travels once counterclockwise around the origin $\log z\longmapsto\log z+2\pi i.$
Therefore, we get
\[
z^{-2H}
\longmapsto
z^{-2H}e^{-4\pi iH}=z^{-2H}e^{4\pi ih_0}.
\]
Since $h_0\Omega^0=0$ by Corollary \ref{co4.13}, we have
$\mathsf M C'(z)v=C'(z)e^{4\pi ih_0}v,$ where \(\mathsf M\) is counterclockwise local monodromy around
\(z=0\). This induces
\begin{eqnarray}\label{id4.8}
\mathsf M Y'(v,z)=Y'(e^{4\pi ih_0}v,z).
\end{eqnarray}
On the other hand, from identity (\ref{id4.6}), we have
$\mathsf M \left(Y'(v,z)f\right)=Y'(v,z)f.$
Putting those together, we have $Y'((e^{4\pi ih_0}-1)v,z)f=0.$
This implies $(e^{4\pi ih_0}-1)v\in I_f.$
By (iii), we have
\begin{eqnarray}\label{id4.9}
e^{4\pi ih_0}
=
\Id_V.
\end{eqnarray}

Fix $n$ and take the Jordan decomposition $h_0=S+N$ on $V_n$, where $S$ is semisimple and $N$ is nilpotent, we also have $[S,N]=0.$
From identity (\ref{id4.9}), we have $e^{4\pi iS}e^{4\pi iN}=I.$
Let $\lambda$ be an eigenvalue of $S$, restricting to the generalized
$\lambda$-eigenspace, we have $e^{4\pi i\lambda}=1.$
Therefore
\[
2\lambda\in\mathbb Z.
\]
Then removing the semisimple part, we have $e^{4\pi iN}=I.$
But
\[
e^{4\pi iN}-I
=
N
\left(
4\pi i I
+
\frac{(4\pi i)^2}{2!}N
+
\frac{(4\pi i)^3}{3!}N^2
+\cdots
\right).
\]
But $\left(
4\pi i I
+
\frac{(4\pi i)^2}{2!}N
+
\frac{(4\pi i)^3}{3!}N^2
+\cdots
\right)$ is invertible, we have $N=0.$
Hence $h_0$ is semisimple on every $V_n$.
This implies that $2h_0$ is semisimple on $V$, and every eigenvalue
is an integer.

(v) Again let $v\in V_n$ satisfy $h_0 v=\lambda v.$ Then $Hv=(n-\lambda)v.$
Thus
\begin{eqnarray*}
E(z)v&=&z^{2L(0)}e^{-\pi iL(0)}e^{-zL(1)}e^{zA}e^{\pi iH}z^{-2H}v\\
&=&z^{-2(n-\lambda)}e^{\pi i(n-\lambda)}z^{2L(0)}e^{-\pi iL(0)}\sum_{q\geq0,p\geq0}\frac{(-1)^q}{q!}\frac{1}{p!}L(1)^qA^pvz^{p+q}\\
&=&z^{-2(n-\lambda)}e^{\pi i(n-\lambda)}\sum_{q\geq0,p\geq0}z^{2(n-p-q)}e^{-\pi i(n-p-q)}\frac{(-1)^q}{q!}\frac{1}{p!}L(1)^qA^pvz^{p+q}\\
&=&\sum_{j\geq0}w_{n-j}z^{2\lambda-j}e^{\pi i(j-\lambda)},
\end{eqnarray*}where $w_{n-j}\in V_{n-j}$.
Now since $2\lambda\in\Z$ and $V$ is lower truncated, we have
\[
E(z)v\in V[z,z^{-1}].
\]
Similarly, we also have $E(z)^{-1}v\in V[z,z^{-1}].$

(vi) Now define a new vertex-operator action on the vector space $K$ by
\[
Y_{K^E}(u,x)
=
Y_K(E(x)u,x).
\]
Since $E(x)u$ is now a finite Laurent polynomial with coefficients
in $V$, the right-hand side is an ordinary lower-truncated formal
field. It is trivial that $E(x)\mathbf1=\mathbf1,$ thus
$Y_{K^E}(\mathbf1,x)=
\Id_K.$

 By Lemma~\ref{lm4.15}, for sufficiently large $N_0$ and $g\in K$, using weak associativity for $V$-module $K$ and part (v), we have
\[
\begin{aligned}
&\qquad (x_0+x_2)^{N_0}Y_K(E(x_0+x_2)u,x_0+x_2)
Y_K(E(x_2)v,x_2)g\\
&=
(x_0+x_2)^{N_0}Y_K\left(
Y(E(x_0+x_2)u,x_0)E(x_2)v,
x_2
\right)g\\
&=(x_0+x_2)^{N_0}Y_K(E(x_2)Y(u,x_0)v,x_2)g.
\end{aligned}
\]
This proves the weak associativity for $Y_K$, thus for $Y_{K^E}$, i.e., modified fields satisfy weak associativity. Using the same argument, we also have weak commutativity, and hence the Jacobi identity.
Therefore, $(K^E,Y_{K^E})$ is a weak $V$-module.

Suppose $N\subset K^E$ is a nonzero $K^E$-submodule.
For any $v\in V$, write $E(x)^{-1}v=
\sum_{j=r}^{s}x^jv^j$,
we have
\[
\begin{aligned}
Y_K(v,x)
=
Y_K(E(x)E(x)^{-1}v,x)
=
\sum_{j=r}^{s}
x^jY_K(E(x)v^j,x)
=
\sum_{j=r}^{s}
x^jY_{K^E}(v^j,x).
\end{aligned}
\]
Since $N$ is stable under all coefficients of
$Y_{K^E}(v^j,x)$, we can see that $N$ is also stable
under all coefficients of the original fields $Y_K(v,x)$.
Thus $N$ is a $V$-submodule of ordinary contragredient module $K$.

The restricted contragredient $K=V'$ of the simple adjoint module
$V$ is simple. Hence $N=K.$
Therefore $K^E$ is simple under the new vertex operator action.

(vii) Recall from part (iii) that we already have an injective linear map $\phi:V\rightarrow K.$ From identity (\ref{id4.6.1}) and identity (\ref{id4.6}), we have
\[
Y_{K^E}(v,z)f
=
e^{zP}\phi(v).
\]

Moreover, from part (ii) and identity (\ref{id4.6.1}), we have $[P,Y_{K^E}(u,z)]=
\frac{d}{dz}Y_{K^E}(u,z).$
Therefore, we get
\[
Y_{K^E}(u,x+z)e^{zP}
=
e^{zP}Y_{K^E}(u,x).
\]
Applying weak associativity in $K^E$ to $u,v$ and the vector $f$.
for sufficiently large $N_0$, using identity (\ref{id4.6}), we have
\begin{eqnarray*}
(x+x_2)^{N_0}Y_{K^E}(Y(u,x)v,x_2)f
&=&(x+x_2)^{N_0}Y_{K^E}(u,x+x_2)Y_{K^E}(v,x_2)f\\
&=&(x+x_2)^{N_0}Y_{K^E}(u,x+x_2)e^{x_2P}\phi(v)\\
&=&(x+x_2)^{N_0}e^{x_2P}Y_{K^E}(u,x)\phi(v).
\end{eqnarray*}
Taking $\Res_{x_2}x_2^{-1}$, we get
\[
Y_{K^E}(u,x)\phi(v)
=
\phi(Y(u,x)v).
\]
Thus $\phi$ is a $V$-module homomorphism. Since $\phi(\textbf{1})=f$, it is nonzero. Furthermore, $V$ and $K^E$ are simple, we have $\phi$ is isomorphic.
This completes the proof.             $\hfill\Box$

\begin{remark}\label{re4.19}
Now we have $h_0$ is semisimple on $V$, so is $H$. Let $v\in V_n$ and $h_0v=\lambda v,$ under isomorphism $\phi$, we have $\phi(v)$ has a new weight $n'=n-2\lambda$ on $K^E=K=V'.$ From part (1) of the proof Proposition \ref{prop4.18}, we have $n'\geq -s.$ Furthermore, the $H$-eigenvalue $t$ of $v$ satisfies $$t=\frac{n+n'}{2}\geq -s,\qquad -s\leq n\leq 2t+s.$$Hence, the new grading is also lower truncated and each homogeneous space is finite dimensional.
\end{remark}

\begin{remark}\label{re4.20}
The operator $E(z)$ modifies the contragredient $V$-module structure on $V'$ and makes it isomorphic to $V$ as $V$-modules.
\end{remark}

\begin{proposition}\label{prop4.21}
Let $K=V'$
be the ordinary contragredient $V$-module. Then there exists a simple ordinary $V$-module $B$ and a
nonzero intertwining operator $I^E$ of type $\binom{V}{K~~B}$.
Consequently, $\Hom_{\mathcal C}(K\boxtimes B,V)\neq0.$
\end{proposition}
\proof (i) Let $B=V^E=V$, and $I\in I_V\binom{K}{K~~V}$ a nonzero intertwining operator, for $u\in V,w\in K,v\in B=V$, define
\[
Y_B(u,z)=Y(E(z)u,z),\qquad I^E(w,z)v=I(E_K(z)w,z)v.
\]
Note that until now, we just know $h_0$ is semisimple on $V$, and we do not know whether it is semisimple on $K$. But following the proof of Proposition \ref{prop4.18}, we can see $E_K(z)$ is still well-defined on $K$. This implies $I^E$ may contain polynomial of $\log z$.

(ii) By Proposition \ref{prop4.18}, $E(x)u\in V[x,x^{-1}]$ for every $u\in V$. Hence $Y_B(u,x)v$ is lower truncated for every $u,v\in V$. Moreover, $E(x)\mathbf1=\mathbf1,$
and therefore $Y_B(\mathbf1,x)=\Id_V.$

Again using Jacobi identity for $Y$ and Lemma \ref{lm4.15}, for $u^1,u^2\in V,v\in B=V$, we have
\[
\begin{aligned}
&x_0^{-1}
\delta\left(\frac{x_1-x_2}{x_0}\right)
Y(E(x_1)u^1,x_1)Y(E(x_2)u^2,x_2)v
\\
&\quad
-
x_0^{-1}
\delta\left(\frac{x_2-x_1}{-x_0}\right)
Y(E(x_2)u^2,x_2)Y(E(x_1)u^1,x_1)v
\\
&=
x_2^{-1}
\delta\left(\frac{x_1-x_0}{x_2}\right)
Y(Y(E(x_1)u^1,x_0)E(x_2)u^2,x_2)v\\
&=x_2^{-1}
\delta\left(\frac{x_1-x_0}{x_2}\right)
Y(E(x_2)Y(u^1,x_0)u^2,x_2)v
\end{aligned}
\]
By the definitions, this identity becomes
\[
\begin{aligned}
&x_0^{-1}
\delta\left(\frac{x_1-x_2}{x_0}\right)
Y_B(u^1,x_1)Y_B(u^2,x_2)v
\\
&\quad
-
x_0^{-1}
\delta\left(\frac{x_2-x_1}{-x_0}\right)
Y_B(u^2,x_2)Y_B(u^1,x_1)v
\\
&=
x_2^{-1}
\delta\left(\frac{x_1-x_0}{x_2}\right)
Y_B(Y(u^1,x_0)u^2,x_2)v.
\end{aligned}
\]
Thus $B$ is a weak $V$-module.

Let $N\subseteq B$ be a nonzero submodule. For every $v\in V$, set $E(z)^{-1}v=\sum_{j=r}^{s}z^jv^j,$ we have
\[
\begin{aligned}
Y(v,z)
&=
Y(E(z)E(z)^{-1}v,z)
\\
&=
\sum_{j=r}^{s}
z^jY(E(z)v^j,z)
\\
&=
\sum_{j=r}^{s}
z^jY_B(v^j,z).
\end{aligned}
\]
Thus every coefficient of every original field $Y(v,z)$ preserves
$N$, and $N$ is also an original $V$-submodule of the adjoint
module $V$. Since $V$ is simple, we also have $B$ is simple. Using Corollary \ref{co3.6}, $B$ is also ordinary.

(iii) Let $I\in I_V\binom{K}{K~~V}$ be an intertwining operator, since $E$ and $E_K$ are well-defined and lower
truncated, we have $I^E$ is lower truncated and finite logarithmic
degree.

By definition, using $L(-1)$-bracket for $I$ and Lemma \ref{lm4.15}, we have
\begin{eqnarray*}
\frac{d}{dz}I^E(w,z)&=&\frac{d}{dz}I(E_K(z)w,z)\\
&=&I(L(-1)E_K(z)w,z)+I(E'_K(z)w,z)\\
&=&I(E_K(z)L(-1)w,z)=I^E(L(-1)w,z).
\end{eqnarray*}This proves $L(-1)$-bracket for $I^E.$
Following the proof of part (ii), we can see $I^E\in I_V\binom{K^E}{K\;\;B}$ is a logarithmic intertwining operator. From Proposition \ref{prop4.18}, we know $K^E\cong V$ as $V$-modules. Thus $I^E$ induces a logarithmic intertwining operator of type $\binom{V}{K~~B}$ under isomorphism $\phi$, still denoted by $I^E.$

Now, we have
\[
\begin{aligned}
I^E(w,z)\mathbf1
=
I(E_K(z)w,z)\mathbf1
=
e^{zL(-1)}E_K(z)w.
\end{aligned}
\]
If $w\neq0$, then $e^{zL(-1)}E_K(z)w\neq0$ since $e^{zL(-1)}E_K(z)$ is invertible. Thus $I^E\neq0.$

(iv) Let $w\in K,v\in B$ are homogeneous of weight $r,s$ respectively,
and suppose that a component of $I^E(w,z)v$ lies in the
$L(0)^{K^E}$-eigenspace of weight $q$. Then, such a coefficient may have form
\[
z^\alpha p(\log z),
\qquad
p\in\mathbb C[t].
\]
Applying $L(0)$-bracket for $I^E$ to $v$ and projecting to  weight $q$ space, we have
\[
(q-r-s)
z^\alpha p(\log z)
=
z\frac{d}{dz}
\left(
z^\alpha p(\log z)
\right)=z^\alpha
\left(
\alpha p(\log z)+p'(\log z)
\right).
\]
Therefore, we get
\[
(q-r-s-\alpha)p(t)=p'(t).
\]
This implies $p(t)=0$ or $q-r-s-\alpha=0$, and $p(t)$ is a constant. In either case, $I^E$ contains no powers of $\log z$.
Thus $I^E$ is in fact an ordinary intertwining operator without logarithmic terms.

(v) From part (iii) and (iv), we know $I^E$ is a nonzero ordinary intertwining operator of type $\binom{V}{K~~B}$. Thus we have a nonzero morphism
\[
\Hom_{\mathcal C}(K\boxtimes B,V)\neq0
\]by the universality of tensor product,
This completes the proof.                  $\hfill\Box$

\subsection{Main Theorem}\label{sec6.4}
Now we are ready to prove our Main Theorem of this article.
\begin{maintheorem}\label{MThm6.22}
Let $V$ be a simple, $C_2$-cofinite lower truncated vertex operator algebra such that the complete trace functions satisfy
\begin{eqnarray*}
SF_V=\sum_{j=0}^r s_jF_{W^j}.
\end{eqnarray*} Then $V$ is rational.
\end{maintheorem}
\proof By Remark \ref{RE5.12}, we just need to find a $V$-module $B$ satisfying desired properties. By Proposition \ref{prop4.21}, that $B$ does exist, which is exactly the modified $V$-module $V^E.$ Thus $V$ is rational.
$\hfill\Box$

\begin{remark}
This Main Theorem is exactly the inverse theorem of Zhu's modular invariance theorem proved in \cite{Zhu}.
\end{remark}

\begin{remark}\label{re6.24}
In \cite{DongRen}, Dong-Ren conjectured that for a simple and $C_2$-cofinite vertex operator algebra $V$, the rationality is equivalent to the modular invariance of the space spanned by the complete one point trace functions of irreducible ordinary $V$-modules. By our Main Theorem \ref{MThm6.22} and Zhu's modular invariance theorem, we can see this conjecture is true.
\end{remark}

\begin{remark}\label{RE5.13}
The hypothesis used throughout this paper is the single complete identity \eqref{Last}, not the whole modular invariance spanned by the space of all irreducible complete one point trace functions. However, after the
rigidity has been established, equation \eqref{ID5.23} shows that every irreducible one-point trace function has its
$S$-transformation in the space spanned by ordinary trace functions, and that all coefficients $s_j$ in the assumed expansion
are nonzero. Thus modular invariance of the whole trace space is recovered from the $S$-transformation of
the adjoint trace.
\end{remark}

\appendix
\section{Appendix}\label{seca}
In this appendix, we prove Theorem \ref{thm3.7}.

\proof (i) Since $W_0\in \mathcal {C}$, it is lower truncated and each homogeneous space is finite dimensional. Thus $\Phi$ is well-defined. Since $V$ is $C_2$-cofinite, by Theorem \ref{thm3.4}, every object of
\(\mathcal C\) is grading-restricted, \(C_1\)-cofinite, and finitely
generated. Now from \cite{Fiordalisi,Huang-C1}, there exist constants $C,N_1$ such that $\tr_{(W_0)_{[\lambda]}}(\mathcal {Y}_1\cdots\mathcal {Y}_n)\leq C(1+|\lambda|)^{N_1}$ on the ordered genus-one product region $\mathcal {R}$. This implies $$|\Phi|=\sum_{\lambda}|\tr_{(W_0)_{[\lambda]}}(\mathcal {Y}_1\cdots\mathcal {Y}_nq^{L(0)-\frac{c}{24}})|\leq\sum_\lambda C(1+|\lambda|)^{N_1}|q|^{\lambda-\frac{c}{24}}.$$ Thus $\Phi$ converges absolutely since $|q|<1$.

Let $\mathcal {K}\subset\mathcal {R}$ be a compact subset, there is $r_0<1$ such that $|q|<r_0$ for all elements of $\mathcal {K}$. Hence $$|\Phi|\leq\sum_\lambda C(1+|\lambda|)^{N_1}r_0^{\lambda-\frac{c}{24}}.$$Thus $\Phi$ converges locally uniformly on $\mathcal {R}$.

(ii) The collision points are precisely the points where two insertion
variables represent the same point on the torus $E_\tau=\mathbb C/(\mathbb Z+\mathbb Z\tau),$ i.e., $z_i-z_j\in\mathbb Z+\mathbb Z\tau,$
and hence the collision divisor is exactly \(D_\tau\). By part (i), \(\Phi\) is absolutely and locally uniformly
convergent on the ordered product region \(\mathcal R\). Hence it
defines a holomorphic function on \(\mathcal R\).

Now let $\gamma:[0,1]\longrightarrow \mathbb C^n\times\mathbb H\setminus D_\tau$ be a path avoiding the collision divisor. Starting from a local
holomorphic branch of the trace function on \(\mathcal R\), $\Phi$ can be analytically continued along \(\gamma\). Hence it extends to a multivalued analytic function on $\mathbb C^n\times\mathbb H\setminus D_\tau$.

The analytic continuation may depend on the homotopy class of the
chosen path. A closed loop therefore induces an action on the space of
analytic branches, namely the monodromy action. The explicit form of
this local monodromy will be described using the collision expansions
in part (iii).

On any simply connected open subset of
\(\mathbb C^n\times\mathbb H\setminus D_\tau\), after fixing a branch of analytic continuation,
the resulting function is single-valued and holomorphic.

(iii) Now fix two insertion variables and introduce the collision coordinate $z=z_i-z_j .$ For fixed \(\tau\in\mathbb H\) and $\rho_\tau$, the punctured disk $0<|z|<\rho_\tau$ contains no collision point other than \(z=0\).

Let $\mathcal Y$ be the logarithmic intertwining operator corresponding to the two colliding insertions. By the definition of logarithmic intertwining
operators, we have a local expansion
\[
\mathcal Y(\widetilde{w},z)
=
\sum_{r\in\mathbb C}
\sum_{k=0}^{K}
\widetilde{w}_{r,k}z^{-r-1}(\log z)^k,
\]where $\widetilde{w}=\mathcal {U}(q_z)w.$
Hence all possible multivalued terms appearing in the collision
variable are generated by complex powers and logarithmic powers of
\(z\).

The powers of \(z\) arising from the \(L(0)\)-grading of the
intermediate module can be analyzed by the operator \(z^{L(0)}\). To describe the contribution of the generalized \(L(0)\)-eigenspaces,
let \(W\) be the intermediate generalized module appearing in the
collision expansion. For a generalized weight vector $w\in W_{[\lambda]},$
we have $(L(0)-\lambda)^Nw=0$ for some \(N\). Writing $L(0)=\lambda+N_w,$ where $N_w=L(0)-\lambda$ is nilpotent, we obtain
\[
\begin{aligned}
z^{L(0)}w
&=
\exp\bigl((\log z)L(0)\bigr)w
\\
&=
\exp\bigl((\log z)(\lambda+N_w)\bigr)w
\\
&=
\exp(\lambda\log z)
\exp\bigl((\log z)N_w\bigr)w
\\
&=
z^\lambda
\sum_{k=0}^{N-1}
\frac{(\log z)^k}{k!}
N_w^k w .
\end{aligned}
\]
Therefore every generalized weight contribution has the form $z^\lambda(\log z)^k$ with only finitely many possible logarithmic degrees.

Using $L(0)$-bracket with intertwining operators and Proposition \ref{prop3.2}, only finitely many exponent classes modulo
\(\mathbb Z\) occur in the local expansion since the \(L(0)\)-bracket relation implies that the exponent of \(z\)
is determined by the generalized weights and Proposition \ref{prop3.2} implies the generalized weights are contained in finitely many cosets.
Therefore, there exist
finitely many complex numbers $\lambda_1,\ldots,\lambda_{N_0}$ and non-negative integers $K_1,\ldots,K_{N_0}$
such that the analytically continued genus-one trace function admits
the expansion
\[
\Phi(z,\tau)
=
\sum_{\alpha=1}^{N_0}
\sum_{k=0}^{K_\alpha}
z^{\lambda_\alpha}
(\log z)^k
F_{\alpha,k}(z,\tau),
\]
where the coefficient functions $F_{\alpha,k}(z,\tau)$
contain all remaining dependence on the other insertion variables,
the inserted vectors, and the modular parameter \(\tau\). By the genus-one convergence and extension results for logarithmic
intertwining operators \cite{Fiordalisi}, those coefficient functions are
single-valued and holomorphic on $0<|z|<\rho_\tau .$ Thus all possible local multivaluedness around the collision point is contained in the explicit factors $z^{\lambda_\alpha}(\log z)^k .$

Since \(F_{\alpha,k}\) are single-valued holomorphic functions on the
punctured disk, they admit convergent Laurent expansions
\[
    F_{\alpha,k}(z,\tau)
    =
    \sum_{m\in\mathbb Z}
    a_{\alpha,k,m}(\tau)z^m
\]
on this domain.

(iv) We prove the equality between the product and iterate expansions.
By Theorem \ref{thm3.4}, the category \(\mathcal C\) carries a vertex
tensor category structure. In particular, the associativity
isomorphism induces the fusing isomorphism between spaces of
logarithmic intertwining operators.

Let $\mathcal Y_1\in I\binom{W_0}{A_1\,W_1},\mathcal Y_2 \in I \binom{W_1}{A_2\,W_0}$
be the intertwining operators. By the associativity isomorphism of the vertex tensor category, there exist intertwining operators $\mathcal Y_3,\mathcal Y_4$ determined by the pair $(\mathcal Y_1,\mathcal Y_2),$
such that $(\mathcal Y_3,\mathcal Y_4)=
\mathcal A(\mathcal Y_1,\mathcal Y_2),$
where $\mathcal A$ denotes the fusing isomorphism.
The product trace function and iterate expansion are defined by
\begin{eqnarray*}
\Phi_{\mathrm{prod}}&=&\operatorname{tr}_{W_0}
\left(
\mathcal Y_1
\left(
\mathcal U(q_{z_1})a_1,q_{z_1}
\right)
\mathcal Y_2
\left(
\mathcal U(q_{z_2})a_2,q_{z_2}
\right)
q^{L(0)-c/24}
\right),
\\
\Phi_{\mathrm{iter}}&=&\operatorname{tr}_{W_0}
\left(
\mathcal Y_3
\left(
\mathcal U(q_{z_2})
\mathcal Y_4
\left(
a_1,e^{\log(z_1-z_2)}
\right)a_2,
q_{z_2}
\right)
q^{L(0)-c/24}
\right).
\end{eqnarray*}

The associativity theorem for logarithmic intertwining operators gives,
on a common convergence region,
\[
\begin{aligned}
&
\mathcal Y_1
(\mathcal U(q_{z_1})a_1,z_1)
\mathcal Y_2
(\mathcal U(q_{z_2})a_2,z_2)
\\
&\qquad =
\mathcal Y_3
\big(
\mathcal U(q_{z_2})
(
\mathcal Y_4
(
a_1,e^{\log(z_1-z_2)}
)a_2,
q_{z_2}
\big).
\end{aligned}
\]
Applying the linear functional $\tr_{W_0}\left(\,\cdot\,q^{L(0)-c/24}\right)$
to both sides gives
\[
\begin{aligned}
\Phi_{\mathrm{prod}}
&=
\tr_{W_0}
\left(
\mathcal Y_1
(\mathcal U(q_{z_1})a_1,z_1)
\mathcal Y_2
(\mathcal U(q_{z_2})a_2,z_2)
q^{L(0)-c/24}
\right)
\\
&=
\tr_{W_0}
\Big(
\mathcal Y_3
\big(
\mathcal U(q_{z_2})
(
\mathcal Y_4
(
a_1,e^{\log(z_1-z_2)}
)a_2,
q_{z_2}
\big)
q^{L(0)-c/24}
\Big)
\\
&=
\Phi_{\mathrm{iter}} .
\end{aligned}
\]
Therefore, on every nonempty common convergence domain, we have
$$\Phi_{\mathrm{prod}}=\Phi_{\mathrm{iter}}.$$

Since both sides admit analytic continuation by part (ii), the identity
theorem implies that their analytic continuations coincide on every
common branch.

(v) We prove that the genus-one trace functions satisfy a finite-dimensional
regular-singular differential system with respect to the modular
parameter \(q=e^{2\pi i\tau}\).

Let $\Phi(a_1,\ldots,a_n;z_1,\ldots,z_n;\tau)$
denote a genus-one trace function.
Using the Ward identity for genus-one correlation functions involving
intertwining operators, we have
\begin{equation*}
\begin{aligned}
q\frac{\partial}{\partial q}
&
\Phi(a_1,\ldots,a_n;z_1,\ldots,z_n;\tau)
\\
&=
\sum_{i=1}^{n}
\tr_{W_0}
\Big(
\mathcal Y_1(\widetilde a_1,z_1)
\cdots
\mathcal Y_i(L[-2]\widetilde a_i,z_i)
\cdots
\mathcal Y_n(\widetilde a_n,z_n)
q^{L(0)-c/24}
\Big)
\\
&\quad+
\sum_{i=1}^{n}
\sum_{m\geq1}
G_{2m}(\tau)
\tr_{W_0}
\Big(
\mathcal Y_1(\widetilde a_1,z_1)
\cdots
\mathcal Y_i(L[2m-2]\widetilde a_i,z_i)
\cdots
\mathcal Y_n(\widetilde a_n,z_n)
q^{L(0)-c/24}
\Big)
\\
&\quad+
\sum_{\substack{i,j=1\\ i\neq j}}^{n}
\sum_{m\geq1}
\wp_m(z_i-z_j,\tau)
\tr_{W_0}
\Big(
\mathcal Y_1(\widetilde a_1,z_1)
\cdots
\mathcal Y_i(L[m-1]\widetilde a_i,z_i)
\cdots
\mathcal Y_n(\widetilde a_n,z_n)
q^{L(0)-c/24}
\Big).
\end{aligned}
\end{equation*}
where $\widetilde{a_i}=\mathcal {U}(q_{z_i})a_i$ is the geometrically modify of $a_i$, \(L[m]\) denotes the Virasoro operators in geometric
coordinates and \(G_{2m}(\tau)\) are the Eisenstein series, the last term represents the elliptic correction terms arising from
the nontrivial geometry of the torus. This is the genus-one Virasoro
recursion formula, see
\cite{Fiordalisi,Hu2}.

We now use the \(C_2\)-cofiniteness condition. Since $\dim V/C_2(V)<\infty ,$
the quotient contains only finitely many independent representatives.
By the \(C_2\)-cofinite recursion theorem
\cite{Zhu,Hu2},
all traces involving Virasoro descendants can be reduced to a finite
linear span of genus-one trace functions.

Therefore, there exist finitely many generating trace functions $\Phi_1(z,q),\ldots,\Phi_N(z,q)$ such that every trace function appearing on the right-hand side of the Ward identity is a linear combination of them. Put
\[
\boldsymbol{\Phi}(z,q)
=
\begin{pmatrix}
\Phi_1(z,q)\\
\vdots\\
\Phi_N(z,q)
\end{pmatrix}.
\]
Consequently, the Ward identity becomes a finite-dimensional system
\[
q\frac{\partial}{\partial q}
\boldsymbol{\Phi}(z,q)
=
A(z,q)\boldsymbol{\Phi}(z,q),
\]
where $A(z,q)\in \mathbb{M}(N,\mathcal O)$ is a matrix whose entries are generated by the elliptic functions appearing in the genus-one recursion formulas.

The coefficients \(A_{ij}(z,q)\) are obtained from the genus-one
recursion formulas. More precisely, they are finite combinations of
the elliptic functions and Eisenstein series appearing in the
recursion relations. Since these functions admit convergent
\(q\)-expansions near \(q=0\), the matrix $A(z,q)\in M_N(\mathbb C)$
is holomorphic near \(q=0\).

Therefore,
\[
\frac{\partial}{\partial q}
\boldsymbol{\Phi}(z,q)
=
\frac1q
A(z,q)
\boldsymbol{\Phi}(z,q),
\]
where the coefficient matrix has at most a simple pole at \(q=0\).
Hence, $q=0$ is a regular singular point.

We may now apply the Frobenius theory for regular-singular systems \cite{CL}.
Therefore the local solutions have the form
\[
\boldsymbol{\Phi}(z,q)
=
\sum_{\alpha=1}^{N_0}
\sum_{k=0}^{K_\alpha}
q^{h_\alpha}
(\log q)^k
\boldsymbol{f}_{\alpha,k}(z,q),
\]
where $h_\alpha\in\mathbb C$ are the characteristic exponents and $\boldsymbol{f}_{\alpha,k}(z,q)$
are holomorphic vector-valued functions near \(q=0\).

Expanding the holomorphic coefficients, $\boldsymbol{f}_{\alpha,k}(z,q)=\sum_{m\geq0}\boldsymbol{\phi}_{\alpha,k,m}(z)q^m ,$
we obtain
\[
\begin{aligned}
\boldsymbol{\Phi}(z,q)
&=
\sum_{\alpha=1}^{N_0}
\sum_{k=0}^{K_\alpha}
q^{h_\alpha}
(\log q)^k
\sum_{m\geq0}
\boldsymbol{\phi}_{\alpha,k,m}(z)q^m
\\
&=
\sum_{\alpha=1}^{N_0}
\sum_{k=0}^{K_\alpha}
\sum_{m\geq0}
\boldsymbol{\phi}_{\alpha,k,m}(z)
q^{h_\alpha+m}
(\log q)^k .
\end{aligned}
\]
where $\boldsymbol{\phi}_{\alpha,k,m}(z)\in\mathbb C^N.$
No assumption that $h_\alpha\in\mathbb Q$ or $h_\alpha\in\mathbb R$ is required.

Thus we complete the proof.                      $\hfill\Box$


\end{document}